\documentclass[12pt,letterpaper]{amsart}

\usepackage{epsfig,amsmath,amsthm,amssymb,graphicx}
\usepackage[top=1in, bottom=1in, left=1in, right=1in]{geometry}
\usepackage{color}

\usepackage{thmtools}      %for restating main thms without
\usepackage{thm-restate}   % re-typing them

\usepackage{hyperref}

\usepackage{cleveref}

\usepackage{xfrac}
\usepackage[table]{xcolor}

\usepackage{comment}
\usepackage{array, multirow}
\usepackage{makecell}
\usepackage{import}

\usepackage{enumerate}

\usepackage{caption}
\usepackage{subcaption}

\usepackage{faktor}

\usepackage{tikz}
\usepackage{tikz-cd}
\usetikzlibrary{matrix, arrows, cd}
\usepackage{pb-diagram}
\usepackage{cancel}

\definecolor{myOlive}{rgb}{.29,.28,.16}
\definecolor{myRed}{rgb}{.78,0,0}

\newtheorem{proposition}{Proposition}[section]
\newtheorem{theorem}[proposition]{Theorem}
\newtheorem{corollary}[proposition]{Corollary}
\newtheorem{lemma}[proposition]{Lemma}

\newtheorem*{theorem*}{Theorem}
\newtheorem*{proposition*}{Proposition}
\newtheorem*{lemma*}{Lemma}
\newtheorem*{corollary*}{Corollary}

\theoremstyle{definition}
\newtheorem{definition}[proposition]{Definition}

\newtheorem{problem}[proposition]{Problem}

\theoremstyle{remark}
\newtheorem{remark}[proposition]{Remark}
\newtheorem{example}[proposition]{Example}

\newcommand{\N}{\mathbb{N}}
\newcommand{\Z}{\mathbb{Z}}

\renewcommand{\H}{\mathcal{H}}

\newcommand{\lk}{\operatorname{lk}}

\renewcommand{\int}[1]{\operatorname{int}(#1)}

\newcommand{\Inv}[1]{#1^{-1}}

\newcommand{\GCD}{\operatorname{GCD}}

\renewcommand{\L}{\mathcal{L}}
\newcommand{\SL}{\mathcal{SL}}

\newcommand{\Sum}{\displaystyle \sum  }
\newcommand{\Prod}{\displaystyle \prod  }

\newcommand{\pref}[1]{(\ref{#1})}

\newcommand{\wt}{\operatorname{wt}}

\newcommand{\chris}[1]{\textcolor{magenta}{[Chris] #1}}
\newcommand{\anthony}[1]{\textcolor{blue}{[Anthony] #1}}
\newcommand{\cotto}[1]{\textcolor{violet}{[cotto] #1}}
\newcommand{\katherine}[1]{\textcolor{teal}{[Katherine] #1}}

\begin{document}
\title[Crossing changes and link homotopy]{How many crossing changes or Delta-moves does it take to get to a homotopy trivial link?}

\author[A.~Bosman]{Anthony Bosman}
\address{Department of Mathematics, Andrews University}
\email{bosman@andrews.edu}
\urladdr{andrews.edu/cas/math/faculty/bosman-anthony.html}

\author[C.~W.~Davis]{Christopher W.\ Davis}
\address{Department of Mathematics, University of Wisconsin--Eau Claire}
\email{daviscw@uwec.edu}
\urladdr{people.uwec.edu/daviscw}

\author[T.~Martin]{Taylor Martin}
\address{Department of Mathematics and Statistics, Sam Houston State University}
\email{taylor.martin@shsu.edu}
\urladdr{shsu.edu/academics/mathematics-and-statistics/faculty/martin.html}

\author[C.~Otto]{Carolyn Otto}
\address{Department of Mathematics, University of Wisconsin--Eau Claire}
\email{ottoa@uwec.edu}
\urladdr{uwec.edu/profiles/ottoa}

\author[K.~Vance]{Katherine Vance}
\address{Department of Mathematics and Computer Science, Simpson College}
\email{katherine.vance@simpson.edu}
\urladdr{simpson.edu/faculty-staff/katherine-vance/}

\date{\today}

\subjclass[2020]{57K10}

\begin{abstract}
The homotopy trivializing number, $n_h(L)$, and the Delta homotopy trivializing number, $n_\Delta(L)$, are invariants of the link homotopy class of $L$ which count how many crossing changes or Delta moves are needed to reduce that link to a homotopy trivial link. In 2022, Davis, Orson, and Park proved that the homotopy trivializing number of $L$ is bounded above by the sum of the absolute values of the pairwise linking numbers and some quantity $C_n$ which depends only on $n$, the number of components.  In this paper we improve on this result by using the classification of link homotopy due to Habegger-Lin to give a quadratic upper bound on $C_n$.  We employ ideas from extremal graph theory to demonstrate that this bound is close to sharp, by exhibiting links with vanishing pairwise linking numbers and whose homotopy trivializing numbers grows quadratically.   In the process, we determine the homotopy trivializing number of every 4-component link.
We also prove a cubic upper bound on the difference between the Delta homotopy trivializing number of $L$ and the sum of the absolute values of the triple linking numbers of $L$.
\end{abstract}

\maketitle

%\chris{Weird .tex error encountereded.  Commented out all of the reptheorems.  Reptheorem is not currently used in the text.  } \katherine{I think we can use thmtools and thm-restate instead.}

%CHRIS:  SEEING HOW TO RESTATE THEOREM
%
%\begin{restatable}[Goldbach's conjecture]{thm}{goldbach}
%\label{thm:goldbach}
%Every even integer greater than 2 can be expressed as the sum of two primes.
%\end{restatable}

%\section{Second}

%We recall \cref{thm:goldbach}:

%\goldbach*

\section{Introduction and statement of results}
 Any link $L$ in $S^3$ can be reduced to the unlink by some sequence of crossing changes. If this can be done by changing only crossings where a component of $L$ crosses over itself, often called a \emph{self-crossing change}, then we say that $L$ is \emph{homotopy trivial}. If links $L$ and $J$ can be transformed into each other by self-crossing changes then we call $L$ and $J$ link homotopic.  Unlike the question of when two links are isotopic, which is famously difficult, link homotopy is classified by Habegger-Lin \cite{HL1}, building on work of Milnor \cite{M1}. 

The number of crossing changes needed to transform a link to the unlink is called its unlinking number.  This invariant has been the target of intense study; 
 see for example \cite{KM86, Kohn91, Kohn93,Likorish82, Yasutaka81}.  In \cite[Section 6]{DOP22} the second author, along with Park and Orson combine the unlinking number with the notion of link homotopy and introduce the \emph{homotopy trivializing number}, $n_h(L)$, the number of crossing changes needed to reduce $L$ to a homotopy trivial link.   In that paper they show that $n_h(L)$ is controlled by the pairwise linking number of $L$ together with the number of components of $L$.  

\begin{theorem*}[\cite{DOP22}, Theorem 1.7]
For any $n\in \N$ there is some $C_n\in \N$ so that for every  $n$-component link $L$, 
$$\Lambda(L)\le n_h(L)\le \Lambda(L)+C_n,$$
where $\Lambda(L)=\Sum_{i<j}|\lk(L_i, L_j)|$ is the sum of the absolute values of the pairwise linking numbers.
\end{theorem*}
  Such a bound is surprising since linking numbers form only the first of a family of higher order Milnor invariants which classify link homotopy \cite{HL1},\cite{M1}.  This result indicates that these higher order invariants have only a bounded impact on the number of crossing changes needed to get to a homotopy trivial link.  While one could parse out a precise value of the constant $C_n$ produced by the techniques of \cite{DOP22}, actually doing so would require a detailed combinatorial analysis and would result in a very large bound. 
 We pose the following problem, on which we make significant progress. %We seek a bound on the number of additional crossing changes, beyond accounting for the pairwise linking of components, necessary to homotopy trivialize a link.
 \begin{problem}\label{prob: main}
 For any $n\in \N$ compute 
 $$C_n:=\max\{
 n_h(L)-\Lambda(L)
 \mid 
 L\text{ is an }n\text{-component link}\}.$$
 \end{problem}
 
 Our first main result follows a different approach than \cite{DOP22} and finds quadratic upper and lower bounds on $C_n$.  

\begin{theorem}\label{thm: main}
For all $n\ge 3$, 
$$
2\left\lceil\frac{1}{3}n(n-2)\right\rceil\le C_n\le (n-1)(n-2)
.$$
In particular, $C_3=2$ and $C_4=6$. 
\end{theorem}

The upper bound we produce on $C_n$ comes from the following result,

\begin{restatable*}{theorem}{UpperBoundTheorem}
\label{upper bound theorem main}
If $L$ is an $n$-component link and 
\[Q(L) = \#\{(i,j)\mid 2\leq i+1<j\leq n\text{ and } \lk(L_i,L_j)= 0\}\]
then $n_h(L)\le \Lambda(L)+2Q(L)$.  %\chris{improved statement to give better control when linking numbers are non-zero.  Proof still needs to be adapted.}
\end{restatable*}

%\UpperBoundTheorem

%\begin{theorem}[Point to Theorem \ref{upper bound theorem main}]
%If $L$ is an $n$-component link and $Q(L) = \#\{(i,j)\mid 2\leq i+1<j\text{ and } \lk(L_i,L_j)= 0\}$ then $n_h(L)\le \Lambda(L)+2Q(L)$.  
%\end{theorem}

In order to see that this should seem quite surprising, note that when $L$ has vanishing pairwise linking number, this theorem gives a very concrete upper bound on the number of crossing changes needed to reduce $L$ to a homotopy trivial link.

\begin{corollary}
If an $n$-component link has vanishing pairwise linking numbers, then \[n_h(L)\le (n-1)(n-2).%=n^2-3n+2.
\]
\end{corollary}

When enough pairwise linking numbers are non-zero, the invariant $Q(L)$ of Theorem~\ref{upper bound theorem main} vanishes, so that the homotopy trivializing number is determined by the pairwise linking numbers.

\begin{corollary}\label{cor: nonzero linking}
Let $L$ be an $n$-component link. 
 If $\lk(L_i,L_j)\neq 0$ for all $i,j$ with $|i-j|>1$, then $\Lambda(L) = n_h(L)$.  
\end{corollary}

Our strategy to compute the homotopy trivializing number reveals a linear bound on $n_h(L)$ over all Brunnian links.  Recall that a link is called Brunnian if its every proper sublink is trivial.

\begin{restatable*}{corollary}{nhlForBruunian}\label{cor: nhl for brunnian}
%\begin{corollary}[Point to Corollary~\ref{cor: nhl for brunnian}]
If $n\ge 3$ and $L$ is an $n$-component Brunnian link, then $n_h(L)\le 2(n-2).$
\end{restatable*}

In \cite{DOP22} the homotopy trivializing number of any 3-component link $L$ is determined in terms of $\Lambda(L)$ along with Milnor's triple linking number, $\mu_{123}(L)$,$$n_h(L)=\begin{cases}\Lambda(L) &\text{ if }\Lambda(L)\neq 0,\\
2&\text{ if }\Lambda(L)=0\text{ and }\mu_{123}(L)\neq 0,
\\0&\text{otherwise}.\end{cases}$$
Our proof that $C_4=6$ passes through an argument that determines the homotopy trivializing number of every 4-component link.  The precise statement (Theorems~\ref{thm:nhl for linking number zero}, \ref{thm:linking greater one}, and \ref{thm:nonvanishing linking}) is too long to state here. Instead we present some  elements of this classification.  Here $ \overline\mu_I(L)$ is the Milnor number of $L$ associated with multi-index $I$. It is only well defined modulo the greatest common divisor (\(\GCD\)) of those $\overline\mu_J(L)$ with $J$ the result of deleting some terms from $I$.  As is convention, the first nonvanishing Milnor invariant is denoted $\mu_I(L)$ since it is well defined as an integer.  

\begin{theorem}[See Theorems~\ref{thm:nhl for linking number zero}, \ref{thm:linking greater one}, \ref{thm:nonvanishing linking}]\label{thm: 4-component sampler}
Let $L=L_1\cup L_2\cup L_3\cup L_4$ be a 4-component link.
\begin{itemize}
\item $n_h(L)-\Lambda(L)=6$ if and only if $\Lambda(L)=0$, none of $\mu_{123}(L)$, $\mu_{124}(L)$, $\mu_{134}(L)$, and $\mu_{234}(L)$  are equal to zero and none of $\overline\mu_{1234}(L)$, $\overline\mu_{1324}(L)$,  and $\overline\mu_{1234}(L)+\overline\mu_{1324}(L)$ are multiples of $\GCD(\mu_{123}(L), \mu_{124}(L), \mu_{134}(L), \mu_{234}(L))$. 
\item If $|\lk(L_1, L_2)|\ge 2$ and $\lk(L_3, L_4)\neq 0$, then $n_h(L)=\Lambda(L)$
\item If $\lk(L_1,L_2)$, $\lk(L_2, L_3)$ and $\lk(L_3, L_4)$ are all nonzero, then $n_h(L)=\Lambda(L)$
\item If any four linking numbers of $L$ fail to vanish, then $n_h(L)=\Lambda(L)$.
\end{itemize}
\end{theorem}

\begin{figure}
     \centering
         \begin{tikzpicture}
         \node at (0,0){\includegraphics[width=.2\textwidth]{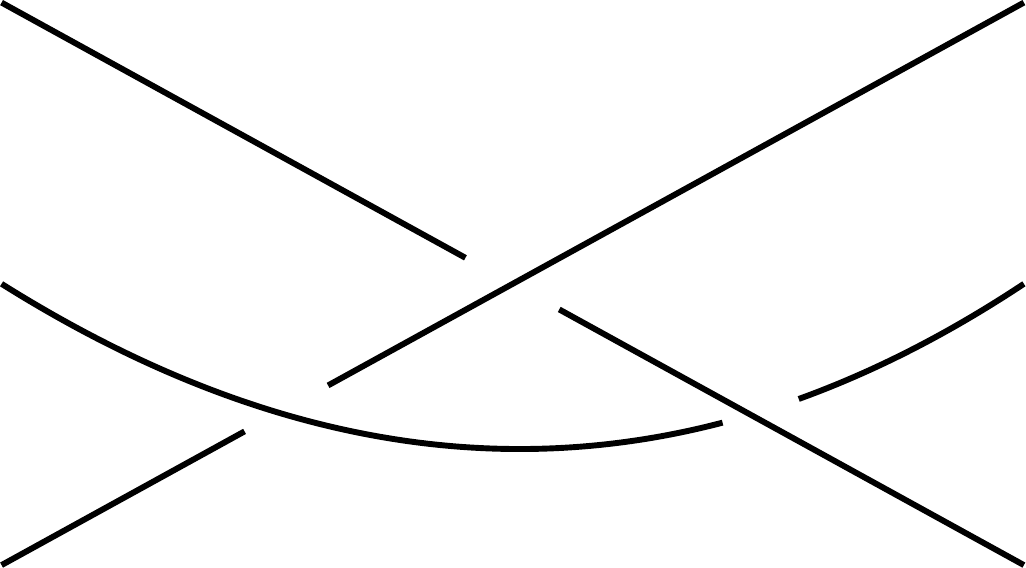}};
         \node at (4,0) {$\longrightarrow$};
         \node at (8,0){\includegraphics[width=.2\textwidth]{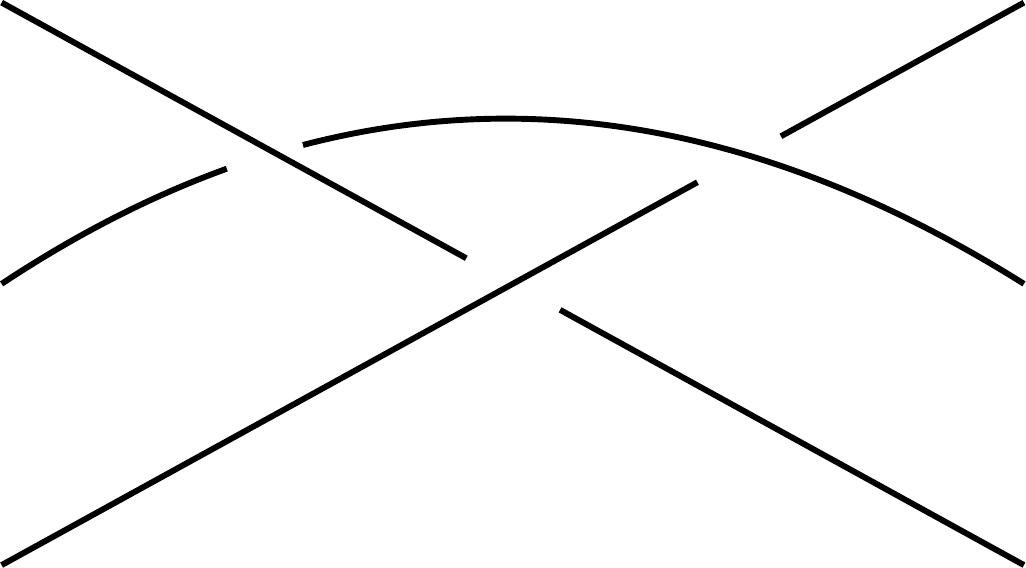}};
         \end{tikzpicture}
        \caption{The $\Delta$-move.}
        \label{fig: Delta Move}
\end{figure}

Another unknotting operation is the Delta-move (henceforth, $\Delta$-move) as pictured in Figure~\ref{fig: Delta Move}. By \cite[Theorem 1.1]{MN89}, any link with vanishing pairwise linking number can be undone by a sequence of $\Delta$-moves. If this $\Delta$-move involves strands of $L_i$, $L_j$ and $L_k$ with $i,j,k$ all distinct, then it changes the triple linking number $\mu_{ijk}$ by precisely 1.  As a consequence, the number of $\Delta$-moves needed to reduce a link with vanishing linking numbers to a homotopy trivial link is bounded below by the sum of the absolute values of the triple linking numbers.  Similarly to Theorem~\ref{thm: main}, we demonstrate an upper bound.  Let $n_{\Delta}(L)$ be the minimal number of $\Delta$-moves needed to transform a link $L$ to a homotopy trivial link and $\Lambda_3(L) = \Sum_{i<j<k} |\mu_{ijk}(L)|$ be the sum of the absolute values of the triple linking numbers of $L$. We show the following:

\begin{restatable*}{theorem}{thmMainDelta}
\label{thm: main Delta}
For any $n$-component link $L$ with vanishing pairwise linking numbers, 
$$\Lambda_3(L)\le n_\Delta(L)\le \Lambda_3(L)+\frac{2}{3}
(n^3-6n^2 + 11 n - 6).
%(n^3 - 3n^2 + 2n - 6).
$$
\end{restatable*}

\begin{comment}***Commented out - replaced with restatable.  
\begin{theorem}\label{thm: main Delta}
For any $n$-component link $L$ with vanishing pairwise linking numbers, 
$$\Lambda_3(L)\le n_\Delta(L)\le \Lambda_3(L)+\frac{2}{3} (n^3 - 3n^2 + 2n - 6).$$
\end{theorem}
\end{comment}

\begin{corollary}\label{cor: vanishing Delta}
For any $n$-component link $L$ if $\Lambda(L)=\Lambda_3(L)=0$, then
$$n_\Delta(L)\le \frac{2}{3} (n^3 - 3n^2 + 2n - 6).$$
\end{corollary}

In order to prove that $C_n\ge 2\left\lceil\frac{1}{3}n(n-2)\right\rceil$ we need to exhibit links $L$ with $n_h(L)\ge 2\left\lceil\frac{1}{3}n(n-2)\right\rceil$.  We do so in Theorem~\ref{thm: large nhl using 4-component sublinks} by studying a link $L$ with vanishing pairwise linking numbers and whose every 4-component sublink has homotopy trivializing number 6. In order to compute the homotopy trivializing number of this link, we study any sequence of crossing changes transforming $L$ to a homotopy trivial link.  We associate to this sequence a weighted graph with vertices $\{v_1,\dots v_n\}$; the edge from $v_i$ to $v_j$ is weighted by half of the number of crossing changes performed between $L_i$ and $L_j$.  Note that by our choice of $L$, the graph spanned by any four vertices of $G$ must have weight at least 3.  We prove the following theorem, which we think will be of independent interest to a graph theorist.

\begin{theorem}\label{thm:min_weight_phi_n}
%Define $\Phi_n$ to be the set of all graphs with non-negative integer weights on their edges which satisfy that for every $G\in \Phi_n$ the total weight of the subgraph spanned by any 4 vertices has total weight at least 3.  Let $\phi_n$ denote the minimum total weight among all graph in $\Phi_n$. For $\phi_n\geq 4$,
Let $G$ be a graph with non-negative integer weights on its edges.  If the total weight of the subgraph of $G$ spanned by any four vertices is at least 3, then the total weight of $G$ is at least $\left\lceil\frac{1}{3}n(n-2)\right\rceil%2\cdot {\left\lceil\frac{n-1}{3}\right\rceil\choose 2}+ {\left\lfloor \frac{2n+1}{3}\right\rfloor\choose 2}.
$.
\end{theorem}

The fact that $2\left\lceil\frac{1}{3}n(n-2)\right\rceil\le C_n$ will follow. Forgetting the link theory context, we pose the following graph theoretic problem motivated by the above result.

\begin{problem}
Fix any integers $n$, $k$, and $w$.  Define $\Phi(n,k,w)$ to be the set of all graphs $G$ with non-negative integer weight on their edges which satisfy that the subgraph spanned by any $k$ vertices of $G$ has total weight at least $w$.  Let $\phi(n,l,w)$ to be the minimal total weight among all $G\in \Phi(n,k,w)$. Determine $\phi(n,k,w)$.   
\end{problem}

When $w=1$, this is essentially determined by a classical theorem of extremal graph theory called Tur\'an's theorem, which determines the graph on $n$-vertices having the maximal number of edges but not containing a $k$-vertex clique.  See \cite{Turan1941} or, for a more modern treatment,  \cite[Theorem 12.2]{GraphsAndDigraphs}.

\subsection{Outline of the paper}

Habiro \cite{Hab1} gives a family of moves called clasper surgery which generalizes both crossing changes and the $\Delta$-move.  In Section~\ref{sect: clasper} we recall this language and use it to verify the intuitive fact that $n_h(L)$ and $n_\Delta(L)$ are invariants of the link homotopy class of $L$.  As a consequence we can take advantage of the classification of link homotopy due to Habegger-Lin in terms of the group $\mathcal{H}(n)$ of string links up to link homotopy.  In Section~\ref{sect: link homotopy classification} we recall elements of this classification and study how $n_{h}$ and $n_{\Delta}$ interact with the structure of this group. The group $\H(n)$ decomposes as a semi-direct product of a sequence of nilpotent groups, called reduced free groups.  By working over this decomposition, in Section~\ref{sect:bounding homotopy trivialing number} we prove half of Theorem~\ref{thm: main}, that $C_n\le (n-1)(n-2)$. In Section~\ref{sect: Delta change} we use a similar logic applied to the $\Delta$-move to prove Theorem~\ref{thm: main Delta}.     When $n=4$,  $\H(4)$ is small enough that we can check the homotopy trivializing numbers of every element of the group.  We do so in Section~\ref{sect: 4 component}, proving much more than is stated as Theorem~\ref{thm: 4-component sampler}.  Finally, in Section~\ref{sect: links with large n_h} we prove the graph theoretic result, Theorem~\ref{thm:min_weight_phi_n}, and use it to complete the proof of Theorem~\ref{thm: main}.  
 \subsection{Acknowledgments}  During the final revisions of this paper, the first, second, third, and fifth authors were supported by National Science Foundation Grant no.~DMS-1928930 while they participated in a program hosted by the Simons-Laufer Mathematical Sciences Institute (Formerly Mathematical Sciences Research Institute) in Berkeley, California during the Summer of 2025.   We would also like to thank an anonymous, thorough, and extremely helpful referee for improving this document considerably.  

\section{Clasper surgery }\label{sect: clasper}

\begin{comment}
\begin{definition}
Let $L$ be a link and let $B$ be a small 3-ball $B$ that intersects $L$ in the tangle of Figure~\ref{fig:BeforeCrossingChange} (or~\ref{fig:AfterCrossingChange} resp.). If $L'$ is given by replacing $L\cap B$ by the tangle of Figure~\ref{fig:AfterCrossingChange} 
(\ref{fig:BeforeCrossingChange} resp.) then we say that $L$ and $L'$ differ by a \emph{crossing change}.   If the two arcs in $B\cap L$ each sit in the same component of $L$, then $L$ and $L'$ differ by a \emph{self-crossing change}, otherwise the differ by a \emph{non-self-crossing change}.
\end{definition}
\end{comment}

In \cite{Hab1}, Habiro introduces the notion of clasper surgery. 
 These moves provide a useful language for crossing changes and $\Delta$-moves.  We use the following definition from \cite[Definition 2.1]{KiMi2023}.  

\begin{definition}
An embedded disk $\tau$ in $S^3$ is called a \emph{simple tree clasper} for a link $L$ if $\tau$ decomposes as a union of bands and disks satisfying the following
\begin{enumerate}
\item Each band connects two distinct disks and each disk is attached to either one or three bands.  A disk attached to only one band is called a \emph{leaf}.
\item $L$ intersects $\tau$ transversely and each point of $L\cap \tau$ is interior to a leaf.
\item Each leaf intersects $L$ transversely in exactly one point.
\end{enumerate}
See Figure~\ref{fig:ClasperDisk} for a generic picture.  A $C_k$-tree is a simple tree clasper with exactly $k+1$ leaves.  Notice that a $C_k$-tree can be reconstructed from its disks together with a single framed arc along each band.  Thus, we will record a clasper as a union of disks and (framed) arcs in between, as in Figure~\ref{fig:ClasperArc}. When no framing is specified, we impose the blackboard framing. 
\end{definition}

\begin{figure}
     \centering
     \begin{subfigure}[b]{0.25\textwidth}
         \centering
         \begin{tikzpicture}
         \node at (0,0){\includegraphics[height=.45\textwidth]{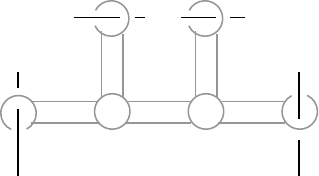}};
         \end{tikzpicture}
         \caption{A simple tree clasper.}
         \label{fig:ClasperDisk}
     \end{subfigure}
     \begin{subfigure}[b]{0.4\textwidth}
     \centering
         \begin{tikzpicture}
         \node at (0,0){\includegraphics[height=.3\textwidth]{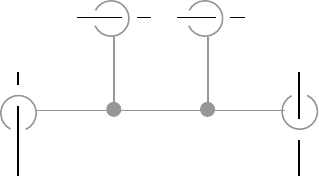}};
         \end{tikzpicture}
         \caption{A clasper as framed arcs and disks.}
         \label{fig:ClasperArc}
     \end{subfigure}
     \begin{subfigure}[b]{0.25\textwidth}
         \centering
         \begin{tikzpicture}
         \node at (0,0){\includegraphics[height=.45\textwidth]{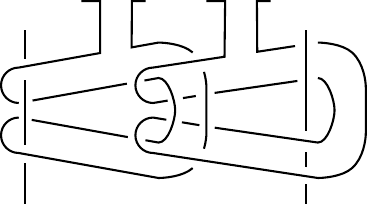}};
         \end{tikzpicture}
         \caption{Clasper surgery.}
         \label{fig:ClasperSurgery}
     \end{subfigure}
     
        \caption{}
        \label{fig: Claspers}
\end{figure}

Given a $C_k$-tree $\tau$ for a link $L$, the result of clasper surgery along $\tau$ is given in Figure~\ref{fig:ClasperSurgery}.  A crossing change can be expressed as clasper surgery along a $C_1$-tree and a $\Delta$-move as clasper surgery along a $C_2$-tree.

\begin{comment}
\begin{figure}
     \centering
     \begin{subfigure}[b]{0.3\textwidth}
         \centering
         \begin{tikzpicture}
         \node at (0,0){\includegraphics[width=.9\textwidth]{}};
         \end{tikzpicture}
         \caption{A clasper realizing a positive crossing change}
         \label{fig: pos cross change clasper}
     \end{subfigure}
     \begin{subfigure}[b]{0.3\textwidth}
         \centering
         \begin{tikzpicture}
         \node at (0,0){\includegraphics[width=.9\textwidth]{}};
         \end{tikzpicture}
         \caption{A clasper realizing a negative crossing change}
         \label{fig: pos cross change clasper}
     \end{subfigure}
     \begin{subfigure}[b]{0.2\textwidth}
     \centering
         \begin{tikzpicture}
         \node at (0,0){\includegraphics[width=.5\textwidth]{HopfStringLink.pdf}};
         \end{tikzpicture}
         \caption{A clasper realizing a $\Delta$-move}
         \label{fig:HopfStringLink}
     \end{subfigure}
         \label{fig:low order clasper}
     
        \caption{Claspers}
        \label{fig: low Clasper surgery}
\end{figure}
\end{comment}

 We can use the language of claspers to define the \emph{homotopy trivializing number}.  A link $L$ can be reduced to a homotopy trivial link in $k$ crossing changes if there is a collection of $k$ disjoint $C_1$-tree for $L$ so that the result of surgery along these claspers is homotopy trivial.  Then, $n_h(L)$ is the minimal such value of $k$.  Similarly, $n_\Delta(L)$ is defined using $C_2$-trees.   The $\Delta$-move is done by a surgery along a single $C_2$-tree and conversely surgery along a $C_2$-tree surgery can be done by a $\Delta$-move.  Figure~\ref{fig: Delta by Clasper surgery} reveals how to perform the $\Delta$-move via surgery along a $C_2$-tree, and Figure~\ref{fig: Clasper surgery by Delta} shows how surgery along a $C_2$-tree surgery can be undone by a $\Delta$-move.  See also \cite[Section 7.1]{Hab1}. 
 Thus, we define $n_{\Delta}(L)$ to be the minimal number of surgeries along $C_2$-trees  needed to transform $L$ to a homotopy trivial link. It follows that $n_h(L)$ and $n_\Delta(L)$ are invariant under link homotopy.

\begin{figure}
     \centering
         \begin{subfigure}[b]{0.3\textwidth}
         \centering
         \begin{tikzpicture}
         \node at (0,0){\includegraphics[width=.9\textwidth]{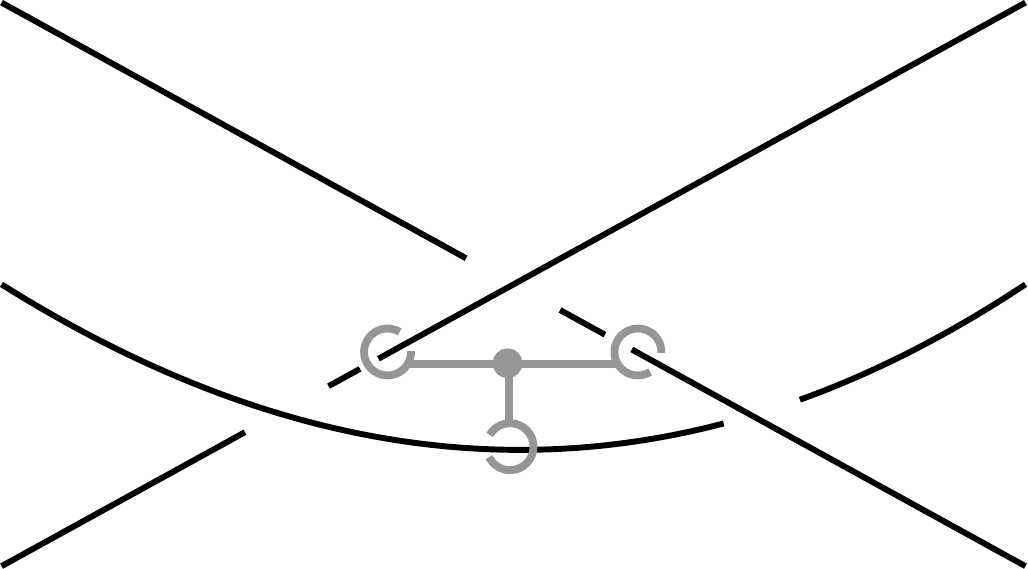}};
         \end{tikzpicture}
         \caption{}
         \label{fig: Clasper for Delta}
     \end{subfigure}
         \begin{subfigure}[b]{0.3\textwidth}
         \centering
         \begin{tikzpicture}
         \node at (0,0) {\includegraphics[width=.9\textwidth]{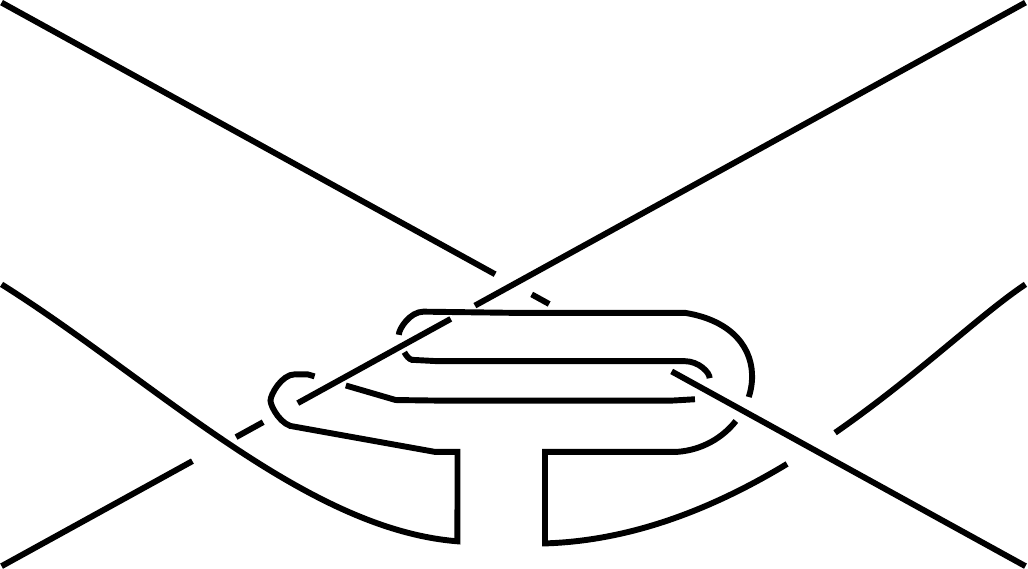}};
         \end{tikzpicture}
         \caption{}
         \label{fig: Clasper for Delta surgery}
     \end{subfigure}
     \begin{subfigure}[b]{0.3\textwidth}
         \centering
         \begin{tikzpicture}
         \node at (0,0) {\includegraphics[width=.9\textwidth]{DeltaMoveAfter}};
         \end{tikzpicture}
         \caption{}
         \label{fig: Clasper for Delta done}
     \end{subfigure}
        \caption{Left to right: \pref{fig: Clasper for Delta} A $C_2$-tree which realizes the $\Delta$-move.  \pref{fig: Clasper for Delta surgery} Performing clasper surgery.  \pref{fig: Clasper for Delta done} After an isotopy we get the result of the $\Delta$-move. }
        \label{fig: Delta by Clasper surgery}
\end{figure}

\begin{figure}
     \centering
         \begin{subfigure}[b]{0.22\textwidth}
         \centering
         \begin{tikzpicture}
         \node at (0,0){\includegraphics[width=.9\textwidth]{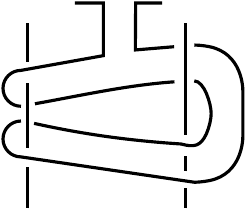}};
         \end{tikzpicture}
         \caption{}
         \label{fig: C2Surgery}
     \end{subfigure}
         \begin{subfigure}[b]{0.22\textwidth}
         \centering
         \begin{tikzpicture}
         \node at (0,0) {\includegraphics[width=.9\textwidth]{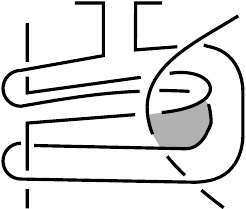}};
         \end{tikzpicture}
         \caption{}
         \label{fig: C2Surgery Delta Locus}
     \end{subfigure}
     \begin{subfigure}[b]{0.22\textwidth}
         \centering
         \begin{tikzpicture}
         \node at (0,0) {\includegraphics[width=.9\textwidth]{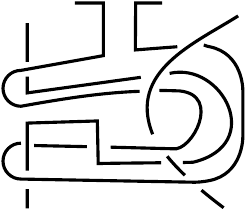}};
         \end{tikzpicture}
         \caption{}
         \label{fig: C2Surgery Delta Done}
     \end{subfigure}
     \begin{subfigure}[b]{0.22\textwidth}
         \centering
         \begin{tikzpicture}
         \node at (0,0) {\includegraphics[width=.9\textwidth]{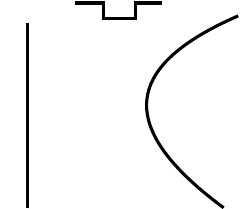}};
         \end{tikzpicture}
         \caption{}
         \label{fig: C2Surgery Delta undone}
     \end{subfigure}
        \caption{Left to right: \pref{fig: C2Surgery} The result of $C_2$-clasper surgery.  \pref{fig: C2Surgery Delta Locus} After an isotopy we see a place to perform a $\Delta$-move.  \pref{fig: C2Surgery Delta Done} Performing the $\Delta$-move.  \pref{fig: C2Surgery Delta undone} An isotopy reduces this to the trivial tangle.  }
        \label{fig: Clasper surgery by Delta}
\end{figure}

\begin{theorem}\label{thm: homotopy invariance}
If $L$ and $J$ are link homotopic, then $n_h(L) = n_h(J)$ and $n_\Delta(L) = n_\Delta (J)$.
\end{theorem}
\begin{proof}
If $L$ and $J$ are link homotopic, there is a collection of $C_1$-trees $\tau$ for $J$, each of which intersects only one component of $J$, so that changing $J$ by surgery along $\tau$ results in $L$.  As a positive crossing change can be undone by a negative crossing change, there is a collection of $C_1$-trees $\overline{\tau}$ for $L$ so that surgery along $\overline{\tau}$ results in $J$.   

Suppose that $n_h(L)=k$.  Then there is a collection of $k$ many $C_1$-trees, $\tau'$, for $L$ so that performing surgery along $\tau'$ changes $L$ to a homotopy trivial link, $L'$.  We may now isotope ${\tau}'$ so that it is disjoint from $\overline{\tau}$. By performing surgery along $\overline{\tau}$, we may now think of $\tau'$ as a sequence of crossing changes for $J$.  For the sake of clarity call this new collection $\tau'_J$, and the link resulting from surgery $J'$. 

Summarizing, we now have a collection of $C_1$-trees $\tau\cup\tau'_J$ for $J$ so that surgery along this collection results in the homotopy trivial link $L'$.  Since the order in which we perform surgery does not affect the result, we may first perform surgery along $\tau'_J$ to get a new link $J'$ and then change $J'$ by surgery along $\tau$.  As each component of $\tau$ intersects only one component of $J'$ it follows that $J'$ is link homotopic to $L'$, and so is itself homotopy trivial.  

We have now produced a collection of $k$ many $C_1$-trees $\tau'_J$ for $J$ so that surgery along $\tau'_J$ results in a homotpy trivial link.  Thus, $n_h(J)\le k=n_h(L)$.  The reverse inequality follows the same argument, as does the proof that $n_\Delta(L)=n_{\Delta}(J)$.  
\end{proof}

\section{String links and Habegger-Lin's classification of link homotopy}\label{sect: link homotopy classification}

Let $\mathcal{LH}_n$ be the set of $n$-component links up to link homotopy.  By Theorem~\ref{thm: homotopy invariance}, $n_h(L)$ depends only on the equivalence class of $L$ in $\mathcal{LH}_n$.  See also \cite[Remark 6.3]{DOP22}.  As a consequence, we can appeal to the classification of links up to link homotopy due to Habegger-Lin \cite{HL1} as well as an earlier work of Goldsmith \cite{Goldsmith73} in order to organize our argument.  In this section we recall some elements of this classification and explain the strategy we will follow.    

\begin{definition}
Let $p_1,\dots, p_n$ be distinct points interior to the unit disk $D^2$.  An $n$-component \emph{string link} $T$ is a collection of disjoint embedded  arcs $T_1\cup\dots\cup T_n$ in $D^2\times[0,1]$ with $T_i$ running from $p_i\times\{0\}$ to $p_i\times\{1\}$.  Two string links are called link homotopic if one can be transformed to the other by a sequence of self-crossing changes. The set of $n$-component string links up to ambient isotopy rel. boundary is denoted by $\SL_n$, and $\H(n)$ is the set of $n$-component string links up to link homotopy. A string link is homotopy trivial if it is link homotopic to the trivial string link.
\end{definition}

The notions of clasper surgery, crossing change, and $\Delta$-moves all extend to string links, and so the definitions of $n_h$ and $n_{\Delta}$ extend in the obvious way to string links where they depend only on the class of a link in $\H(n)$.

%\begin{definition}
%For a string link $T$, $n_h(T)$ is the minimal number of crossing changes needed to transform $T$ to a homotopy trivial string link.  $n_{\Delta}(L)$ is the minimal number of $\Delta$-moves.  
%\end{definition}

 \begin{figure}[h]
\subcaptionbox{$A*B$: The result of stacking string links $A$ and $B$.\label{fig: stacking}}[0.3\linewidth]{
 \begin{tikzpicture}[scale=1.2]

\draw[thick, black] (0.2,-0.2) -- (0.2, 2.2);
\draw[thick, black] (0.5,-0.2) -- (0.5, 2.2);
\draw[thick, black] (0.8,-0.2) -- (0.8, 2.2);

\draw[thick, black] (1.8,-0.2) -- (1.8, 2.2);
\draw[thick, black] (2.1,-0.2) -- (2.1, 2.2);
\draw[thick, black, fill=white] (0,0) rectangle (2.3,0.8);

\draw[thick, black, fill=white] (0,1.2) rectangle (2.3,2);

\node[inner sep=0pt] at (1.17,0.4) {$B$};
\node[inner sep=0pt] at (1.17,1.63) {$A$};
\node[inner sep=0pt] at (1.34,-0.13) {\scalebox{1}{$\dots$}};
\node[inner sep=0pt] at (1.34,1) {\scalebox{1}{$\dots$}};
\node[inner sep=0pt] at (1.34,2.13) {\scalebox{1}{$\dots$}};

%fudgenode
\node[above] at (1, -0.3) {$\,$};

\end{tikzpicture}
}
\hfill
\subcaptionbox{{The closure $\widehat T$ of a string link $T$ together with a $d$-base, $D$.\label{fig: closure}}}[0.3\linewidth]{
\begin{tikzpicture}[scale=0.6]

%Draw the d-base
\draw (-.2,-.5)--(2.2,-.5) arc(-90:90:.2)--(-.2,-.1)arc(90:270:.2);
\draw[fill=gray, opacity=.5] (-.2,-.5)--(2.2,-.5) arc(-90:90:.2)--(-.2,-.1)arc(90:270:.2);
\node[left] at(-.2,-.3) {$D$};
%Points of intersection With d-base
\draw[fill=black](.2,-.3) circle (.04);
\draw[fill=black](.5,-.3) circle (.04);
\draw[fill=black](.8,-.3) circle (.04);
\draw[fill=black](1.8,-.3) circle (.04);
\draw[fill=black](2.1,-.3) circle (.04);

%left side
\draw[thick, black] (0.2,-0.3) -- (0.2, 1);
\draw[thick, black] (0.2,-0.5) -- (0.2, -.7);
\draw[thick, black] (0.5,-0.3) -- (0.5, 1);
\draw[thick, black] (0.5,-0.5) -- (0.5, -.7);
\draw[thick, black] (0.8,-0.3) -- (0.8, 1);
\draw[thick, black] (0.8,-0.5) -- (0.8, -.7);

\draw[thick, black] (1.8,-0.3) -- (1.8, 1);
\draw[thick, black] (1.8,-0.5) -- (1.8, -.7);
\draw[thick, black] (2.1,-0.3) -- (2.1, 1);
\draw[thick, black] (2.1,-0.7) -- (2.1, -.5);
\draw[thick, black, fill=white] (0,0) rectangle (2.3,0.8);

%\draw[thick, black, fill=white] (0,1.2) rectangle (2.3,2);

\node[inner sep=0pt] at (1.2,0.4) {$T$};
\node[inner sep=0pt] at (2.4,-1.5) {\scalebox{1}{$\vdots$}};
\node[inner sep=0pt] at (2.4, 2.2) {\scalebox{1}{$\vdots$}};
%\node[inner sep=0pt] at (1.34,2.13) {\scalebox{1}{$\dots$}};

%top arc
\draw[thick, black] (2.1, 1) arc (-180:-360:.2);
\draw[thick, black] (1.8, 1) arc (-180:-360:{.2+.3});
\draw[thick, black] (.8, 1) arc (-180:-360:{.2+.3+1});
\draw[thick, black] (.5, 1) arc (-180:-360:{.2+.3+1+.3});
\draw[thick, black] (.2, 1) arc (-180:-360:{.2+.3+1+.3+.3});

%Bottom arcs
\draw[thick, black] (2.1, -0.7) arc (180:360:.2);
\draw[thick, black] (1.8, -0.7) arc (180:360:{.2+.3});
\draw[thick, black] (.8, -0.7) arc (180:360:{.2+.3+1});
\draw[thick, black] (.5, -0.7) arc (180:360:{.2+.3+1+.3});
\draw[thick, black] (.2, -0.7) arc (180:360:{.2+.3+1+.3+.3});

%Right side lines
\draw[thick, black] ({2*2.3-0.2},-0.7) -- ({2*2.3-0.2}, 1);
\draw[thick, black] ({2*2.3-0.5},-0.7) -- ({2*2.3-0.5}, 1);
\draw[thick, black] ({2*2.3-0.8},-0.7) -- ({2*2.3-0.8}, 1);

\draw[thick, black] ({2*2.3-1.8},-0.7) -- ({2*2.3-1.8}, 1);
\draw[thick, black] ({2*2.3-2.1},-0.7) -- ({2*2.3-2.1}, 1);

\end{tikzpicture}

}
\hfill
\subcaptionbox{{$\phi\colon RF(n-1)\to \mathcal{H}(n)$ sends $x_i$ to the string link $x_{in}$ above.  \label{fig: psi(x_i)}}}[0.3\linewidth]{
\begin{tikzpicture}[scale=1]

\node[above] at (0.2, 2) {$T_1$};
\draw[thick, black] (0.2,-0.2) -- (0.2, 2);

\node[inner sep=0pt] at (.7,2.3) {\scalebox{1}{$\dots$}};

\node[above] at (1.2, 2) {$T_i$};
\draw[thick, black] (1.2, 2)--(1.2,1.3);
\draw[thick, black] (1.2, 1.1)--(1.2,-0.2);
\node[inner sep=0pt] at (1.7,2.3) {\scalebox{1}{$\dots$}};

\node[above] at (2.4, 2) {$T_{n-1}$};
\draw[thick, black] (2.4, 2)--(2.4,1.3);
\draw[thick, black] (2.4, 1.1)--(2.4,.5);
\draw[thick, black](2.4,.3)--(2.4,-0.2);

\node[above] at (3.2, 2) {$T_{n}$};
\draw[thick, black] (3.2, 2)--(3.2,1.6) arc (0:-90:.4)--(1.1,1.2) arc(90:270:.4);
\draw[thick, black] (1.3,.4)--(2.8,.4) arc(90:0:.4)--(3.2,-.2);

%fudgenode
\node[above] at (1, -0.4) {$\,$};

\end{tikzpicture}
}
\caption{}
% (Figure sourced from \cite{DOP22}
\label{fig: stacking and closure}
\end{figure} 

 Since any link is the closure of some string link, the maps $\SL_n\to \L_n$ and  $\mathcal{H}(n) \to \mathcal{LH}_n$ sending a string link $T$ to its closure $\widehat{T}$ are surjective. See Figure~\ref{fig: closure}.  The disk $D$ also appearing in Figure~\ref{fig: closure} is called a \emph{$d$-base} for $L$.   It is clear that $n_h(\widehat{T})\le n_h(T)$; indeed, a sequence of crossing changes reducing $T$ to a homotopy trivial string link immediately gives rise to a sequence of crossing changes reducing $\widehat{T}$ to the trivial link.  More surprisingly the reverse inequality holds, so that nothing is lost by studying the homotopy trivializing number over string links instead of links.

 \begin{proposition}
For any $T\in \H(n)$, $n_h(T)=n_h(\widehat{T})$ and $n_{\Delta}(T) = n_{\Delta}(\widehat{T})$.
 \end{proposition}
 \begin{proof}
    Let $T$ be a string link, $L=\widehat T$, and $D$ be the associated $d$-base.  If $n_h(L)=k$ then there exists a collection of $k$ disjoint $C_1$-trees, $\tau$, for $L$ so that surgery along $\tau$ transforms $L$ to a homotopy trivial link $L'$.
    
    First isotope $\tau$ so that its every leaf is disjoint from $D$.   As in Figure~\ref{fig: Crossing change disk avoids dbase} we may now perform a further isotopy to arrange that all of $\tau$ is disjoint from $D$.  As a consequence we can view $\tau$ as collection of $C_1$-trees for $T$ in $S^3\setminus \nu(D)\cong D^2\times[0,1]$, where $\nu(D)$ is a regular neighborhood of $D$.  After changing $T$ by surgery along $\tau$ one arrives at a new string link $T'$ which satisfies that $\widehat{T'}=L'$ is homotopy trivial.  According to \cite[Corollary 2.7]{HL1}, then $T'$ is homotopy trivial.  As a consequence $n_h(T)\le k= n_h(L)$. 
\begin{figure}
     \centering
     \hfill
     \begin{subfigure}[b]{0.4\textwidth}
         \centering
         \begin{tikzpicture}
         \node at (0,0){\includegraphics[width=.8\textwidth]{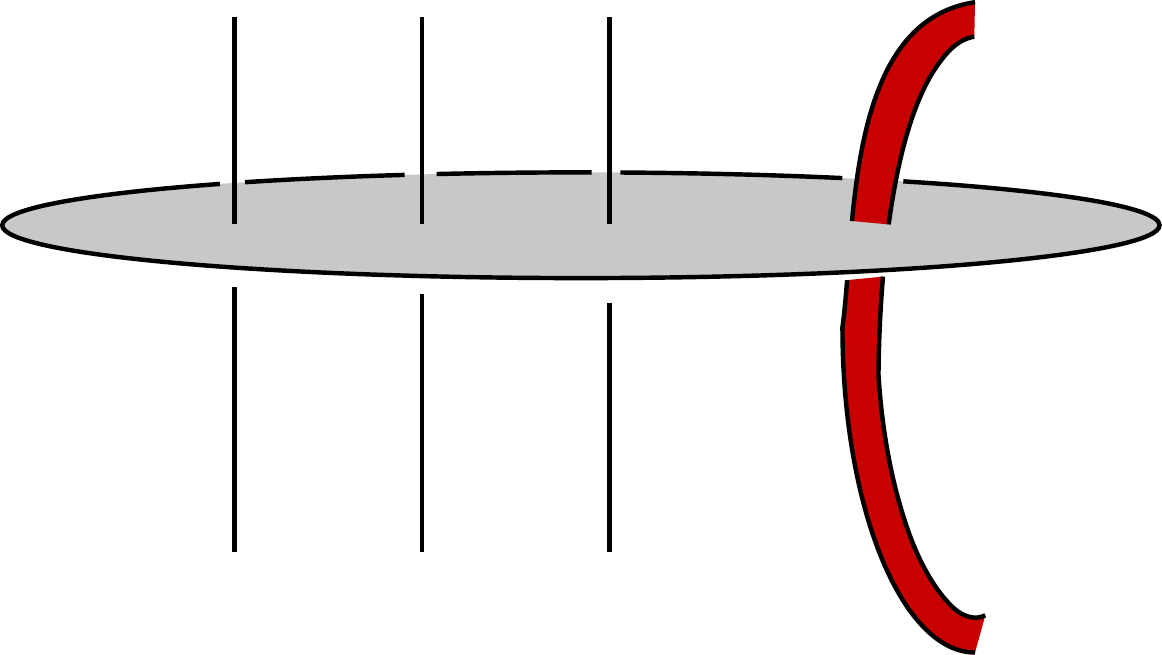}};
         \end{tikzpicture}
         \caption{A $C_k$-tree for $\widehat T$ intersecting a $d$-base in an arc.}
         \label{fig:dbaseIntersectsCrossingDisk}
     \end{subfigure}
     \hfill
     \begin{subfigure}[b]{0.4\textwidth}
         \centering
         \begin{tikzpicture}
         \node at (0,0){\includegraphics[width=.8\textwidth]{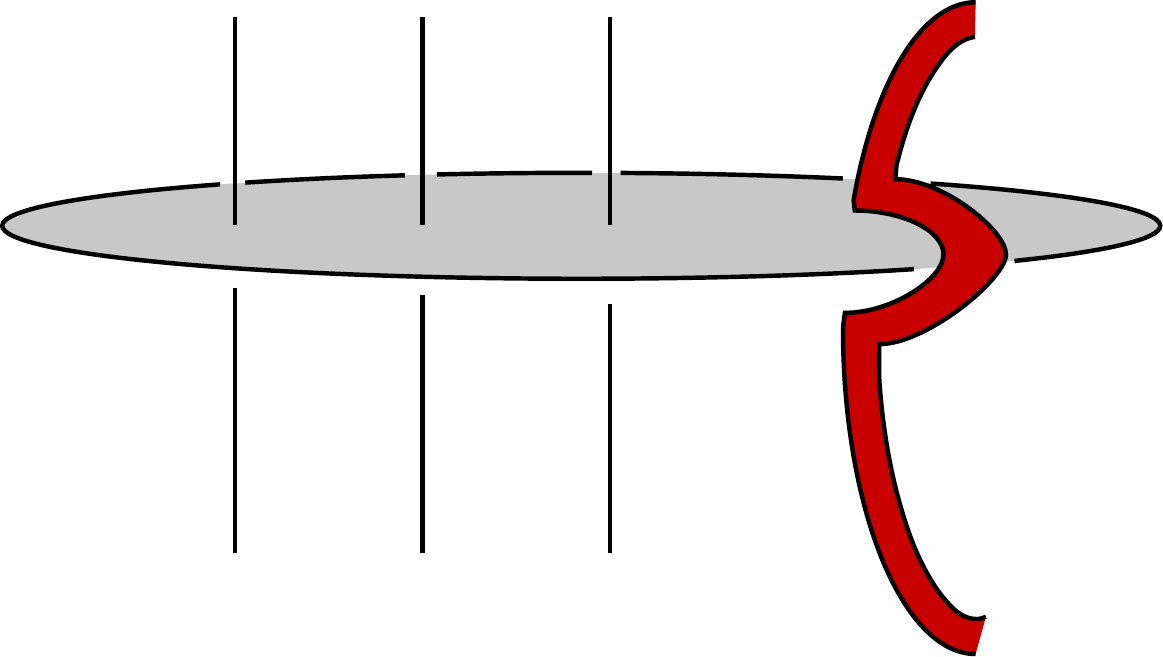}};
         \end{tikzpicture}
         \caption{After an isotopy we remove this intersection.}
         \label{fig:dbaseMissesCrossingDisk}
     \end{subfigure}
     \hfill
     
        \caption{}
        \label{fig: Crossing change disk avoids dbase}
\end{figure}
 \end{proof}

   The advantage of working with string links rather than links up to link homotopy is that string links form a group under the stacking operation of Figure~\ref{fig: stacking}.  The inverse operation $\overline{T}$ is given by first reflecting $T$ over $D^2\times\{1/2\}$ and then reversing the orientations.    A key step in Habegger-Lin's classification of links up to link homotopy \cite{HL1} is the following split short exact sequence.
\begin{equation}\label{exact sequence}
\begin{tikzcd}
    0 \arrow{r} & RF(n-1)\arrow{r}{\phi}  & \mathcal{H}(n)\arrow{r}{p}  & \arrow[dashed, bend left=33]{l}{s}\mathcal{H}(n-1)\arrow{r} & 0.
\end{tikzcd}
\end{equation}

Recall that $RF(n-1)$ is the \emph{reduced free group}, that is it is the quotient of the free group $F(n-1) = F(x_1,\dots, x_{n-1})$ given by killing the commutator of each $x_i$ with any conjugate of itself. Thus, in $RF(n-1)$, $x_i$ commutes with $\gamma x_i \gamma^{-1}$ for each $i$ and any $\gamma\in RF(n-1)$. The map $\phi$ is given by sending  the generator $x_i$ to the string link $x_{i,n}$ of Figure~\ref{fig: psi(x_i)}.  When it will not result in confusion, we drop the comma and write $x_{in}$.  The map $p:\H(n)\to \H(n-1)$ is given by deleting the $n$'th component of a string link, and the splitting $s:\H(n-1)\to \H(n)$ is given by introducing a new unknotted component unlinked from the rest.  %Every string link is link homotopic to a pure braid, we will employ this result at several points in our analysis below.

Recall that for any group $G$ and any $g,h\in G$ the commutator of $g$ with $h$ is defined by $[g,h] = {g}^{-1}{h}^{-1} gh$, (so that $gh = hg[g,h]$).   The following results will turn out to be central to the proof of Theorem~\ref{thm: main}.

    \begin{proposition}\label{prop: basics of nhl}
    Let $n\in \N$, $i\neq j\in \{1,\dots, n\}$, $r\in RF(n-1)$, and $T,S\in \H(n)$. Then:
\begin{enumerate}
 \item\label{item:nh(x_i)} %For any $i\neq j$, 
 $n_h(x_{ij})=1$ and  $n_{\Delta}(x_{ij})=\infty$.
\item \label{item:nh(product)} 
%For any $T,S\in \H(n)$, 
 $n_h(T*S) \le n_h(T)+n_h(S)$, and $n_\Delta(T*S) \le n_\Delta(T)+n_\Delta(S)$.
 \item \label{item:nh(commutator)}$n_h([T,S])\le 2\cdot \min(n_h(T), n_h(S))$ and $n_\Delta([T,S])\le 2\cdot \min(n_\Delta(T), n_\Delta(S))$.
 \item \label{item:nh([x,w])} 
 %For any $T\in \H(n)$, 
 $n_h([T,x_{ij}])\le 2$.
 \item \label{item:ndelta(commutator)}
 %For any $r\in RF(n-1)$ and any of the preferred generators $x_i$ and $x_j$ of $RF(n-1)$, 
 $n_{\Delta}(\phi([rx_i r^{-1}, x_j]))\le 1$.
 \end{enumerate}
\end{proposition}
\begin{proof}

To see the first conclusion, observe that $x_{ij}$ is transformed to the trivial string link by changing a single crossing.  Thus, $n_h(x_{ij})\le1$.  Since linking number is a link homotopy obstruction, and $x_{ij}$ is not homotopy trivial, it follows that $n_h(x_{ij})=1$.  The $\Delta$-move preserves linking number, so $x_{ij}$ cannot be unlinked by $\Delta$-moves. Thus, $n_\Delta(x_{ij})=\infty$.

Next, suppose $n_h(T)=k$ and $n_h(S)=\ell$.  Then $T*S$ can be transformed to $I*S=S$ by $k$ crossing changes.  Here $I$ is (link homotopic to) the $n$-component trivial string link.  An additional $\ell$ crossing changes transforms this to the trivial element of $\H(n)$.  The same argument holds for $n_\Delta$.

To see the third result, notice that by changing $k$ crossings, $[T,S] = {T}^{-1}{S}^{-1}TS$ is transformed to ${T}^{-1}{S}^{-1}S = T^{-1}$. 
 Another $k$ crossing changes transforms it to a homotopy trivial string link.  Thus, $n_h([T,S])\le 2n_h(T)$.  By a similar analysis, $n_h([T,S])\le 2n_h(S)$.  The same argument holds for $n_\Delta$.
 
The fourth result is an immediate corollary of the first and third.  

Finally, let $r\in RF(n-1)$, and $S=\phi([rx_i r^{-1}, x_j])$.  In Figure~\ref{fig: 3-clasper for Delta} we see a  $C_2$-tree $c$ on the trivial string link.  In Figure~\ref{fig: 3-clasper surgery for delta} we see the result of clasper surgery, call it $T$.   As the leftmost $n-1$ components of each of $T$ and $S$ are unlinked, $S,T\in \H(n)$ depend only on the class of their $n$'th component in the fundamental group of the complement of the first $n-1$ components, which is the free group on the meridians $m_1,\dots, m_{n-1}$. Using the Wirtinger presentation we write the homotopy classes of  $T_n$ and $S_n$ as words in these meridians.  In each case we get that $[T_n]=[S_n] = [m_i,\psi(r)m_j\psi(r)^{-1}]$. Here $\psi$ is the map given by replacing each $x_k$ by the corresponding meridian $m_k$.   Claim~(\ref{item:ndelta(commutator)}) follows.

 \begin{figure}[h]
\subcaptionbox{
A $C_2$-tree on the trivial string link $T$.% ties into  $T$ along $\phi(r)$.
\label{fig: 3-clasper for Delta}}[0.4\linewidth]{
\begin{tikzpicture}
\node at (0,0) {
\includegraphics[height=.2\textheight]{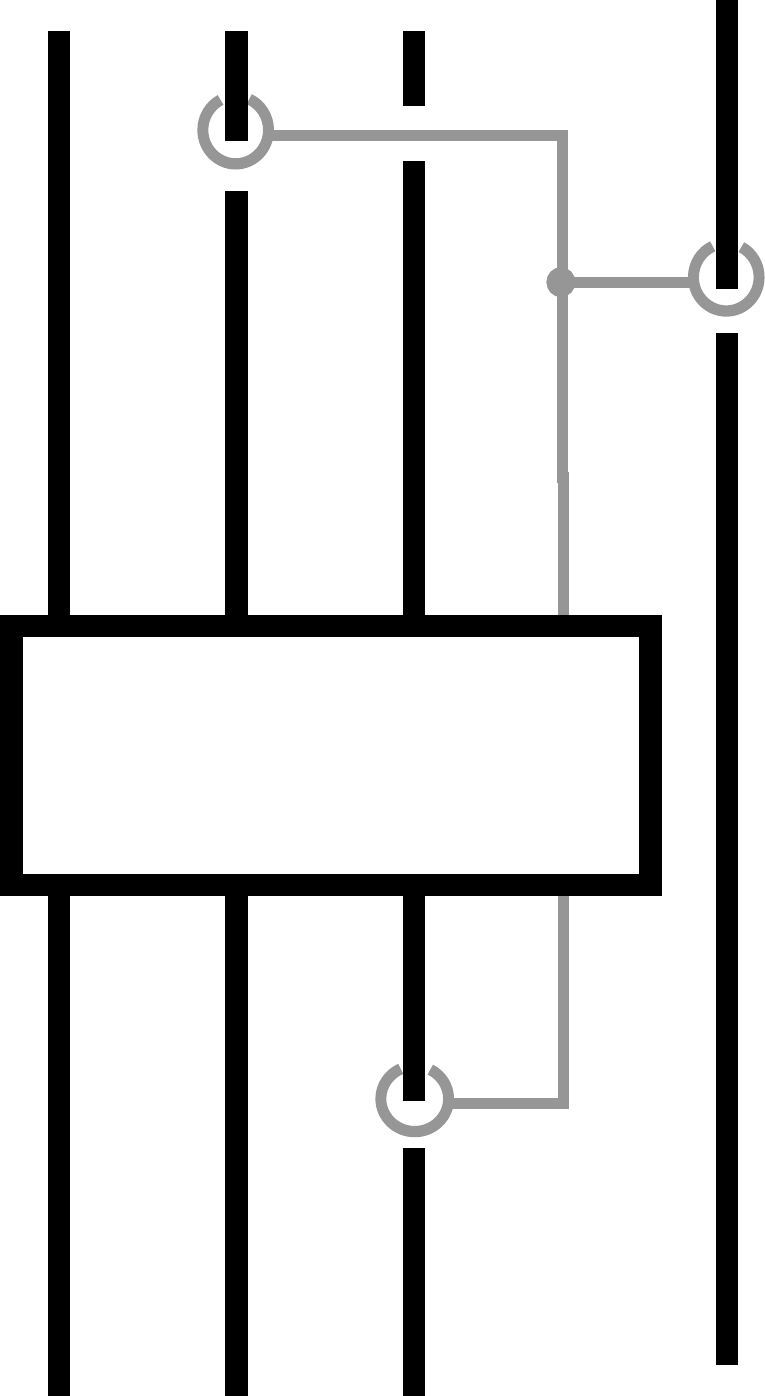}};
\node at (0.3,2.5) {$T_j$};
\node at (-0.4,2.5) {$T_i$};
\node at (1.3,2.5) {$T_m$};
\node at (-.2,-.2) {$\phi(r)$};
\end{tikzpicture}
}
\hspace{.1\textwidth}
\subcaptionbox{Performing clasper surgery.
\label{fig: 3-clasper surgery for delta} 
}[0.35\linewidth]{
\begin{tikzpicture}
\node at (0,0) {\includegraphics[height=.2\textheight]{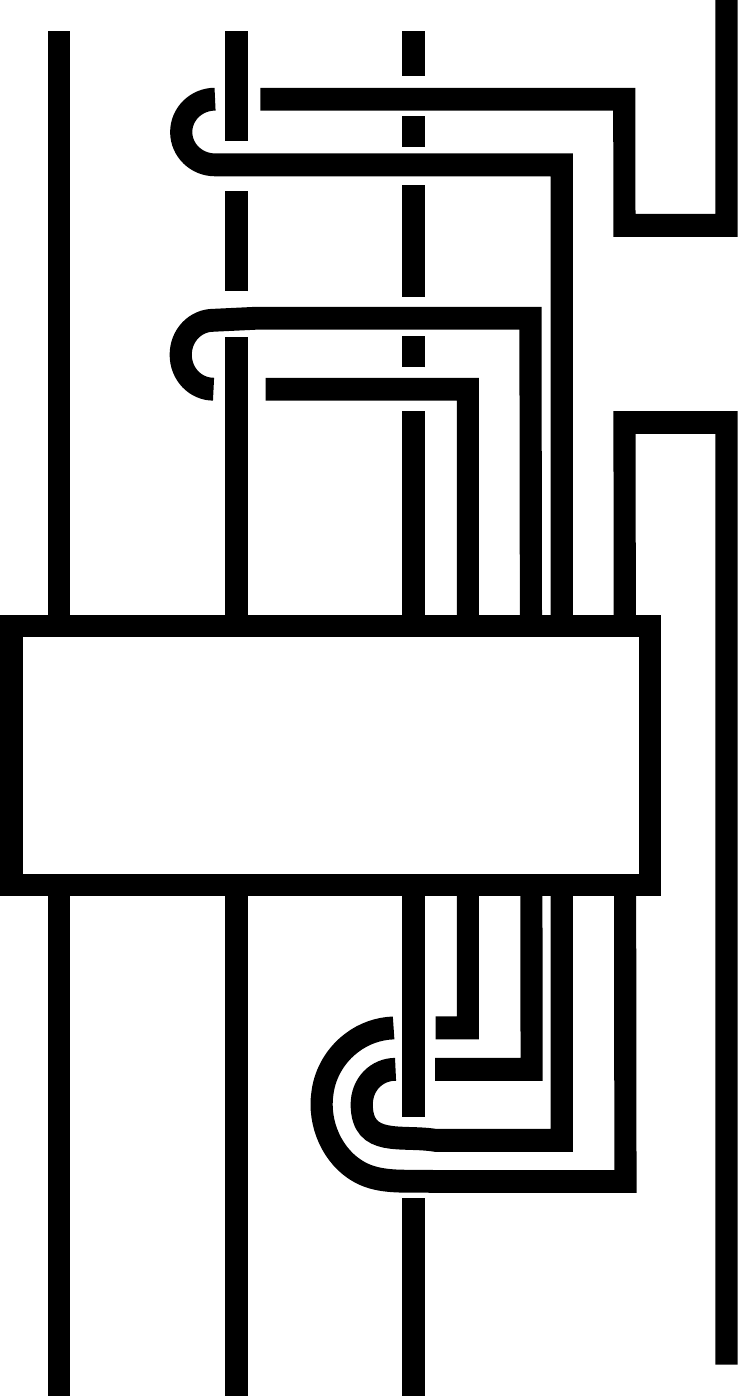}};
\node at (-.2,.-.22) {$\phi(r)$};
\node at (0.3,2.5) {$T_j$};
\node at (-0.4,2.5) {$T_i$};
\node at (1.3,2.5) {$T_m$};
\end{tikzpicture}

}

%\subcaptionbox{{
%$\phi\colon RF(n-1)\to \mathcal{H}_n$ sends $x_i$ to the string link $T$ above.
%\label{fig: phi([rxrbar,y]) HomoIso}}}

\begin{comment}[0.2\linewidth]{
\begin{tikzpicture}
\node at (0,0) {\includegraphics[scale=.15]{rbarxr_y_3}};
\node at (-0.25,1.8) {$r$};
\node at (-.25,.-.25) {$\overline r$};
\end{tikzpicture}
}
\hfill
\subcaptionbox{{
%$\phi\colon RF(n-1)\to \mathcal{H}_n$ sends $x_i$ to the string link $T$ above.
\label{fig: phi([rxrbar,y]) clasp pass}}}[0.2\linewidth]{
\begin{tikzpicture}
\node at (0,0) {\includegraphics[scale=.15]{rbarxr_y_Delta}};
\node at (-0.25,1.8) {$r$};
\node at (-.25,.-.25) {$\overline r$};
\end{tikzpicture}
}
\end{comment}
\caption{
}

\label{fig: unknotting phi([rxrbar,y])}
\end{figure} 
\end{proof}

\section{Bounding the homotopy trivializing number}\label{sect:bounding homotopy trivialing number}

In this section, we  prove Theorem~\ref{upper bound theorem main} which we use in the introduction to conclude that $C_n\le (n-1)(n-2)$.  The bulk of our work will be in proving the following theorem which allows us to realize elements of $RF(m)$ as a product of a minimal number of powers of the preferred generators along with a short list of commutators.

\begin{theorem}\label{hackey idea}
%Any element of $RF(m)$ can be written in the form
%$$
%\prod_{k=0}^{m-1} x_{m-k}^{\alpha_{m-k}} \prod_{k=1}^{m-1}[\omega_k, x_k]
%$$
%with $\alpha_1,\dots, \alpha_m\in \Z$ and $\omega_1,\dots, \omega_{m-1}\in RF(m)$. 

%\chris{Proposed change for referee concern:}

For any $x\in RF(m)$, there are some $\alpha_1,\dots,\alpha_m\in \Z$ and $\omega_1,\dots \omega_{m-1}$ so that 
$$
x = \prod_{k=0}^{m-1} x_{m-k}^{\alpha_{m-k}} \prod_{k=1}^{m-1}z_k.
$$
Here, for each $k$,  $z_k$ can be chosen to be either $[\omega_k,x_k]$ or $[x_k,\omega_k]$.

\end{theorem}

Before proving Theorem~\ref{hackey idea}, we will use it to prove  Theorem~\ref{upper bound theorem main}.  We start with the proof in the special case of a string link in the image of $\phi:RF(n-1)\to \H(n)$.   

\begin{corollary}\label{nhL for RF}
Suppose that $T = T_1 \cup \dots \cup T_n \in \H(n)$ is in the image of $\phi:RF(n-1)\to \H(n)$.  Let $Q(T)=\#\{1\leq k<n-1\mid \lk(T_n,T_k)=0\}$. 
 Then $n_h(T)\le \Lambda(T)+2Q(T)$.  %\chris{Restated this to give better control when there are nonvnishing linking numbers.  Proof still needs updating.}
\end{corollary}
\begin{proof}

Let $T=\phi(t)$ with $t\in RF(n-1)$.  Recall that $\phi:RF(n-1)\to \H(n)$ is given by $\phi(x_i)=x_{i,n}$.  We apply Theorem~\ref{hackey idea} to $t$ with $z_k = [\omega_k,x_k]$ when $\alpha_k\leq 0$ and $z_k=[x_k,\omega_k]$ when $\alpha_k> 0$.  We then consider $T=\phi(t)$ and emphasize the terms of each product involving $x_{1,n}$,
%$$
%T=
%\prod_{k=0}^{n-2} x_{n-1-k, n}^{\alpha_{n-1-k}} \prod_{k=1}^{m-2}[\omega_k, x_{k,n}]
%=
%\left(\prod_{k=0}^{n-3} x_{n-1-k,n}^{\alpha_{n-1-k}}\right)\cdot x_{n,1}^{\alpha_{1}}\cdot z_1 \left(\prod_{k=2}^{m-2}[\omega_k, x_{n,k}]\right).
%$$
%\chris{replacement:}
$$
T=
\prod_{k=0}^{n-2} x_{n-1-k, n}^{\alpha_{n-1-k}} \prod_{k=1}^{n-2}z_k
=
\left(\prod_{k=0}^{n-3} x_{n-1-k,n}^{\alpha_{n-1-k}}\right)\cdot x_{1,n}^{\alpha_{1}}\cdot z_1 \left(\prod_{k=2}^{n-2}z_k\right).
$$
For notational ease, we have conflated $z_k$ with $\phi(z_k)$ and $\omega_k$ with $\phi(\omega_k)$.  If $\alpha_{1}>0$ then we can undo the center-most terms $x_{1,n}z_1 = x_{1,n}^{\alpha_{1}}\cdot [x_{1,n},\omega_1]$ %\chris{$x_{n_1}^{\alpha_1}z_1$} 
in $\alpha_{1}$ crossing changes.  Indeed,  
$$x_{1,n}^{\alpha_{1}}\cdot z_1 =x_{1,n}^{\alpha_{1}}\cdot[x_{1,n},\omega_1] = x_{1,n}^{(\alpha_{1}-1)}\omega_1^{-1} x_{1,n} \omega_1.$$
 After $\alpha_{1}$ crossing changes, this is transformed to $\omega_1^{-1}\omega_1=1$.  Similarly, if $\alpha_{1}<0$ then   $$x_{1,n}^{\alpha_1}\cdot z_1=x_{1,n}^{\alpha_{1}}\cdot[\omega_1, x_{1,n}] 
 =x_{1,n}^{\alpha_1}\omega_1^{-1}x_{1,n}^{-1}\omega_1x_{1,n}
=x_{1,n}^{(\alpha_{1}+1)}\omega_1^{-1} x_{1,n}^{-1} \omega_1$$
since $x_{1,n}$ commutes with $\omega_1^{-1} x_{1,n}^{-1} \omega_1$. Note this can be undone in $|\alpha_{n,1}|$ crossing changes. Finally, if $\alpha_{n,1}=0$, by Proposition \ref{prop: basics of nhl}(4), $z_i$ can be undone in 2 crossing changes.  
%After making the crossing changes described above, $T$ is reduced to 
%$$x_{n,n-1}^{\alpha_{n-1}}\dots x_{n,2}^{\alpha_{2}}\cdot[\phi(\omega_2), x_{n,2}]\dots[\phi(\omega_n-2), x_{n,n-2}].$$
Thus, if we set $q_{1} = \begin{cases}0&\text{ if }\alpha_{1}\neq 0\\2&\text{ if }\alpha_{1}= 0\end{cases},$ then after $|\alpha_{1}|+q_{1}$ crossing changes, $T$ is transformed into 
$$
\prod_{k=0}^{n-3} x_{n, n-1-k}^{\alpha_{n-1-k}}\prod_{k=2}^{n-2}[\omega_k, x_{n,k}].
$$
A direct induction now reveals that 
$$n_h(T)\le \sum_{k=1}^{n-1} |\alpha_{k}|+q_{k}
$$
where $q_{k} = \begin{cases}0&\text{ if }|\alpha_{k}|\neq 0\text{ or }k=n-1\\2&\text{otherwise}\end{cases}$.  Observing that $\alpha_{k}=\lk(T_n, T_k)$ and that $\Sum_{k=1}^{n-1} q_{k}= 2 \cdot Q(L)$ completes the proof.
\end{proof}

%Thus giving Theorem 1.1. \anthony{Question: Is this sharp in general? We only argued it is when the link is algebraically slice.} \chris{We might be able to answer this when $n=4$ using our "4-component" section, but I do not know how good a lower bound we can do in this case with our current techniques...}

Now suppose that $L=L_1\cup\dots\cup L_n$ is a Brunnian link.  It follows then that $L_1\cup\dots\cup L_{n-1}$ is the unlink, and if we realize $L$ as $\widehat{T}$ for some $T\in \H(n)$ then we may take $T_1\cup\dots\cup T_{n-1}$ to be the trivial string link and thus $T$ is in the image of $\phi:RF(n-1)\to \H(n)$.  If $L$ is Brunnian and has at least 3 components, then all of the pairwise linking numbers vanish, so $\Lambda(T)=0$ and $Q(T)=n-2$. 
 The corollary below follows.

 \nhlForBruunian

%\begin{corollary}\label{cor: nhl for brunnian}
%Let $n\ge 3$ and $L$ be an $n$-component Brunnian link.  Then $n_h(L)\le 2(n-2)$.
%\end{corollary}

Induction and the decomposition $\H(n)\cong \H(n-1)\ltimes RF(n-1)$ now lets us control the homotopy trivializing number over all of $n$-component links.

\UpperBoundTheorem

\begin{proof}
We proceed inductively on the number of components. 
 Realize $L$ as $L=\widehat T$ for some $T\in \H(n)$. As a consequence of the split exact sequence of \pref{exact sequence},  $T=\phi(S)s(T')$ with $S\in RF(n-1)$ and $T'\in \H(n-1)$. By Corollary~\ref{nhL for RF}, $n_h(\phi(S))\le \Lambda(\phi(S))+2Q(\phi(S))$.  Appealing to induction, 
 $
 n_h(s(T'))=n_h(T')\le \Lambda(T')+2Q(T').
 $
   Putting this together,
 $$
 n_h(L)\le n_h(\phi(S))+n_h(T')\le \Lambda(\phi(S))+\Lambda(T')+2Q(\phi(S))+2Q(T') = \Lambda(L)+2Q(L).
 $$  This completes the proof.  
\end{proof}

\subsection{Representing elements of $RF(m)$ as products without too many commutators.}

In this subsection we prove Theorem~\ref{hackey idea}.  
We begin this with a recollection of basic properties of commutators and the lower central series.  A standard reference is the work of Magnus-Karrass-Solitar \cite[Chapter 5]{MKS76}.  We begin by describing elementary commutators and their weight. We then apply these facts to the reduced free group. %This section closes with the proof of the upper bound on $C_n$ in Theorem~\ref{thm: main}.

\begin{definition}
Let $G$ be a group and $x_1,\dots, x_m$ be a generating set for $G$.  We call $x_1,x_1^{-1},\dots, x_m, x_m^{-1}$ \emph{weight 1 elementary commutators}.  If $c_1$ and $c_2$ are elementary commutators of {weight} $w_1$ and $w_2$ respectively, then $[c_1, c_2]$ is an elementary commutator of weight $w_1+w_2$.  
\end{definition}

If $c$ is an elementary commutator of weight $w$, then we write $\wt(c)=w$.  Note that as $[c_1,c_2]^{-1} = [c_2,c_1]$, the set of elementary commutators of weight $w$ is closed under the inverse operation.

\begin{definition}
If $H$ and $J$ are subgroups of $G$, then $[H,J]\le G$ is the subgroup generated by elements of the form $[h,j]$ with $h\in H$ and $j\in J$.
\end{definition}

\begin{definition}
The \emph{lower central series} of a group $G$ is defined recursively by $G_1=G$ and $G_{k+1}=[G_k, G]$.
\end{definition}

We give several well-known properties of commutators and their behavior modulo lower central series quotients.  Many of these are grouped together in \cite[Theorem 5.1]{MKS76} as the Witt-Hall identities. 
%as the Hall-Witt identities.  

\begin{proposition}\label{prop: commutators}
Let $G$ be a group with generators $x_1,\dots, x_m$.  Let $a,b,c\in G$.  
\begin{enumerate}
\item \cite[Theorem 5.3 (8)]{MKS76} $[G_k, G_\ell] \subseteq G_{k+\ell}$.
\item\label{lower central normal} $G_k\unlhd G$ is a normal subgroup.
\item \label{LCS Abelian}
%\cite[Theorem 5.4]{MKS76} 
$G_k/G_{k+1}$ is an Abelian group generated by the set of all weight elementary $k$ commutators.
\item \label{item Commutator product 1} \cite[Theorem 5.1 (9), (10)]{MKS76} $[a, bc]=[a,c][a,b][[a,b],c]$ and $[bc,a] = [b,a][[b,a],c][c,a]$.
\item \cite[Theorem 5.3 (5), (6)]{MKS76}  \label{item Commutator product 2} If $a\in G_u$, $b\in G_v$ and $c\in G_w$ then in $G/G_{u+v+w}$, $[a, bc]=[a,b][a,c]$ and $[bc,a] = [b,a][c,a]$.
\item \label{commutator inverse -1}\cite[Theorem 5.1 (8)]{MKS76} $[a,b]^{-1} = [b,a]$.
\item \label{commutator inverse} $[a, b^{-1}]=[a,b]^{-1}[b, [a,b^{-1}]]$ and $[a^{-1},b]=[a,b]^{-1}[a, [a^{-1},b]]$.
\item \label{commutator inverse 2} If $a\in G_u$ and $b\in G_v$ then in $G/G_{u+2v}$, $[a,b^{-1}] = [a,b]^{-1}$.  In $G/G_{2u+v}$, $[a^{-1},b]=[a,b]^{-1}$.
\item \label{commutation well defined LCS} If $a=b$ in $G/G_u$ and $c\in G_v$ then $[a,c]=[b,c]$ in $G/G_{u+v}$.
\end{enumerate}
\end{proposition}
\begin{proof}
We prove only those results which do not explicitly appear in \cite{MKS76}.  If $A$ and $B$ are normal in $G$, then $[A,B]$ is also normal (see for example \cite[Lemma 5.1]{MKS76}).  Together with induction, \pref{lower central normal} follows.  From \pref{LCS Abelian} follows from \cite[Theorem 5.4]{MKS76} since the simple $k$-fold commutators defined in \cite[Section 5.3]{MKS76} are all elementary weight $k$ commutators.  
 
If $a,b \in G$ then by \pref{item Commutator product 1} 
$$
\begin{array}{l}
1=[a,b^{-1}\cdot b] =[a,b][a,b^{-1}][[a,b^{-1}],b],
\\
1=[a^{-1}\cdot a, b] = [a^{-1},b][[a^{-1}, b],a][a,b].\end{array}
$$claim \pref{commutator inverse} follows.  If $a\in G_u$ and $b\in G_v$, then  $[b, [a,b^{-1}]]\in [[G_u, G_v],G_v]\subseteq G_{u+2v}$, proving \pref{commutator inverse 2}.  Finally, if $a=b$ in $G/G_u$, then $b=a q$ with $q\in G_u$, and $$[b,c] = [aq,c]=[a,c][[a,c],q][q,c].$$
If $c\in G_v$ then $[[a,c],q]$ and $[q,c]$ are each in $G_{u+v}$, proving \pref{commutation well defined LCS}.  
\end{proof}

Recall that the reduced free group $RF(m)$ on letters $x_1,\dots, x_m$ is the quotient of the free group on $x_1,\dots, x_m$ given by requiring each conjugate of $x_i$ to commute with each other conjugate of $x_i$ for all $i$.  This results in some commutativity relations among commutators.  First, we explain recursively the fairly intuitive notion of what it means for a generator to be ``\emph{in}" an elementary commutator.  

\begin{definition}
Let $x_1,\dots, x_m$ be generators of a group $G$.  We say that $x_i$ is \emph{in} $x_j$ (and $x_i$ is in $x_j^{-1}$) with multiplicity 1 if $i=j$ and otherwise $x_i$ is in $x_j$ (and $x_j^{-1}$) with multiplicity $0$.  If $a$ and $b$ are elementary commutators such that $x_i$ is in $a$ with multiplicity $p$ and $x_i$ is in $b$ with multiplicity $q$, then $x_i$ is in $[a,b]$ with multiplicity $p+q$.  Whenever $x_i$ is in $a$ with multiplicity greater than 0, we will simply say that $x_i$ is in $a$.  
\end{definition}

The reader should compare this to the notion of simple $k$-fold commutators from \cite[Section 5.3]{MKS76}.  

\begin{proposition}\label{prop: RF commute}
If $a$ and $b$ are elementary commutators in $RF(m)$ and $x_i$ is in each of $a$ and $b$, then for any $\gamma,\delta\in RF(m)$ and any $k,\ell\in \Z$, $[\gamma a^k\gamma^{-1},\delta b^\ell \delta^{-1}]=1$ in $RF(m)$.   
\end{proposition}
\begin{proof}
The definition of the reduced free group immediately implies that the normal subgroup generated by $x_i$ is Abelian.  We proceed by demonstrating by induction on $\wt(a)$ that if $x_i$ is in an elementary commutator $a$ then $a$ is in the normal subgroup generated by $x_i$.  When $\wt(a)=1$, $a=x_i$ or $a=x_i^{-1}$ and we are done.  

If $\wt(a)>1$ then $a=[u,v]$ and $x_i$ is in at least one of $u$ and $v$.  Without loss of generality, assume that it is in $u$, so that we may inductively assume that $u$ is in the normal subgroup generated by $x_i$.    Thus, $u^{-1}$ and $v^{-1} u v$ are each in the normal subgroup generated by $x_i$.  As a consequence, $a=[u,v]=u^{-1} (v^{-1} u v)$ is in the normal subgroup generated by $x_i$. 

Thus, each of $\gamma a^k \gamma^{-1}$ and $\delta b^\ell \delta^{-1}$ is in the normal subgroup generated by $x_i$, which is Abelian by the definition of $RF(m)$.  We conclude that $[\gamma a^k\gamma^{-1},\delta b^{\ell} \delta^{-1}]=1$ as we claimed.  
\end{proof}

\begin{proposition}\label{prop inverses in RF}
For any elementary commutators $a$ and $b$ in $RF(m)$ and $k\in \Z$, $[a,b]^{k}=[a^{k},b]=[a, b^{k}]$.
\end{proposition}
\begin{proof}
For $k\ge 0$ we proceed by induction.  By Proposition~\ref{prop: commutators}, \pref{item Commutator product 1}
$$
[a^k,b] = [a^{k-1},b][[a^{k-1},b],a][a,b].
$$
Let $x_i$ be any of the preferred generators of $RF(n)$ which is in $a$.  Then $[a^{k-1},b]$ and $a$ each sit in the normal subgroup generated by $x_i$, which is Abelian.  Thus, $[[a^{k-1},b],a]=1$.  Appealing to an inductive assumption completes the argument when $k\ge 0$.  

It suffices now to verify the claim when $k=-1$.  By Proposition~\ref{prop: commutators}, \pref{commutator inverse}, $[a^{-1},b] = [a,b]^{-1}[a,[a^{-1},b]]$.  As $a$ and $[a^{-1},b]$ each sit in the normal subgroup generated by some $x_i$, $[a,[a^{-1},b]]=1$.  

Since $[a, b^k] = [b^k, a]^{-1}$ the final claimed identity follows.  
\end{proof}

\begin{proposition}
    \label{prop:rearrage}
    For any elementary commutators $a$, $b$ and $c$,
    \[[[a,b],c] = [[c,b],a][[a,c],b]\]
    in the reduced free group.
\end{proposition}
\begin{proof}
%Follow the convention that $[x,y]=\overline{x}\, \overline y x y$ so that $xy=yx[x,y]$.

\begin{comment}
\chris{Revised approach using results in \cite{MKS76}.}

By \cite[Theorem 5.1 (12)]{MKS76}
$$
[[a,b],c][[b,c],a][[c,a],b] = [b,a][c,a][c,b]^a[a,b][a,c]^b[b,c]^a[a,c][c,a]^b
$$
where $x^y=y^{-1}xy$.  Any pair of commutators on the right hand side of this equation have at least one of the commutators $a$, $b$, and $c$ in common.  As a consequence,the factors on the right hand side all commute.  Since $[x,y]^{-1}=[y,x]$, we see that all of the factors above cancel;
$$
[[a,b],c][[b,c],a][[c,a],b] = [b,a][a,b][c,a][a,c][c,b]^a[b,c]^a[a,c]^b[c,a]^b=1,
$$
and so
$$
[[a,b],c] = [[c,a],b]^{-1}[[b,c],a]^{-1}.$$
By Proposition~\ref{prop inverses in RF},
$[[c,a],b]^{-1} = [[c,a]^{-1},b]=[[a,c],b]$.  Repeat a similar analysis on $[[b,c],a]^{-1}$ to conclude that
$$[[a,b],c] = [[a,c],b][[c,b],a].
$$
\end{comment}

This can be shown directly by expanding out the commutator $[[a,b],c]$ and then using that $xy = yx[x,y]$ to follow the algorithm that realizes the Hall Basis Theorem \cite[Theorem 5.13 A]{MKS76} to gather terms together.   What results is
$$[[a,b],c] = [\Inv{a},b]~ [\Inv{a},c] [\Inv b,c][[\Inv{b},c],a][a,b][a,c][[a,c],b][b,c].$$
Verifying the preceding step is a useful exercise in commutator calculus.
Each pair of commutators in this product has at least one of $a$, $b$, and $c$ in common, and so by Proposition~\ref{prop: RF commute} they commute in $RF(n)$.  Additionally, by Proposition~\ref{prop inverses in RF},
%as a consequence of Proposition~\ref{prop: commutators} \pref{commutator inverse} and Proposition~\ref{prop: RF commute},  in the reduced free group
$[\Inv{x},y] = \Inv{[x,y]}$ whenever $x$ and $y$ are elementary commutators. As a consequence, most of the terms in this product cancel.
%and we are left with
%\begin{eqnarray*}[[a,b],c]
%&=&[[\Inv{b},c],a][[a,c],b].
%\end{eqnarray*}
Finally, $[[\Inv{b},c],a] = [\Inv{[b,c]},a]=[[c,b],a]$ by Proposition~\ref{prop inverses in RF}. An alternative yet similar argument can be composed by starting with \cite[Theorem 5.1 (12)]{MKS76} and then using properties of the reduced free group.
%\begin{eqnarray*}[[a,b],c]
%&=&[[c,b],a][[a,c],b]
%\end{eqnarray*}
\end{proof}

%\footnote{\katherine{Commented out this corollary as we agreed it wasn't needed.}}
\begin{comment}
\begin{corollary}\label{cor:Rearrange commutators}
For commutators $w,x,y,z$, 
$$[[w,x],[y,z]] = [[[y,z],x],w][[w,[y,z]],x]
$$
in the reduced free group.

\chris{Comment: THis corollary is not used later?  }  \katherine{It's not.  I think we can remove it.} \anthony{Agreed.}
\end{corollary}
\begin{proof}
Proposition \ref{prop:rearrage} states $[[a,b],c] = [[c,b],a][[a,c],b]$; assign $a=w$, $b=x$, $c=[y,z]$.
\end{proof}
\end{comment}

We are now ready to progress in earnest to the proof of Theorem~\ref{hackey idea}.   
\begin{lemma}\label{lem: commutator as nice product}
If $c$ is an elementary commutator of weight $\wt(c)\ge 2$ in $RF(m)$, then there is a sequence $c_1,\dots, c_{N(c)}$ of elementary commutators  of weight $\wt(c)-1$ %and their inverses
and $i_1,\dots, i_{N(c)}<m$ so that 
$$
c=\left(\Prod_{j=1}^{N(c)} [c_j, x_{i_j}]\right) c'
$$
with $c'\in RF(m)_{\wt(c)+1}$.
\end{lemma}
\begin{remark}
Without the condition $i_1,\dots, i_{N(c)} <m$, this lemma 
%is an immediate consequence of the Hall Basis Theorem.
has no content.
The key result is that any such $c$ can be realized without any factors of the form $[c_j,x_m]$ ever appearing.  We encourage the reader to run the proof below on the example of $[[x_1, x_2],x_m]$.
\end{remark}
\begin{proof}

Let $c$ be an elementary commutator of weight $w\ge 2$.  Then $c=[a,b]$ for some elementary commutators with $\wt(a)+\wt(b)=\wt(c)$.  If $x_m$ is in both $a$ and $b$ then $[a,b]=1$ by Proposition~\ref{prop: RF commute}.  Without loss of generality, we may assume that $x_m$ is not in $a$, for if $x_m$ were in $a$ then by Proposition~\ref{prop inverses in RF} $c=[a,b]=[b,a]^{-1}=[b^{-1},a]$ would be in $RF(m)$. %\footnote{\katherine{I know we've used this fact a lot but it's not obvious to me why it's true.}  \chris{Proposition~\ref{prop inverses in RF} has been added}.}

We now proceed by induction on the weight of $a$.  As a base case, if $a$ is weight 1 then $a=x_i$ (or $x_i^{-1}$) for some $i< m$ and $c=[x_i,b]=[b^{-1}, x_i]$ (or $c=[x_i^{-1},b] = [b, x_i]$) so we are done.  We now assume that $\wt(a)>1$ so that $a=[\alpha,\beta]$.  As $x_m$ is not in $a$, it is in neither $\alpha$ nor $\beta$.  Finally, since $\wt(a) = \wt(\alpha)+\wt(\beta)$,  $\wt(\alpha)<\wt(a)$.   We now appeal to our inductive assumption to conclude that $a = \Prod_{j=1}^{N(a)} [a_j, x_{i_j}] a'$ where $i_j<m$, $a_j$ is an elementary commutator with $\wt(a_j)=\wt(a)-1$ and $a'\in RF(m)_{\wt(a)+1}$.  
%Since $x_m$ is not in $a$, it follows that $a$ is a commutator in the copy $RF(m-1)$ generated by $x_1,\dots, x_{m-1}$.  Since $\wt(a)>1$, $a=[u,v]$ for some commutators $u$ and $v$ of weight less than $\wt(a)$.  Moreover $x_m$ is in neither $u$ nor $v$.  Therefore by our inductive assumption, we conclude that $a=\Prod_{j=1}^n [a_j, x_{i_j}] a'$ with $a'\in RF(m)_{\wt(a)+1}$, each $a_j\in RF(m-1)$ being a commutator (or its inverse) with $\wt(a_j)=\wt(a)-1$ and $i_j<m-1$.
Thus,
$$
c=[a,b] = \left[\Prod_{j=1}^{N(a)} [a_j, x_{i_j}] a', b\right].
$$

Notice that $\Prod_{j=1}^{N(a)}[a_j, x_{i_j}]\in RF(m)_{\wt(a_j)+1} =  RF(m)_{\wt(a)}$, $a'\in RF(m)_{\wt(a)+1}$ and $b\in RF(m)_{\wt(b)}$.  Appealing to Proposition~\ref{prop: commutators} \pref{item Commutator product 2}, it follows that modulo $RF(m)_{2\wt(a)+1+\wt(b)}\subseteq RF(m)_{\wt(c)+1}$,
$$
c \equiv \left[\Prod_{j=1}^{N(a)} [a_j, x_{i_j}], b\right]\cdot \left[ a', b\right].
$$
Since $\left[ a', b\right]\in RF(m)_{\wt(a)+1+\wt(b)} = RF(m)_{\wt(c)+1}$, we have that modulo $RF(m)_{\wt(c)+1}$, 
$$c\equiv\left[\Prod_{j=1}^{N(a)} [a_j, x_{i_j}], b\right].$$
Since each $[a_j, x_{i_j}]\in RF(m)_{\wt(a)}$ and $b \in RF(m)_{\wt(b)}$, we may iteratively apply Claim~\pref{item Commutator product 2} of Proposition~\ref{prop: commutators} to see that modulo $RF(m)_{2\cdot \wt(a)+\wt(b)}\subseteq RF(m)_{\wt(c)+1}$, 
$$c\equiv\Prod_{j=1}^{N(a)} \left[ [a_j, x_{i_j}], b\right].
$$
Next apply Proposition~\ref{prop:rearrage} to obtain,  
$$c\equiv\Prod_{j=1}^{N(a)} \left[ [b, x_{i_j}], a_j\right]\left[[a_j, b], x_{i_j}\right]\mod RF(m)_{\wt(c)+1}.
$$
Recall that $x_m$ is not in $a_j$ and $\wt(a_j)=\wt(a)-1$.  Thus, we may again apply the inductive assumption and conclude that for each $j=1,\dots N(a)$, there is some $ N\left([ [b, x_{i_j}], a_j]\right)\in \N$ so that
$$\left[ [b, x_{i_j}], a_j\right]\equiv \Prod_{k=1}^{N\left([ [b, x_{i_j}], a_j]\right)} [c_{j,k}, x_{i_{k,j}}]\mod{RF(m)_{\wt(c)+1}}%c_{j,k}'
$$
with $\wt(c_{j,k}) = \wt(c)-1$
%, $c_{j,k}'\in RF(m)_{\wt(c)+1}$,
and $i_{k,j}<m$.  Putting this together, %modulo $RF(m)_{\wt(c)+1}$,
$$
c\equiv\Prod_{j=1}^{N(a)}  \left(\Prod_{k=1}^{N\left([ [b, x_{i_j}], a_j]\right)} [c_{j,k}, x_{i_{k,j}}]\right)\left[[a_j, b], x_{i_j}\right]\mod{RF(m)_{\wt(c)+1}}
$$
which completes the proof.
\end{proof}

We now have everything we need to prove Theorem~\ref{hackey idea}.

  \begin{proof}[Proof of Theorem~\ref{hackey idea}]
Let $z\in RF(m)$.  We will inductively show that for all $p\in \N $,% in $RF(m)/RF(m)_p$
%$$z = \prod_{k=0}^{m-1} x_{m-k}^{\alpha_{m-k}}  \prod_{k=1}^{m-1}[\omega_k, x_k].$$
%\chris{Revision for referee concern:}
$$z \equiv \prod_{k=0}^{m-1} x_{m-k}^{\alpha_{m-k}}  \prod_{k=1}^{m-1}z_k \mod RF(m)_p
$$
where $z_k = [\omega_k,x_k]$ or $[x_k,\omega_k]$ for each $k$.
When $p=2$, $RF(m)/RF(m)_p=\Z^m$ is the free Abelian group on $x_1,\dots, x_m$, so 
$$z \equiv x_m^{\alpha_m}\cdot x_{m-1}^{\alpha_{m-1}}\cdot\dots \cdot x_1^{\alpha_1}\cdot z' = \prod_{k=0}^{m-1} x_{m-k}^{\alpha_{m-k}}\mod{RF(m)_2}.$$
%with $z'\in RF(m)_2$.
This completes the proof when $p=2$.  

%\cotto{we may want to switch to $p$ and prove $p+1$}

For convenience, we set $z_0=\Prod_{k=0}^{m-1} x_{m-k}^{\alpha_{m-k}}$.  We now inductively assume that 
%$$z = z_0 \prod_{k=1}^{m-1}[\omega_k, x_k]\cdot z'$$ with $z'\in RF(m)_{p}$.
%\chris{revision}: 
$$
z=z_0\prod_{k=1}^{m-1}z_k\cdot z'
$$ with $z'\in RF(m)_{p}$ and $z_k$ as in the theorem. Appealing to Proposition~\ref{prop: commutators} \pref{LCS Abelian}, modulo $RF(m)_{p+1}$, $z'$ is a product of weight $p$ elementary commutators and so we can express
$z'\equiv \Prod_{q} c_q\mod{RF(m)_{p+1}}$ where each $c_q$ is a weight $p$ elementary commutator.
%and $z''\in RF(m)_{p+1}$.  
Appealing to Lemma~\ref{lem: commutator as nice product},% modulo $RF(m)_{p+1}$ each $c_q$ can be rewritten as a product in the form 
$c_q\equiv\Prod_r[d_{q,r},x_{i_{q,r}}]\mod{RF(m)_{p+1}}$ with $i_{q,r}<m$ and $\wt(d_{q,r})=p-1$.   
%
%\cotto{for me...first equals is because central  and the second is reduced free group...The $r$'s and $\beta_q$ are different}
%
Still working modulo $RF(m)_{p+1}$, these factors commute by Proposition~\ref{prop: commutators} \pref{LCS Abelian}, so we can relabel and reorder this product so as to sort by $x_{i_{q,r}}$'s.  
$$z'\equiv\prod_{i=1}^{m-1}\prod_q [d_{q,i}, x_{i}] \mod RF(m)_{p+1}.$$ 
We start by rewriting $\Prod_q [d_{q,i},x_i]$.  For each $i$, if $z_i=[\omega_i,x_i]$, then we use  Proposition~\ref{prop: commutators} \pref{item Commutator product 2} to say
% Lemma~\ref{lem:product in comm}, 
 %$[d, x_i][e,x_i] = [de,x_i][[d,x_i],\overline{c}]\equiv[de,x_i]$ modulo $RF(m)_{k+2}$ 
 $$\prod_q[d_{q,i}, x_i]\equiv\left[\prod_q d_{q,i},x_i\right]\mod RF(m)_{p+1}.$$  Set $D_i=\Prod_q d_{q,i}$ and $W_i=[D_i,x_i]$.
 
 On the other hand, if $z_i=[x_i,\omega_i]$, then we use Proposition~\ref{prop: commutators} \pref{commutator inverse -1}, \pref{commutator inverse 2}, and \pref{item Commutator product 2}, 
 $$\Prod_q[d_{q,i}, x_i] = \Prod_q[x_i,d_{q,i}]^{-1}\equiv \Prod_q[x_i,d_{q,i}^{-1}]\equiv \left[x_i,\Prod_q d_{q,i}^{-1}\right]\mod RF(m)_{p+1}.$$ Set $D_i=\Prod_q d_{q,i}^{-1}$ and $W_i=[x_i, D_i]$. 
 We now have  
$$z'\equiv\prod_{i=1}^{m-1} W_i\mod RF(m)_{p+1},$$
with $W_i=[D_i,x_i]$ or $[x_i,D_i]$ and $D_i\in RF(m)_{p-1}$. Therefore,
\begin{eqnarray*}
z&=&z_0 \prod_{k=1}^{m-1}z_k\cdot z'
\equiv z_0 \prod_{k=1}^{m-1}z_k\cdot \prod_{k=1}^{m-1} W_k\,  \mod RF(m)_{p+1}.
\end{eqnarray*}
As $W_k\in RF(m)_{p}$ is central in $RF(m)/RF(m)_{p+1}$ it follows that
$$
z\equiv z_0 \prod_{k=1}^{m-1}z_kW_k \mod RF(m)_{p+1}.
$$
Appealing to Proposition~\ref{prop: commutators} \pref{item Commutator product 2}, either
$$
z_kW_k = [\omega_k, x_k][D_k,x_k]\equiv[\omega_kD_k, x_k]\,\mod RF(m)_{p+1}
$$
or
$$
z_kW_k = [x_k,\omega_k][x_k,D_k]\equiv[ x_k,\omega_kD_k]\,\mod RF(m)_{p+1}.
$$
If we set $\omega_k'=\omega_k D_k$ and where $z_k' = [x_k, \omega_k']$ or $[\omega_k', x_k]$, then 
$$
z \equiv z_0\Prod_{k=1}^{m-1}z_k'\mod{ RF(m)_{p+1}},
$$
 completing the inductive step.  We conclude that for every $p\in\N$, 
$$
z\equiv \prod_{k=0}^{m-1} x_{m-k}^{\alpha_{m-k}} \prod_{k=1}^{m-1}z_k \mod RF(m)_{p}
$$
where $z_k$ is equal to $[\omega_k,x_k]$ or $[x_k,\omega_k]$ as desired.  Since $RF(m)_{m+1}=1$, taking $p=m+1$ completes the proof.
%}
\begin{comment}Finally, $[D_m, x_m]\in RF(m)/RF(m)_{k+2}$ depends only on $D_m\in RF(m)_k/RF(m)_{k+2}$, and so we can write 
$[D_m, x_m] = [\prod_k C_k, x_m]\equiv \prod_k [C_k, x_m]$. If we can rewrite this as a product of the form $[E_j, x_{i_j}]$ with each $i_j\neq m$ then we are happy...at least not sad.  
\end{comment}
\end{proof}

\section{Link homotopy and $\Delta$-moves: The proof of Theorem \ref{thm: main Delta}}\label{sect: Delta change}

A very similar proof to that of Theorem~\ref{hackey idea} produces the following result.

\begin{theorem}\label{DeltaMoveVersion}
Any element $z \in RF(m)$ can be written in the form
$$z=
\prod_{i=1}^m x_i^{\alpha_i}\prod_{1\le i\le j\le m} [x_j, x_i]^{\beta_{i,j}} \prod_{i=1}^{m-1}\prod_{j=1}^{m-1}[[\omega_{i,j}, x_i],x_j]
$$
with $\alpha_i, \beta_{i,j},\in \Z$ and $\omega_{i,j}\in RF(m)$.    
\end{theorem}

\begin{proof}[Proof of Theorem~\ref{DeltaMoveVersion}]

Let $z\in RF(m)$.  We will inductively show that for all $p\in \N$,%in $RF(m)/RF(m)_p$,
$$z \equiv \prod_{i=1}^mx_i^{\alpha_i}\prod_{1\le i\le j\le m} [x_j, x_i]^{\beta_{i,j}} \prod_{i=1}^{m-1}\prod_{j=1}^{m-1}[[\omega_{i,j}, x_i],x_j] \mod{RF(m)_p}.$$
When $p\le 3$, the result follows from the Hall Basis Theorem, \cite[Theorem 5.13]{MKS76}.

We now take $p\ge 3$ and inductively assume that $$z =  \prod_{i=1}^m x_i^{\alpha_i}\prod_{1\le i\le j\le m} [x_j, x_i]^{\beta_{i,j}} \prod_{i=1}^{m-1}\prod_{j=1}^{m-1}[[\omega_{i,j}, x_i],x_j]\cdot z'$$ with $z'\in RF(m)_{p}$.  By Proposition~\ref{prop: commutators} \pref{LCS Abelian}, 
$z'\equiv \Prod_{q} c_q \mod{RF_{p+1}}$ where each $c_q$ is an elementary weight $p$ commutator.  As in the proof of Theorem~\ref{hackey idea} we apply Lemma~\ref{lem: commutator as nice product} to each $c_q$, to see that modulo $RF(m)_{p+1}$, $z'$ can be rewritten as a product in the form $z'\equiv \Prod_r[d_{r},x_{i_{r}}]$ where $i_{r}<m$ and $d_r$ is a weight $p-1$ commutator not containing $x_{i_r}$. This product commutes modulo $RF(m)_{p+1}$ so that we may sort it by $i_r$.  Next, using claim \pref{item Commutator product 2} of Proposition \ref{prop: commutators}, we see that 
$$z'\equiv \Prod_{i=1}^{m-1}[D_i,x_i]\mod RF(m)_{p+1}$$
where $D_i\in RF(m-1)_{p-1}$ sits in the $(p-1)$'th term of the lower central series of the copy of $RF(m-1)$ generated by $x_1, \dots, x_{i-1}, x_{i+1},\dots, x_m$.  
Applying Proposition~\ref{prop: commutators} \pref{LCS Abelian}, we may write each $D_i$ as a product of elementary weight $p-1$ commutators, and then apply Lemma \ref{lem: commutator as nice product} to each rewrite each of these elementary commutators,
$$z'\equiv \Prod_{i=1}^{m-1}\left[\Prod_{q=1}^{N(D_i)}[e_{iq},x_{i_q}],x_i\right]\mod RF(m)_{p+1}$$
where $i_q<m$, each $e_{iq}$ is an elementary commutator of weight $p-2$, and $N(D_i)$ is number of commutators needed for $D_i$.  Again using commutativity and claim \pref{item Commutator product 2} of Proposition \ref{prop: commutators} similarly to before, 
$$z'\equiv \Prod_{i=1}^{m-1}\Prod_{j=1}^{m-1}\left[[E_{ij},x_{j}],x_i\right]\mod RF(m)_{p+1},$$
where $E_{ij}\in RF(m)_{p-2}$.  Since $RF(m)_p$ is central in $RF(m)/RF(m)_{p+1}$, 
$$
\begin{array}{rcl}z &\equiv&  \displaystyle\prod_{i=1}^m x_i^{\alpha_i}\prod_{1\le i\le j\le m} [x_j, x_i]^{\beta_{i,j}} \prod_{i=1}^{m-1}\prod_{j=1}^{m-1}\left([[\omega_{i,j}, x_i],x_j][[E_{ij}, x_i],x_j]\right)
\\
&\equiv&\displaystyle \prod_{i=1}^m x_i^{\alpha_i}\prod_{1\le i\le j\le m} [x_j, x_i]^{\beta_{i,j}} \prod_{i=1}^{m-1}\prod_{j=1}^{m-1}[[\omega_{i,j}E_{ij}, x_i],x_j].\end{array}$$

We now inductively conclude that the claim holds modulo $RF(m)_p$ for any $p$.  Since $RF(m)_{m+1}=1$, we can set $p=m+1$ to complete the proof.
\end{proof}

\begin{corollary}
\label{nDeltaL for RF}
If $T$ has vanishing pairwise linking numbers and is in the image of $RF(n-1)\to \H(n)$ then 
$$n_{\Delta}(T)\le 2(n-2)(n-3)+\Sum_{1\le i<j< n} |\mu_{ijn}(T)|.$$
\end{corollary}
\begin{proof}

We begin by using Theorem~\ref{DeltaMoveVersion} with $m=n-1$ to rewrite $T$.  Note that as $T$ has vanishing pairwise linking numbers each of the $a_i$ in  Theorem~\ref{DeltaMoveVersion} vanish, and $T=\phi(z)$ where
$$
z=\prod_{1\le i\le j\le n-1} [x_j, x_i]^{\beta_{i,j}} \prod_{i=1}^{n-2}\prod_{j=1}^{n-2}[[\omega_{i,j}, x_i],x_j].
$$

Since $[[\omega_{i,j}, x_i],x_j]=1$ whenever $i=j$, the product $\prod_{i=1}^{n-2}\prod_{j=1}^{n-2}[[\omega_{i,j}, x_i],x_j]$ has at most $(n-2)(n-3)$ terms. Each of these terms reduces as 
$$\alpha_{ij}:=[[\omega_{i,j}, x_i],x_j] = x_i^{-1} [ \omega_{ij}^{-1} x_i^{-1}\omega_{ij}, x_j] x_j^{-1} x_i x_j.$$

By Proposition~\ref{prop: basics of nhl} \pref{item:ndelta(commutator)}, $n_{\Delta}(\phi[ \omega_{ij}^{-1} x_i^{-1}\omega_{ij}, x_j])\le 1$, so that after one $\Delta$-move, $\phi(\alpha_{ij})$ is reduced to $\phi([x_i,x_j])$ which in turn has $n_{\Delta}=1$.  Thus $n_{\Delta}(\phi(\prod_{i=1}^{n-2}\prod_{j=1}^{n-2}[[\omega_{i,j}, x_i],x_j]))\le 2(n-2)(n-3)$. 

We focus now on the remaining product, $\Prod_{1\le i\le j\le n-1} [x_j, x_i]^{\beta_{i,j}}$.  Again appealing to Proposition~\ref{prop: basics of nhl} \pref{item:ndelta(commutator)}, $n_{\Delta}(\phi( \Prod_{1\le i\le j\le n-1} [x_j, x_i]^{\beta_{i,j}}))\le \Sum_{1\le i\le j\le n-1} |\beta_{i,j}|$.  The proof is completed by noting that
 $|\beta_{i,j}| = |\mu_{ijn}(\widehat{T})|$. In order to see this, first observe that the closure of $\phi([x_j, x_i])$ is a split link $L$ with $L_i\cup L_j\cup L_n$ tied into the Borromean rings and the other components unlinked.  In \cite{M1}, Milnor shows $\mu_{ijn}(L)=-1$.  Since the first non-vanishing Milnor invariant is additive by \cite{O1}, we conclude $\mu_{ijn}(T) = \mu_{ijn}(\widehat{T}) = -\beta_{ij}$.  Putting this all together, $n_{\Delta}(T)\le 2(n-2)(n-3)+\Sum_{1\le i<j< n} |\mu_{ijn}(T)|.$
\end{proof}

\thmMainDelta

\begin{comment}****COMMENTED OUT: REPLACED WITH RESTATABLE
\begin{corollary}\chris{**This is Theorem~\ref{thm: main Delta} - We should use reptheorem.}
    Let $L$ be an $n$-component link with vanishing pairwise linking numbers.  Then 
    \[ n_\Delta (L) \leq \Sum_{1\le i<j<k\le n} |\mu_{ijk}(L)|+\Sum_{k=3}^{n-1} 2k(k-1) = \Lambda_{3}(L)+\frac{2}{3} (n^3 - 3n^2 + 2n - 6)\]
\end{corollary}
\end{comment}
\begin{proof}

%CHRIS NEEDS TO CHANGE THIS STATEMENT AND PROOF IN LIGHT OF 5.2's CHANGE
    
Our proof follows an induction identical to that of Theorem~\ref{upper bound theorem main}.  
 Realize $L$ as $L=\widehat T$ for some $T\in \H(n)$. As a consequence of the split exact sequence of \pref{exact sequence},  $T=\phi(t)s(T')$ with $t\in RF(n-1)$ and $T'\in \H(n-1)$. By Corollary~\ref{nDeltaL for RF}, 
 $$n_\Delta (\phi(t))\leq 2(n-2)(n-3)+ \Sum_{1\le i<j\le n-1 }|\mu_{ijn}(L)|.$$  Appealing to induction, 
 $$
 n_\Delta(s(T'))=n_{\Delta}(T')\le \Lambda_3(T')
 +\frac{2}{3}((n-1)^3 - 6(n-1)^2 + 11(n-1) - 6).
 %+\Sum_{k=3}^{n-2} 2k(k-1).
 $$
Adding these together,
 $$
 n_\Delta(L)= n_\Delta(T)\le n_\Delta(\phi(t))+n_\Delta(s(T'))\le 
\Lambda_3(L)
%+\Sum_{k=3}^{n-1} 2k(k-1),
+\frac{2}{3} (n^3 - 6n^2 + 11n - 6),
 $$
 completing the proof.
\end{proof}

%%%%%%%%%%%%%%%%%%%%%%%%%%%%%%%%%%%

\begin{comment}
\section{Delta-change Theorem}\label{sect: Delta change}

\begin{theorem}
    Any element $g\in RF(n)$ can be expressed
    \[g=\prod_l [x_{i_l},x_{j_l}]\prod_s [w_s,[x_{i_s},x_{j_s}]]\]
    where $g$ has weight at least 2.
\end{theorem}
\begin{itemize}
    \item We proceed by induction on the weight of $g$. Note that by the Hall Basis Theorem,
    \[g\equiv \prod_l [x_{i_l},x_{j_l}]\mod{RF(n)_3}.\]
    That is,
    \[g= \prod_l [x_{i_l},x_{j_l}]\mod{RF(n)_3}\cdot z\]
    where $z\in RF(n)_3$.
    
    Now suppose,
    \[g= \prod_l [x_{i_l},x_{j_l}]\prod_s [w_s,[x_{i_s},x_{j_s}]]\cdot z\]
    where $z\in RF(n)_{k+1}$. Then as $RF(n)_{k+1}=[RF(n)_k,RF(n)]$, we have
    \[z=\prod_i [z_i,x_i]^{\pm 1}\]
    where each $z_i\in RF(n)_k$. Hence, as $z_i$ is weight $k$,
    we may write,
    \[z=\prod_k \left[=\prod_l [r_{k,l},x_{i_l}],x_{i_k}\right].\]

    And then by Proposition 4.4.4, we have,
    \[z=\prod_k \prod_l \left[ [r_{k,l},x_{i_l}],x_{i_k}\right].\]

    Then reordering,
     \[z\equiv \prod_i \prod_j \prod_k \left[ [r_{k},x_{i}],x_{j}\right] \mod{RF(n)_{k+2}}.\]

     Then applying Prop. 4.4.4 again (twice),
     \[z\equiv \prod_i \prod_j  \left[ [R_{i,j},x_{i}],x_{j}\right] \mod{RF(n)_{k+2}}.\]

    Note that terms die if $x_i$ and $x_j$ are the same. Hence there are $2\cdot {n\choose 2}$ terms in the product above.
\end{itemize}
\end{comment}

\section{Computing the homotopy unlinking number of 4-component links.}\label{sect: 4 component}

In \cite{DOP22}, the second author, along with Orson and Park, determine the homotopy trivializing number of any 3-component string link,
\begin{proposition*}[Theorem 1.7 of \cite{DOP22}]
Let $L$ be a 3-component string link, then 
$$n_{h}(L) = \begin{cases}
\Lambda(L) &\text{if }\Lambda(L)\neq 0,\\
2&\text{ if }\Lambda(L)=0 \text{ and }\mu_{123}(L)\neq 0,\\
0&\text{otherwise}.
\end{cases}$$
\end{proposition*}

In this section we turn our attention to an analogous computation for all 4-component links.  As should not be surprising, this classification is significantly more involved than for 3-component links.  We begin by recalling some elements of the classification of 4-component links up to link homotopy.  

The classification of string links up to link homotopy is provided in \cite{HL1} and is made explicit in \cite{KiMi2023} for 4-component links (another classification appears in \cite{Graff}). We state this classification restricted to 4-component links: Any $T\in \H(4)$ can be expressed uniquely as $T=A_1A_2A_3$ where 
\begin{equation}\label{eqn:4-component classification}
\begin{array}{l}
A_1 = x_{12}^{a_{12}}x_{13}^{a_{13}}x_{14}^{a_{14}}x_{23}^{a_{23}}x_{24}^{a_{24}}x_{34}^{a_{34}},
\\
A_2 = x_{123}^{a_{123}}x_{124}^{a_{124}}x_{134}^{a_{134}}x_{234}^{a_{234}},
\\
A_3 = x_{1234}^{a_{1234}}x_{1324}^{a_{1324}},
\end{array}
\end{equation}
where $x_{ijk} = [x_{ik},x_{ij}]$, $x_{1jk4}=[[x_{14},x_{1k}],x_{1j}]$ and each $a_{I}$ above is an integer.  If $L=\widehat{T}$ is the closure of $T$, then each exponent $a_I$ recovers the corresponding Milnor invariant $\bar{\mu}_I(L)$ up to sign. In particular, $a_I=\bar{\mu}_I(L)$ when $I$ has length 2 and $a_I=-\bar{\mu}_I(L)$ when $I$ has length 3 or 4.

%We need one more piece of notation for our classification, as we will track also what crossing changes are needed to reduce a link to a homotopy trivial link.  We say that  $(T_{i_1},T_{j_1})^{a_1^+,b_i^-}, \dots (T_{i_k},T_{j_k})^{a_k^+,b_k^-}$ is a \textbf{trivializing sequence for $L$} if $L$ admits a sequence of crossing changes consisting of $a_\ell^+$ positive and $a_\ell$ negative crossing changes between $T_{i_\ell}$ and $T_{j_\ell}$.  \footnote{\chris{As I play with these, it seems like we can get much shorter statements without the trivializing sequences.}}

In the following theorems, which are summarized in the introduction as Theorem \ref{thm: 4-component sampler}, $L$ is a 4-component link and $T=A_1A_2A_3$ as in \pref{eqn:4-component classification} is a 4-component string link with $\widehat{T}=L$.

\begin{theorem}\label{thm:nhl for linking number zero}
If $L$ has vanishing pairwise linking numbers, then $n_h(L)$ is given by the following table:

%\begin{table}[!ht]
\noindent \begin{tabular}{|l|}\hline
    \cellcolor{black!05}$n_h(L)=0$ if and only if $a_I=0$ for every choice of $I$.\\ \hline
\cellcolor{black!05}{\begin{minipage}{.85\textwidth}$n_h(L)=2$ if and only if $L$ is not homotopy trivial and at least one condition (below) is met\end{minipage}} \\
 \makecell{\begin{minipage}{.98\linewidth}
 \begin{itemize}
 \item $a_{1324} \in (a_{123}, a_{124}) $ and  $a_{134}=a_{234}=0$,
 \item $a_{1234} \in (a_{123},a_{134})$ and $a_{124}=a_{234}=0$, 
 \item $a_{1234}+a_{1324} \in (a_{124},a_{134})$ and $a_{123}=a_{234}=0$, 
 \item $a_{1234}+a_{1324}\in (a_{123},a_{234})$ and $a_{124}=a_{134}=0$, 
\item $a_{1234} \in (a_{124},a_{234})$  and $a_{123}=a_{134}=0$,
\item $a_{1324} \in (a_{134},a_{234})$  and $a_{123}=a_{124}=0$. \end{itemize}\end{minipage}} \\\hline
\cellcolor{black!05}$n_h(L)=4$ if and only if $n_h(L)\not\in\{0,2\}$ and at least one condition (below) is met \\
\makecell{ \begin{minipage}{.98\linewidth}
 \begin{itemize}
 %\item $a_{ijk}=a_{ikl}=0$,
\item  $a_{jk\ell}=0$, for some $1\le j<k<\ell\le 4$
\item  $a_{1324}\in (a_{123},a_{124}, a_{134},a_{234})$,
\item  $a_{1234}\in (a_{123},a_{124}, a_{134},a_{234})$,
\item  $a_{1234}+a_{1324}\in (a_{123},a_{124}, a_{134},a_{234})$. \end{itemize}\end{minipage}} \\\hline
\cellcolor{black!05} $n_h(L)=6$ if and only if $n_h(L)\notin\{0,2,4\}$.  \\ \hline
\end{tabular}
%\caption{The homotopy trivializing number of 4-component links with vanishing linking number.}
%\label{table:nhl 4 comp lnk=0}
%\end{table}

\end{theorem}

Assume now that $L$ is a link with at least one non-vanishing pairwise linking number. It is demonstrated in \cite{DOP22} that if $L$ is a 3-component link, then $n_h(L)=\Lambda(L)$.  The same is not true for 4-component links, but the reader should note that once the pairwise linking numbers get sufficiently complicated, they determine the homotopy trivializing number. In order to cut down on the number of cases needed we will assume (at the cost of possibly permuting the components of $L$ and changing the orientation of some components) that $\lk(L_1,L_2)\ge |\lk(L_i,L_j)|$ for all $i\neq j$.  With that convention established, Theorems~\ref{thm:linking greater one} and \ref{thm:nonvanishing linking} complete our classification of the homotopy trivializing number.

\begin{theorem}\label{thm:linking greater one}
  If $\lk(L_1,L_2)>0$ and all other pairwise linking numbers vanish, then $n_h(L)$ is determined by the following table:

\noindent  \begin{tabular}{|l|}\hline
\centerline{ Suppose $\lk(T_1,T_2)=1$.}\\ \hline 
\cellcolor{black!05}{$n_h(L)=1$ if and only if $a_{134}=a_{234}=0$ and $a_{1324}=-a_{123}a_{124}$.}\\ \hline
\cellcolor{black!05}$n_h(L)=3$ if and only if $n_h(L)\not=1$ and at least one condition is met: \\
 %\makecell{
 \begin{minipage}{.98\linewidth}
 \begin{itemize}
 \item $a_{134}=0 $ or $a_{234}=0$,
 %\item $a_{234}=0$,%with one crossing change between $T_1$ and $T_2$ along with either two crossing changes between $T_1$ and $T_3$ or two crossing changes between $T_1$ and $T_4$
 %\item $a_{134}=0$.  Such a link has  trivializing sequence $(T_1,T_2)^{1^+,0^-}, (T_2,T_3)^{1^+,1^-}$ and  trivializing sequence $(T_1,T_2)^{1^+,0^-}, (T_2,T_4)^{1^+,1^-}$ %with one crossing change between $T_1$ and $T_2$ along with either two crossing changes between $T_2$ and $T_3$ or two crossing changes between $T_2$ and $T_4$
 \item $a_{1324} + a_{123}a_{124} \in (a_{134}, a_{234})$.  
 \end{itemize}\end{minipage}
 %}
 \\\hline
\cellcolor{black!05}$n_h(L)=5$ if and only if none of the above conditions are met: \\
 %\makecell{
 \begin{minipage}{.98\linewidth}
 In other words, $n_h(L)=5$ if and only if $a_{134}\cdot a_{234}\not=0$ and $a_{1324}+a_{134}\cdot a_{123} \not\in (a_{134},a_{234})$.  
 \end{minipage}
 %}
 \\\hline \hline 
 \centerline{Suppose $\lk(T_1,T_2)=2$.}\\ \hline 
 \cellcolor{black!05}$n_h(L)=2$ if and only if $a_{134}=a_{234}=0$  and at least one condition is met: \\
 %\makecell{
 \begin{minipage}{.98\linewidth}
 \begin{itemize}
 \item  at least one of $a_{123}$ or $a_{124}$ is odd,
 \item $a_{1324}$ is even. 
 \end{itemize}\end{minipage}
 %} 
 \\\hline 
  \cellcolor{black!05}$n_h(L)=4$ if and only if $n_h(L)\not=2$ and at least one condition is met: \\
 {\begin{minipage}{.98\linewidth}
 \begin{itemize}
 \item $a_{134}=0$ or  $a_{234}=0$,
 \item at least one of $a_{123},a_{134},a_{124}, a_{234}$ is odd or $a_{1324}$ is even. 
 \end{itemize}\end{minipage}} \\\hline
\cellcolor{black!05}$n_h(L)=6$ if and only if none of the above conditions are met.   \\
\hline \hline
  \centerline{Suppose $\lk(T_1,T_2)\geq 3$.}\\ \hline 
\cellcolor{black!05}
$n_h(L)\in \{\Lambda(L), \Lambda(L)+2\}$, and $n_h(L)=\Lambda(L)$ if and only if $a_{134}=a_{234}=0$. \\ \hline
\end{tabular}
\end{theorem}

The reader will notice that as $\lk(L_1,L_2)$ grows, the effect of the higher order Milnor invariants shrinks. This phenomenon persists for links with multiple nonvanishing pairwise linking numbers.    

\begin{theorem}\label{thm:nonvanishing linking}
If $L$ has at least two nonvanishing pairwise linking numbers, $\lk(L_1,L_2)>0$, and $\lk(L_1,L_2)\ge |\lk(L_i,L_j)|$ for all $i,j$, then $n_h(L)$ is given by the following table:

\noindent \begin{tabular}{|l|}\hline
\centerline{ Suppose $\lk(T_1,T_2)=1$, $\lk(T_3,T_4)=1$, and $\lk(T_i,T_j)=0$ for all other $i,j$.}\\ \hline 
\cellcolor{black!05}{$n_h(L)\in \{2,4\}$, and  $n_h(L)=2$ if and only if $a_{1324}=-a_{123}a_{124}-a_{134}a_{234}$. } \\ \hline \hline
\centerline{Suppose $\lk(T_1,T_2)=2$, $\lk(T_3,T_4)=1$, and $\lk(T_i,T_j)=0$ for all other $i,j$.}\\ \hline 
\cellcolor{black!05}{\begin{minipage}{.85\textwidth}$n_h(L)\in \{3,5\}$, and furthermore $n_h(L)=3$ if and only if at least one condition (below) is met\end{minipage}}\\
\begin{minipage}{.9\linewidth}
\begin{itemize}
    \item $a_{1324}+a_{134}a_{234}$ is even,
    \item 
    either of $a_{123}$ or $a_{124}$ is odd
\end{itemize} 
\end{minipage}\\ \hline \hline
\centerline{Suppose $\lk(T_1,T_2)=2$, $\lk(T_3,T_4)=2$, and $\lk(T_i,T_j)=0$ for all other $i,j$.}\\ \hline 
\cellcolor{black!05}{$n_h(L)\in \{\Lambda(L),\Lambda(L)+2\}$ and  $n_h(L)=\Lambda(L)$ if and only if at least one condition (below) is met}\\
\begin{minipage}{.9\linewidth}
\begin{itemize}
    \item at least one of $a_{123},a_{124},a_{134}, a_{234}$ are odd,
    \item $a_{1324}$ is even
\end{itemize} 
\end{minipage}\\ \hline \hline
\centerline{Suppose $\lk(T_1,T_2)\geq 3$, $\lk(T_3,T_4)\geq1$, and $\lk(T_i,T_j)=0$ for all other $i,j$.}\\ \hline 
\cellcolor{black!05}{$n_h(L)=\Lambda(L)$.}\\
 \hline \hline
 \centerline{ Suppose $\lk(T_1,T_2)\not=0$, $\lk(T_1,T_3)\not=0$, and $\lk(T_i,T_j)=0$ for all other $i,j$.}\\ \hline 
\cellcolor{black!05}{$n_h(L)\in \{\Lambda(L),\Lambda(L)+2\}$, and $n_h(L)=\Lambda(L)$ if and only if $a_{234}=0$. } \\ \hline \hline
\centerline{Suppose $\lk(T_1,T_2)\lk(T_2,T_3)\lk(T_3,T_4)\not=0$ or $\lk(T_1,T_2)\lk(T_1,T_3)\lk(T_2,T_3)\not=0$.}\\ \hline  
\cellcolor{black!05}{$n_h(L)= \Lambda(L)$.}\\ \hline \hline
\centerline{Suppose $\lk(T_1,T_2), \lk(T_1,T_3), \lk(T_1,T_4)$ are nonzero and  $\lk(T_i,T_j)=0$ for all other $i,j$.}\\ \hline 
\cellcolor{black!05}{$n_h(L) \in\{ \Lambda(L),\Lambda(L) +2$ \} and $n_h(L)= \Lambda(L)$  if and only if $a_{234}=0$.}\\
\hline 
\end{tabular}
\end{theorem}

The remainder of the section is organized as follows.  In Subsection~\ref{subsect: nhL=1} we express the homotopy trivializing number in terms of a word length problem in $\H(n)$.  We close this subsection by determining which elements of $\H(n)$ have $n_h(T)=1$.  In Subsection~\ref{subsect: linking equals zero} we compute the homotopy trivializing numbers when pairwise linking numbers vanish, proving Theorem~\ref{thm:nhl for linking number zero}.  In subsections~\ref{subsect: one nonvanishing linking} and \ref{subsect: many nonvanishing linking} respectively we complete the section with the proofs of Theorem \ref{thm:linking greater one} and Theorem \ref{thm:nonvanishing linking}.

\subsection{The homotopy trivializing number as a word length.}\label{subsect: nhL=1}

The braids $x_{ij}$ with $1\le i<j\le n$ generate $\H(n)$.  Thus, they also normally generate $\H(n)$.  The following proposition reveals that that $n_h(T)$ is given by counting how many conjugates of the $x_{ij}$ must be multiplied together to get $T$.  

\begin{proposition}\label{prop: algegraic translation}
Consider any $T = T_1 \cup \cdots \cup T_n \in \H(n)$.
$T$ can be undone by a sequence of crossing changes consisting of changing $p_{ij}$ positive crossings and  $n_{ij}$ negative crossings in between $T_i$ and $T_j$ for each $1\le i < j\le n$ if and only if 
$$T=\prod_{k=1}^m W_k^{-1} x_{i_k, j_k}^{\epsilon_{k}} W_k$$\
where $m=\Sum_{i,j}p_{ij}+n_{ij}$, 
each $W_k\in \H(n)$, $\epsilon_{k}\in \{\pm 1\}$, and for each $i$ and $j$ there are a total of $p_{ij}$ (or $n_{ij}$ resp.) values of $k$ with $i_k=i$, $j_k=j$ and $\epsilon_{k}=+1$ ($\epsilon_k=-1$ resp.). 
\end{proposition}
\begin{proof}
Sufficiency is obvious.  In order to see the converse, we begin with an argument in the case $m=1$.  For an example, in Figure~\ref{fig: Crossing change to conjugate A} we see a link produced by changing the trivial string link by a single positive crossing change and in Figure~\ref{fig: Crossing change to conjugate C} we see that after a link homotopy, it is a conjugate of $x_{ij}$. While some steps in the homotopy are provided, we encourage the reader to convince themselves that if they perform the clasper surgery described in \ref{fig: Crossing change to conjugate B} and in \ref{fig: Crossing change to conjugate C} then they will see two 4-component links whose first three components form the unlink, whose complement has fundamental group free of rank 3, and whose fourth components represent the same element in this free group, as this is the philosophy of the proof that follows.

If a string link $T$ can be undone by a single crossing change, then $T$ is the result of surgery on a single $C_1$-tree on the unlink.  This tree consists of a pair of disks intersecting, say, the $i$'th and $j$'th components, each in a single point along with an arc $\alpha$ between them.  

Let  $T^i$ be the sublink of $T$ obtained by deleting the component $T_i$.  After performing this surgery, $T^i$ is the unlink, and  the class of $T$ in $\H(n)$ depends only on the homotopy type of $T_i$ in the exterior of this unlink.  In the exterior of $T^i$, $T_i$ follows the arc $\alpha$, wraps once around the meridian of $T_j$ (or the reverse of this meridian), and then follows $\alpha^{-1}$.  Thus, up to link homotopy, $T$ agrees with $\beta x_{ij} \beta^{-1}$ (or $\beta x_{ij}^{-1}\beta$) where $\beta$ is the braid whose $i$'th component follows $\alpha$ as it winds about $T^i$.  In conclusion, $T$ is a conjugate of $x_{ij}$ (or $x_{ij}^{-1}$), as claimed.   
%
%Thus, a link is reduced to the unlink by a single crossing change if and only if it is a conjugate of $x_{ij}^{\pm 1}$ for some $i,j$. 

Now proceed inductively.  If $T$ can be reduced to the trivial string link by $m+1$ crossing changes, then by performing one of these crossing changes and appealing to induction, we get a new string link $S=\prod_{k=1}^m W_k^{-1} x_{i_k, j_k}^{\epsilon_{k}} W_k$.  As $T$ and $S$ differ by a single crossing change, $S^{-1}T$ can be undone by a single crossing change, so that $S^{-1}T = W_{m+1}^{-1}x_{i_{m+1}j_{m+1}}^{\epsilon_{m+1}}W_{m+1}$.  The result follows.   
\begin{figure}
     \centering
     \begin{subfigure}[b]{0.3\textwidth}
         \centering
         \begin{tikzpicture}
         \node[rotate=180] at (0,0){\includegraphics[height=.1\textheight]{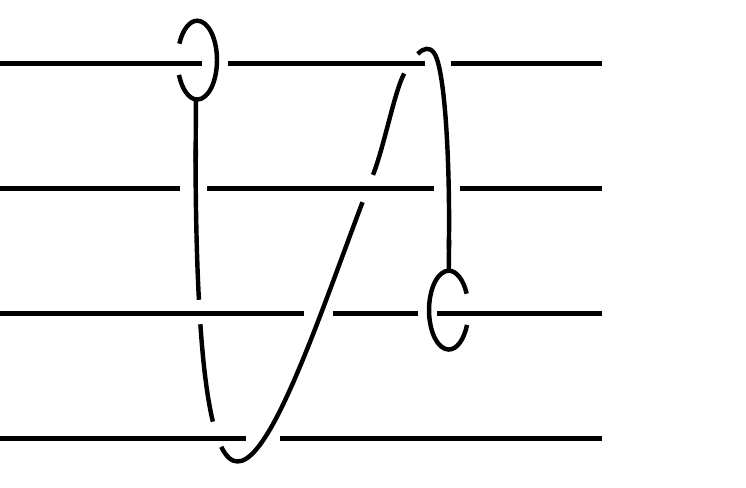}};
         \end{tikzpicture}
         \caption{}
         \label{fig: Crossing change to conjugate A}
         
     \end{subfigure}
     \begin{subfigure}[b]{0.3\textwidth}
     \centering
         \begin{tikzpicture}
         \node[rotate=180] at (0,0){\includegraphics[height=.1\textheight]{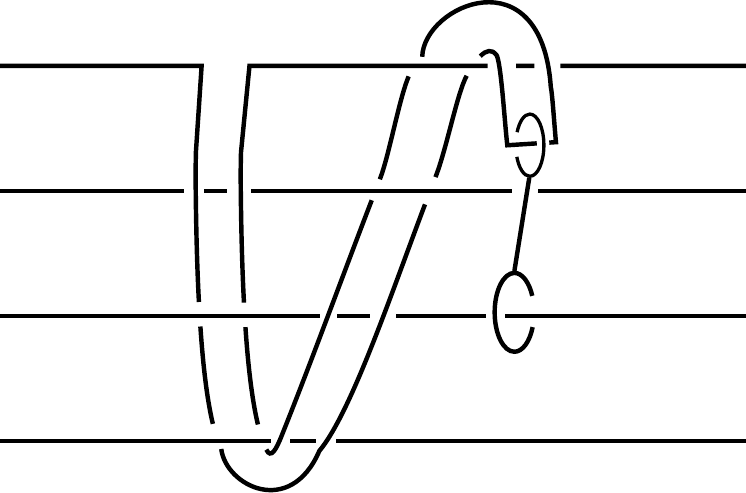}};
         \end{tikzpicture}
         \caption{}
         \label{fig: Crossing change to conjugate B}
    \end{subfigure}
     \begin{subfigure}[b]{0.3\textwidth}
     \centering
         \begin{tikzpicture}
         \node[rotate=180] at (0,0){\includegraphics[height=.1\textheight]{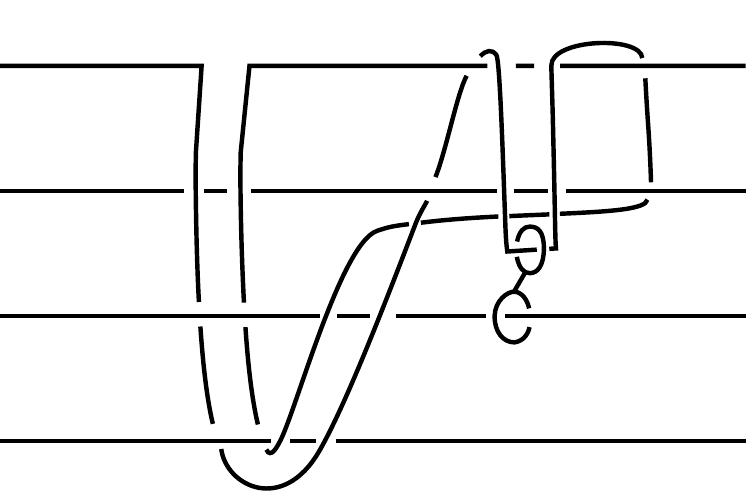}};
         \end{tikzpicture}
         \caption{}
         \label{fig: Crossing change to conjugate BC}
         
     \end{subfigure}
     \\
     \begin{subfigure}[b]{0.3\textwidth}
         \centering
         \begin{tikzpicture}
         \node[rotate=180] at (0,0){\includegraphics[height=.1\textheight]{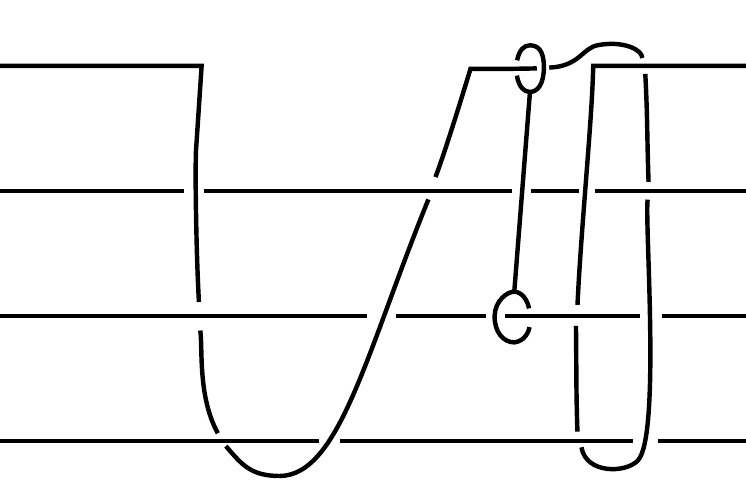}};;
         \end{tikzpicture}
         \caption{}
         \label{fig: Crossing change to conjugate BCCC}
         
     \end{subfigure}
     \begin{subfigure}[b]{0.3\textwidth}
         \centering
         \begin{tikzpicture}
         \node[rotate=180] at (0,0){\includegraphics[height=.1\textheight]{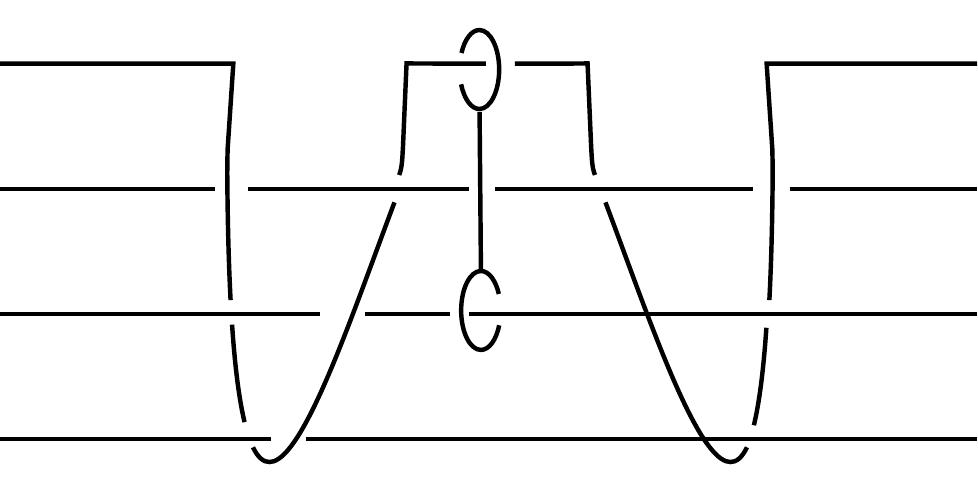}};;
         \end{tikzpicture}
         \caption{}
         \label{fig: Crossing change to conjugate C}
         
     \end{subfigure}
     
        \caption{\pref{fig: Crossing change to conjugate A} A $C_1$-tree that changes the trivial string link by a single positive crossing change.  \pref{fig: Crossing change to conjugate B} The same diagram after an isotopy.  \pref{fig: Crossing change to conjugate BC}-\pref{fig: Crossing change to conjugate C} A link homotopy reducing this diagram to a conjugate of $x_{13}$.  }
        \label{fig: Crossing change to conjugate}
\end{figure}
\end{proof}

The next lemma reveals that the terms in the product in Proposition~\ref{prop: algegraic translation}  commute at a cost of changing the conjugating elements $W_k$.  The proof amounts to expanding out both sides.
\begin{lemma}\label{lem: reorder cross change}
For any group $G$ and any $W,V,x, y\in G$, 
$$W^{-1}x W V^{-1}yV = V^{-1}y V (W')^{-1}x (W'),$$
where $W'=W V^{-1}yV$.
\end{lemma}

Thus, in order to compute the homotopy trivializing number of every 4-component string link $T$, we need only determine the minimal number of conjugates of the preferred generators $x_{ij}$ needed to multiply together to get $T$. As a first step we see exactly what string links are conjugates of these generators.

\begin{lemma}\label{cor: nhL=1}
Let $T\in \mathcal{H}(4)$. 
\begin{itemize}
\item $T$ is a conjugate of $x_{12}$ if and only if $T=x_{12}x_{123}^\alpha x_{124}^\beta x_{1234}^\gamma x_{1324}^{-\alpha\beta}$ for some $\alpha,\beta, \gamma\in \Z$.
\item $T$ is a conjugate of $x_{13}$ if and only if $T=x_{13}x_{123}^\alpha x_{134}^\beta x_{1234}^{\alpha\beta} x_{1324}^{\gamma}$ for some $\alpha,\beta, \gamma\in \Z$.
\item $T$ is a conjugate of $x_{14}$ if and only if $T=x_{14}x_{124}^\alpha x_{134}^\beta x_{1234}^{\gamma} x_{1324}^{\delta}$ for some $\alpha,\beta, \gamma,\delta\in \Z$ with $\gamma+\delta =\alpha\beta$.
\item $T$ is a conjugate of $x_{23}$ if and only if $T=x_{23}x_{123}^\alpha x_{234}^\beta x_{1234}^{\gamma} x_{1324}^{\delta}$ for some $\alpha,\beta, \gamma,\delta\in \Z$ with $\gamma+\delta =\alpha\beta$.
\item $T$ is a conjugate of $x_{24}$ if and only if $T=x_{24}x_{124}^\alpha x_{234}^\beta x_{1234}^{\alpha\beta} x_{1324}^{\gamma}$ for some $\alpha,\beta, \gamma\in \Z$.
\item $T$ is a conjugate of $x_{34}$ if and only if $T=x_{34}x_{134}^\alpha x_{234}^\beta x_{1234}^{\gamma} x_{1324}^{-\alpha\beta}$ for some $\alpha,\beta, \gamma\in \Z$.
\end{itemize}
\end{lemma}

During the proof of Lemma~\ref{cor: nhL=1} we will make use of Table~\ref{commutator table} describing the commutator of $x_{ij}$ with each of the basis elements in \pref{eqn:4-component classification}.  Each entry in this table follows from an application of Proposition~\ref{prop:rearrage}, Proposition~\ref{prop: commutators}, and the fact that $[x_{ij},x_{ik}]$ is link homotopic to $[x_{ik},x_{jk}]$. (This can be seen by using the Wirtinger presentation to express the $k$'th component of $[x_{ij},x_{ik}]$ in terms of the preferred meridians of $T_i$ and $T_j$ followed by an appeal to the homomorphism $\phi$ of \pref{exact sequence}.)  For the sake of clarity we justify the entry corresponding to $[x_{123}, x_{14}]$ as follows:
$$
\begin{array}{rcll}
[x_{123},x_{14}] &=& [[x_{13},x_{12}],x_{14}] = [[x_{14},x_{12}],x_{13}][[x_{13},x_{14}],x_{12}]
\\&=& [[x_{14},x_{12}],x_{13}][[x_{14},x_{13}],x_{12}]^{-1} = x_{1324}x_{1234}^{-1} = x_{1234}^{-1}x_{1324}.
\end{array}
$$
Note that we have used that $x_{1324}$ and $x_{1234}$ commute.  In fact, they are central in $\mathcal{H}(4)$.

\begin{table}[h!]
\begin{tabular}{|r||l|l|l|l|l|l|}
\hline
&$x_{12}$&$x_{13}$&$x_{14}$&$x_{23}$&$x_{24}$&$x_{34}$
\\\hline\hline
$x_{12}$&1&$x_{123}^{-1}$&$x_{124}^{-1}$&$x_{123}$&$x_{124}$&$1$
\\\hline
$x_{13}$&$x_{123}$&1&$x_{134}^{-1}$&$x_{123}^{-1}$&$x_{1324}$&$x_{134}$
\\\hline
$x_{14}$&$x_{124}$&$x_{134}$&$1$&$1$&$x_{124}^{-1}$&$x_{134}^{-1}$
\\\hline
$x_{23}$&$x_{123}^{-1}$&$x_{123}$&$1$&1&$x_{234}^{-1}$&$x_{234}$
\\\hline
$x_{24}$&$x_{124}^{-1}$&$x_{1324}^{-1}$&$x_{124}$&$x_{234}$&$1$&$x_{234}^{-1}$
\\\hline
$x_{34}$&$1$&$x_{134}^{-1}$&$x_{134}$&$x_{234}^{-1}$&$x_{234}$&$1$
\\\hline
$x_{123}$&1&1&$x_{1234}^{-1}\cdot x_{1324}$
&1&
$x_{1324}^{-1}$
&$x_{1234}$
\\\hline
$x_{124}$&1&$x_{1324}$&1&$x_{1234}x_{1324}^{-1}$&1&$x_{1234}^{-1}$
\\\hline
$x_{134}$&$x_{1234}$&1&1&$x_{1234}^{-1}\cdot x_{1324}$&$x_{1324}^{-1}$&1
\\\hline
$x_{234}$&$x_{1234}^{-1}$&$x_{1324}$&$x_{1234}x_{1324}^{-1}$&1&1&1
\\\hline
\end{tabular}\caption{A multiplication table for the operation $[A,x_{ij}]$. 
 $A$ takes values in the terms in the leftmost column while $x_{ij}$  takes those of the first row.}\label{commutator table}
\end{table}

We are ready to prove Lemma~\ref{cor: nhL=1}.

\begin{proof}[Proof of Lemma~\ref{cor: nhL=1}]
The proof of each of the claims amounts to an identical computation. 
 We will focus on the case that $(ij)=(12)$. Notice $T$ is a conjugate of $x_{12}$ if and only if $T=S^{-1}x_{12}S$ for some $S$.  Let $S=A_1A_2A_3$, where $A_1$, $A_2$ and $A_3$ are as in \pref{eqn:4-component classification}.   We shall show that 
$$
S^{-1}x_{12} S = x_{12}x_{123}^\alpha x_{124}^\beta x_{1234}^\gamma x_{1324}^{-\alpha\beta}
$$
for $\alpha=a_{23}-a_{13}$, $\beta=a_{24}-a_{14}$, and $\gamma = z-a_{134}+a_{234}$ where $z$ depends only on $A_1$.  The value for $z$ will be revealed in equation \pref{[X12,A1]} at the end of the proof, but it is not relevant to our analysis.  The claimed result will follow.  Proceeding, 
$$S^{-1}x_{12}S = x_{12}[x_{12},S] = x_{12}[x_{12},A_1A_2]
.$$
In the second equality above, we have used that $A_3\in \H(4)_3$ is central.  By Proposition~\ref{prop: commutators}~\pref{item Commutator product 1}, then 
\begin{equation}\label{S^-1x_12S}S^{-1}x_{12}S = x_{12}[x_{12},A_2][x_{12},A_1][[x_{12},A_1],A_2] =x_{12}[x_{12},A_2][x_{12},A_1].
\end{equation}
The second equality above relies on that $[[x_{12},A_1],A_2]\in \H(4)_4$, which is the zero subgroup.  We now compute $[x_{12},A_2]$ by using Proposition~\ref{prop: commutators}~\pref{item Commutator product 1} again, along with the fact $x_{12}$ commutes with $x_{123}$ and $x_{124}$ and that $\H(4)_2$ is Abelian,
$$\begin{array}{rcl}[x_{12},A_2] = [x_{12}, x_{134}^{a_{134}} x_{234}^{a_{234}}] = [x_{12},x_{134}]^{a_{134}} [x_{12},x_{234}]^{a_{234}}.
\end{array}$$
Finally we compute each of these commutators using Table~\ref{commutator table}.
\begin{equation}\label{[X12,A2]}\begin{array}{rcl}[x_{12},A_2] = x_{1234}^{-a_{134}+a_{234}}.
\end{array}\end{equation}

Next, we compute $[x_{12},A_1]$ via an iterated appeal to Proposition~\ref{prop: commutators}~\pref{item Commutator product 1}.
$$\begin{array}{rcl}[x_{12},A_1] = \Prod_{(pq)}[x_{12}, x_{pq}]^{a_{pq}} \Prod_{(pq)<(rs)}[[x_{12},x_{pq}], x_{rs}]^{a_{pq}a_{rs}}
\end{array}.$$
Here we use the lexicographical ordering $(12)<(13)<(14)<(23)<(24)<(34)$.  We compute this product by again referencing Table~\ref{commutator table},
\begin{equation}\label{[X12,A1]}[x_{12},A_1] = x_{123}^{a_{23}-a_{13}}x_{124}^{a_{24}-a_{14}} \cdot x_{1234}^{a_{13}a_{14}-a_{13}a_{34}-a_{14}a_{23}+a_{14}a_{34}+a_{23}a_{34}-a_{24}a_{34}}x_{1324}^{-(a_{13}-a_{23})(a_{14}-a_{24})}.
\end{equation}
If we let $z$ be the exponent of $x_{1234}$ in the preceding line then we may combine equations \pref{S^-1x_12S} \pref{[X12,A2]}, \pref{[X12,A1]} and recall our choices of $\alpha$, $\beta$, and $\gamma$ to complete the proof in the case that $(ij)=(12)$.  Identical computations complete the proof in the remaining cases.
\end{proof}

\subsection{Four component links with vanishing pairwise linking numbers: the proof of Theorem~\ref{thm:nhl for linking number zero}}\label{subsect: linking equals zero} 

Each case of Theorems~\ref{thm:nhl for linking number zero}, \ref{thm:linking greater one},  and \ref{thm:nonvanishing linking} amounts to using Lemmas~\ref{cor: nhL=1} and \ref{lem: reorder cross change} to express being undone in a sequence of crossing changes as a system of equations that the powers $a_I$ of \pref{eqn:4-component classification} must satisfy and then performing the number theory to see when these have a solution.  As a consequence, we will include less detail in the arguments as we proceed. 

\begin{proof}[Proof of Theorem~\ref{thm:nhl for linking number zero}]
Let $L$ be a 4-component link with all pairwise linking numbers zero.  Suppose $T\in \H(4)$ satisfies $\widehat{T}=L$.  Then $T$ can be written as $x_{123}^{a_{123}}x_{124}^{a_{124}}x_{134}^{a_{134}}
x_{234}^{a_{234}}
x_{1234}^{a_{1234}}x_{1324}^{a_{1324}}$ . 

 By Proposition~\ref{prop: algegraic translation}, $T\in \H(4)$ can be undone in two crossing changes with opposite signs between components $T_i$ and $T_j$ if and only if $T = V^{-1}x_{ij}^{-1}V\cdot W^{-1}x_{ij}W$ for some $V,W\in\H(4)$.  Each of these factors has its form determined by Lemma~\ref{cor: nhL=1}.  The proof amounts to expanding these products and simplifying.  In the case $(ij)=(12)$,
$$V^{-1}x_{12}^{-1}V\cdot W^{-1}x_{12}W =x_{123}^{\alpha'-\alpha} x_{124}^{\beta'-\beta} x_{1234}^{\gamma'-\gamma} x_{1324}^{\alpha\beta-\alpha'\beta'}.$$

%$$\begin{array}{rcl} V^{-1}x_{12}^{-1}V\cdot W^{-1}x_{12}W &=& (x_{12}x_{123}^\alpha x_{124}^\beta x_{1234}^\gamma x_{1324}^{-\alpha\beta})^{-1} \cdot x_{12}x_{123}^{\alpha'} x_{124}^{\beta'} x_{1234}^{\gamma'} x_{1324}^{-\alpha'\beta'}\\ &=&x_{123}^{\alpha'-\alpha} x_{124}^{\beta'-\beta} x_{1234}^{\gamma'-\gamma} x_{1324}^{\alpha\beta-\alpha'\beta'}.
 %\end{array}$$

Thus, $T$ factors as above if and only if the following system of equations has a solution:$$
\begin{array}{c}a_{134}=a_{234}=0, \alpha'=a_{123}+\alpha, \beta'=a_{124}+\beta,\text{ and }
\\a_{1324} = \alpha\beta - \alpha'\beta' = -a_{123}a_{124}-\alpha a_{124}-\beta a_{123}.
\end{array}$$
 As $-\alpha a_{124}-\beta a_{123}$ is a generic element of the ideal $(a_{123},a_{124})$, we see that this system of equations has a solution if and only if $a_{134}=a_{234}=0$ and $a_{1324}
 %+a_{123}a_{124}
 \in (a_{123},a_{124})$, as indicated in the first bullet point of the theorem under the case $n_h(L)=2$.  Similarly, the remaining bullet points determine when $T$ can be undone by any other pair of crossing changes of opposite sign between the same two components.  

For the next claim, note $n_h(L)=4$ if and only if for some $(ij)$ and $(k\ell)$ where $i,j,k,\ell\in\{1,2,3,4\}$, $T$ can be realized as 
\begin{equation}\label{eqn:Lambda=0,nhl=4}
T=V^{-1}x_{ij}V\cdot (W^{-1}x_{ij}W)^{-1}\cdot X^{-1}x_{k\ell}X\cdot (Y^{-1}x_{k\ell}Y)^{-1}.
\end{equation}
When $(ij)=(k\ell)=(12)$, Lemma~\ref{cor: nhL=1} transforms \pref{eqn:Lambda=0,nhl=4} into 
$$T = x_{123}^{\alpha - \alpha' + a - a'}x_{124}^{\beta - \beta' + b - b'}x_{1234}^{\gamma - \gamma' + c - c'}x_{1324}^{\alpha'\beta'-\alpha\beta + a'b' - ab}.$$
%\begin{align*}
    %T =& V^{-1}x_{12}V\cdot (W^{-1}x_{12}W)^{-1}\cdot X^{-1}x_{12}X\cdot (Y^{-1}x_{12}Y)^{-1} \\
        %=& x_{12}x_{123}^{\alpha}x_{124}^{\beta}x_{1234}^{\gamma}x_{1324}^{-\alpha\beta}\cdot\left(x_{12}x_{123}^{\alpha'}x_{124}^{\beta'}x_{1234}^{\gamma'}x_{1324}^{-\alpha'\beta'}\right)^{-1}\cdot x_{12}x_{123}^{a}x_{124}^{b}x_{1234}^{c}x_{1324}^{-ab}
        \\
        %&\cdot\left(x_{12}x_{123}^{a'}x_{124}^{b'}x_{1234}^{c'}x_{1324}^{-a'b'}\right)^{-1} \\
        %=&x_{123}^{\alpha - \alpha' + a - a'}x_{124}^{\beta - \beta' + b - b'}x_{1234}^{\gamma - \gamma' + c - c'}x_{1324}^{\alpha'\beta'-\alpha\beta + a'b' - ab}.
%\end{align*}
Notice any $x_{123}^{a_{123}}x_{124}^{a_{124}}x_{1234}^{a_{1234}}x_{1324}^{a_{1324}}$ can be achieved by setting
\[\alpha'=1,\;a=a_{123}+1,\;\beta=a_{124}+a_{1324},\;\beta'=a_{1324},\;\gamma=a_{1234},\;\alpha=a'=b=b'=\gamma'=c=c'=0.\]
Therefore $L$ can be undone by four crossing changes between $L_1$ and $L_2$ if and only if $a_{134}=a_{234}=0$.  An analogous result follows if $L$ can be undone by four crossing changes all between $L_i$ and $L_j$ for any $i<j$.

A similar argument holds for each pair of $(ij)$ and $(k\ell)$ where $i = k$ as well as pairs $(ij)$ and $(k \ell)$ where $i,j, k, \ell$ are all distinct, which completes the classification of links with linking number zero with $n_h(L)=4$.

The final conclusion, that any 4-component link with vanishing pairwise linking numbers can be undone in six crossing changes, is an immediate consequence of Theorem~\ref{upper bound theorem main}.

%Finally, we claim that every link can be undone in six crossing changes, in order to see this, we consider any link with vanishing pairwise linking numbers, then 
%$$
%T=\left(x_{123}^{a_{123}}\right)\left(x_{124}^{a_{124}}x_{134}^{a_{134}}x_{234}^{a_{234}}x_{1234}^{a_{1234}}x_{1324}^{a_{1324}}\right)
%$$
%Note that by earlier conclusions of the Theorem~\ref{thm:nhl for linking number zero} $n_h(x_{124}^{a_{124}}x_{134}^{a_{134}}x_{234}^{a_{234}}x_{1234}^{a_{1234}}x_{1324}^{a_{1324}})\le 4$ and $n_h(x_{123}^{a_{123}}) \le 2$.  Thus, $n_h(L)\le 6$.
\end{proof}

\subsection{Links with one nonvanishing  linking number.}\label{subsect: one nonvanishing linking}

Theorem~\ref{thm:linking greater one} classifies the homotopy trivializing number of 4-component links with precisely one non-vanishing pairwise linking number.  In order to control the number of cases, we permute components and change some orientations if needed to arrange that $\lk(L_1,L_2)>0$ and that all other pairwise linking numbers vanish.  We will further break our proof into cases depending on $\lk(L_1,L_2)$.

\begin{proof}[Proof of Theorem~\ref{thm:linking greater one} when $\lk(L_1,L_2)=1$]

Let $L$ be a link.  Assume that $\lk(L_1,L_2)=1$ and that every other linking number vanishes. 
 Let $T\in \H(4)$ satisfy $\widehat{T}=L$.  Then $T=x_{12}x_{123}^{a_{123}}x_{124}^{a_{124}}x_{134}^{a_{134}}
x_{234}^{a_{234}}
x_{1234}^{a_{1234}}x_{1324}^{a_{1324}}$. 

%Notice first that 
The only way that $n_h(L)$ could be equal to 1 is if $T$ can be undone by a single crossing change
%.  Since $\lk(L_1,L_2)=1$, this crossing change must be
between $T_1$ and $T_2$.  Thus, by lemmas~\ref{thm:nonvanishing linking} and \ref{cor: nhL=1},
$T=V^{-1}x_{12}V = x_{12}x_{123}^\alpha x_{124}^\beta x_{1234}^\gamma x_{1324}^{-\alpha\beta}$ for some $\alpha,\beta, \gamma\in \Z$.  The first result of the theorem follows.  

Similarly, $L$ can be undone in 3 crossing changes if and only if there are some $V, W, X\in \H(4)$ so that 
\begin{equation}\label{eqn three conjugates}T = V^{-1}x_{12}VW^{-1}x_{ij}W(X^{-1}x_{ij}X)^{-1}.\end{equation}  The subcases in Theorem~\ref{thm:linking greater one} for a homotopy trivializing number of 3 are now proven by evaluating this expression for the six choices of $(ij)$.

If $(ij)=(12)$, then by Lemma~\ref{cor: nhL=1}, \pref{eqn three conjugates} becomes 
$$T=x_{12}x_{123}^{\alpha+\alpha'-\alpha''} x_{124}^{\beta+\beta'-\beta''} x_{1234}^{\gamma+\gamma'-\gamma''} x_{1324}^{-\alpha\beta-\alpha'\beta'+\alpha''\beta''}.$$
We claim that $T$ can be realized as such a product if and only if $a_{134}=a_{234}=0$. The necessity of this condition is clear. For sufficiency, 
%note that if $a_{134}=a_{234}=0$, then $a_{123} = \alpha+\alpha'-\alpha''$, $a_{124} = \beta+\beta'-\beta''$, $a_{1234} = \gamma+\gamma'-\gamma''$ and $a_{1324} = -\alpha\beta-\alpha'\beta'+\alpha''\beta''$ is satisfied by setting 
take
 $$\alpha = a_{123}+1, \alpha'=0, \alpha''=1, \beta=0,\beta'=a_{124}-a_{1324}, \beta''=a_{1324}, \gamma=a_{1234}, \text{and } \gamma'=\gamma''=0.$$

 The remaining cases of $(ij)$ being $(13)$, $(14)$, $(23)$, $(24)$, or $(34)$ are all highly similar.  
 %Each results in one of $a_{134}$ or $a_{234}$ vanishing.  We address the case $(ij)=(13)$, in particular.  In this case, Lemma~\ref{cor: nhL=1} transforms \pref{eqn three conjugates} into
 %$$
% T=x_{12}x_{123}^{\alpha+\alpha'-\alpha''}x_{124}^\beta x_{134}^{\beta'-\beta''}x_{1234}^{\gamma+\alpha'\beta'-\alpha''\beta''}x_{1324}^{-\alpha\beta+\gamma'+\gamma''}.
%$$
%The necessity that $a_{234}=0$ is now automatic.  To see its sufficiency, notice that setting $a_{123}=\alpha+\alpha'-\alpha''$, $a_{124}=\beta$, $a_{134}=\beta'-\beta''$, $a_{1234}=\gamma+\alpha'\beta'-\alpha''\beta''$ and $a_{1324} = -\alpha\beta+\gamma'+\gamma''$ can now be solved for $\alpha$, $\beta$, $\beta'$, $\gamma$ and $\gamma'$.  The proof in the cases that $(ij)$ is one of $(14)$, $(23)$, or $(24)$ is identical.

%Finally in the case that $(ij)=(34)$, Lemma~\ref{cor: nhL=1} transforms \pref{eqn three conjugates} into
% $$
% T=x_{12}x_{123}^{\alpha}x_{124}^\beta x_{134}^{\alpha'-\alpha''}x_{234}^{\beta'-\beta''}x_{1234}^{\gamma+\gamma'-\gamma''}x_{1324}^{-\alpha\beta-\alpha'\beta'+\alpha''\beta''}.
%$$
%If we set $a_{123}=\alpha$, $a_{124}=\beta$, $a_{134}=\alpha'-\alpha''$, $a_{234}=\beta'-\beta''$, $a_{1234} = \gamma+\gamma'-\gamma''$, then $a_{1324}=-\alpha\beta-\alpha'\beta'+\alpha''\beta''$ becomes $a_{1324} + a_{123}a_{124} = -a_{234}a_{134}-\beta'a_{134}-\alpha'a_{234}$, which admits a solution if and only if $a_{1324} + a_{123}a_{124} \in (a_{134}, a_{234})$.
This completes the classification of 4-component links when $\lk(L_1,L_2)=1$, all other linking numbers vanishing, and $n_h(L)=3$.

It remains only to show that any link with $\lk(L_1,L_2)=1$ and all other linking numbers vanishing can be undone in at most five crossing changes.  By reordering the components, we may instead arrange that $\lk(L_1,L_3)=1$. We now appeal to  Theorem~\ref{upper bound theorem main}.  Since $Q(L)=2$ and $\Lambda(L)=1$, $n_h(L)\le 5$ as claimed. 
%
%To do so, note that by Proposition~\ref{prop: basics of nhl} we have that $x_{134}^{a_{134}} = [x_{14}, x_{13}^{a_{134}}]$ can be undone in two crossing changes.  Thus, two crossing changes reduces $T$ to 
%$$T'=x_{12}x_{123}^{a_{123}}x_{124}^{a_{124}}x_{134}^{0}
%x_{234}^{a_{234}}
%x_{1234}^{a_{1234}}x_{1324}^{a_{1324}}.$$ \cotto{do we need the $x_{134}^{0}$ term here? } This can now be undone in three crossing changes by the case of the Theorem we have already proven.  This complete the computation of homotopy trivializing numbers for all links with $\lk(L_1,L_2)=1$ and all other linking numbers vanishing.  
\end{proof}

\begin{proof}[Proof of Theorem~\ref{thm:linking greater one}  when $\lk(L_1,L_2)=2$]

Next we address the case that $\lk(L_1,L_2)=2$ and all other pairwise linking numbers vanish.  Thus, if $T$ is a string link with $\widehat{T}=L$ then 
\begin{equation}\label{eqn SL lk=2}T = x_{12}^2x_{123}^{a_{123}}x_{124}^{a_{124}}x_{134}^{a_{134}}
x_{234}^{a_{234}}
x_{1234}^{a_{1234}}x_{1324}^{a_{1324}}.\end{equation}
The only way that $L$ can be undone in exactly two crossing changes is if
$T=V^{-1}x_{12}VW^{-1}x_{12}W.$
Applying Lemma~\ref{cor: nhL=1}, this is equivalent to $T$ having the form
$$T=x_{12}^2x_{123}^{\alpha+\alpha'}x_{124}^{\beta+\beta'}x_{1234}^{\gamma+\gamma'}x_{1324}^{-\alpha\beta-\alpha'\beta'}.$$
Setting the exponents in these two expressions for $T$ equal to each other, we see $\alpha'=a_{123}-\alpha$, $\beta'=a_{124}-\beta$, $a_{134}=a_{234}=0$, $\gamma'=a_{1234}-\gamma$, and
\begin{equation}\label{linking=2 nhl=2}a_{1324}+a_{123}a_{124}=-2\alpha\beta+a_{123}\beta+a_{124}\alpha.\end{equation}
Thus, we need only see what choices of $a_{123}, a_{124},a_{1324}$ result in \pref{linking=2 nhl=2} having a solution.  Note that if $a_{123}$ and $a_{234}$ are both even and $a_{1324}$ is odd, then we get a contradiction, thus the necessity of the condition that $a_{123}$ or $a_{124}$ is odd or $a_{1324}$ is even.  To see the converse notice that \pref{linking=2 nhl=2} is equivalent to
$$2a_{1324}+a_{123}a_{124}=-(2\alpha-a_{124})(2\beta-a_{123}).
$$
If $a_{123}$ is odd then we may choose $\alpha$ and $\beta$ so that $2\beta-a_{123}=-1$ and $2\alpha-a_{124}=2a_{1234}+a_{123}a_{124}$.  We may do similarly if $a_{124}$ is odd.  If $a_{1324}$, $a_{123}$ and $a_{124}$ are all even then by dividing both sides by four, 
$$\frac{a_{1324}}{2}+\frac{a_{123}}{2}\frac{a_{124}}{2}=-\left(\alpha-\frac{a_{124}}{2}\right)\left(\beta-\frac{a_{123}}{2}\right).
$$
And we may again choose $\alpha$ and $\beta$ so that $\beta-\frac{a_{123}}{2}=-1$ and $\alpha-\frac{a_{124}}{2}=\frac{a_{1324}}{2}+\frac{a_{123}}{2}\frac{a_{124}}{2}$.  This determines which links with $\lk(L_1,L_2)=2$ and no other nonvanishing linking numbers have $n_h(L)=2$.

We now determine when a link with $\lk(L_1,L_2)=2$ and no other nonvanishing linking numbers can be undone in 4 crossing changes.  This is the case if and only if $T$ factors as 
$$T=(Vx_{12}V^{-1})( Wx_{12}W^{-1}X x_{ij}X^{-1}(Yx_{ij}Y^{-1})^{-1}).$$
Each of these factors has $\lk(T_1,T_2)=1$.  The first can be undone in a single crossing change and the second can be undone in three.  We have already classified homotopy trivializing numbers for such links.  Taking advantage of this classification, we factor $T$ as 
 $$T=(x_{12}x_{123}^{\alpha}x_{124}^{\beta}x_{1234}^{a_{1234}}x_{1324}^{-\alpha\beta})(x_{12}x_{123}^{a_{123}-\alpha}x_{124}^{a_{124}-\beta}x_{134}^{a_{134}}x_{234}^{a_{234}}x_{1324}^{a_{1324}+\alpha\beta}).$$

 The first of these terms is a conjugate of $x_{12}$.  The second can be undone in three crossing changes if and only if one of the following:
\begin{itemize}
\item $a_{134}=0$,
\item $a_{234}=0$, or
\item $
a_{1324}+\alpha\beta+(a_{123}-\alpha)(a_{124}-\beta) \in (a_{134}, a_{234})$.
\end{itemize}
Notice that the first and second of these bullet points agree with one of the conditions claimed by the theorem.  
Expanding out the third,
%We must determine when there is a choice of $\alpha,\beta\in \Z$ satisfying the third.  Expanding out, we see that this means that 
\begin{equation}\label{eqn: ln=2 nhl=4}
a_{1324}+2\alpha\beta+a_{123}a_{124}-\alpha a_{124}-\beta a_{123} = xa_{134}+ya_{234}
\end{equation}
for some $\alpha, \beta, x, y\in \Z$.  It immediately follows that if $a_{123}$, $a_{124}$, $a_{134}$, and $a_{234}$ are all even then so must $a_{1324}$ be.  Thus, it remains only to show that if $a_{ijk}$ is odd for some $(ijk)$ or $a_{1324}$ is even then \pref{eqn: ln=2 nhl=4} is satisfied for some $\alpha$ and $\beta$. 

Some factoring reduces \pref{eqn: ln=2 nhl=4} to 
\begin{equation}\label{eqn: more on lk=2, nhl=4}
a_{1324}+\frac{a_{123}a_{124}}{2}+(2\alpha-a_{123})\left(\beta-\frac{a_{124}}{2}\right) = xa_{134}+ya_{234}.
\end{equation}
If $a_{123}$ is odd and $a_{124}$ is even, then we may select $\alpha$ so that $2\alpha-a_{123}=1$ and $\beta$ so that $\beta-\frac{1}{2}a_{124} = xa_{134}+ya_{234}-a_{1324}-\frac{a_{123}a_{124}}{2}$.  A similar analysis applies if $a_{123}$ is even and $a_{124}$ is odd.  

If both of $a_{123}$ and $a_{124}$ are odd then we multiply both sides of \pref{eqn: more on lk=2, nhl=4} by 2.  If $a_{1324}$, $a_{123}$, $a_{124}$, $a_{134}$, and $a_{234}$ are all even, then we divide by 2.  From there we proceed identically to the argument for $n_h(L)=2$.
%multiplication by 2 yields
%$$
%2a_{1324}+{a_{123}a_{124}}+(2\alpha-a_{123})\left(2\beta-a_{124}\right) = 2xa_{134}+2ya_{234}.
%$$
%Again, we may select $\alpha$ and $\beta$ so that $2\alpha-a_{123}=1$ and $2\beta-a_{124}=2xa_{134}+2ya_{234}-2a_{1324}-{a_{123}a_{124}}.$ 

%If $a_{1324}$, $a_{123}$, $a_{124}$, $a_{134}$, and $a_{234}$ are all even then we divide \pref{eqn: more on lk=2, nhl=4} by two and proceed identically.
%results in 
%$$
%\frac{a_{1324}}{2}+\frac{a_{123}}{2}\frac{a_{124}}{2}+\left(\alpha-\frac{a_{123}}{2}\right)\left(\beta-\frac{a_{124}}{2}\right) = \frac{1}{2}(xa_{134}+ya_{234}).
%$$
%Again, we may select $\alpha$ so that $\left(\alpha-\frac{1}{2}a_{123}\right)=1$ and $\left(\beta-\frac{1}{2}a_{124}\right) = \frac{1}{2}(xa_{134}+ya_{234})-\frac{a_{1324}}{2}-\frac{a_{123}}{2}\frac{a_{124}}{2}.$

Finally, if either of $a_{134}$ or $a_{234}$ is odd then 2 is a unit in $\Z/(a_{134}, a_{234})$ so it has an inverse $\overline{2}$. 
To solve \pref{eqn: more on lk=2, nhl=4} it suffices to find some $\alpha, \beta\in\Z/(a_{134}, a_{234})$ satisfying
$$
-a_{1324}-\overline{2}\cdot a_{123}a_{134} \equiv 
(2\alpha-a_{123})(\beta-\overline{2}a_{124}) \mod (a_{134}, a_{234}).
$$
This is satisfied by selecting $\alpha$ and $\beta$ so that $(2\alpha-a_{123})\equiv 1$ and $(\beta-a_{124}\overline{2})\equiv -a_{1324}-a_{123}a_{134}\cdot \overline{2}$.  

That $n_h(L)\le 6= \Lambda(L)+2Q(L)$ follows from Theorem~\ref{upper bound theorem main}.% as used in the proof when $\lk(L_1,L_2)=1$.
\end{proof}

\begin{proof}[Proof of Theorem~\ref{thm:linking greater one}  when $\lk(L_1,L_2)\ge 3$]
We close by considering any link with $\lk(L_1,L_2)\ge 3$ and all other pairwise linking numbers equal to zero.  Note that $n_h(L)=\Lambda(L)$ if and only if $L$ is a a product of conjugates of positive powers of $x_{12}$.  
By Lemma~\ref{cor: nhL=1}, any such $T$ will have the form 
\begin{equation}\label{lk>=3, Lambda=lk}T=x_{12}^{a_{12}}x_{123}^{a_{123}}x_{124}^{a_{124}}x_{1234}^{a_{1234}}x_{1324}^{a_{1324}}.\end{equation}
  (Note the absence of $x_{134}$ and $x_{234}$-terms.) If $L$ has such a form, then let $a_{123}'\in\{0,1\}$ be the result of reducing $a_{123}$ mod 2.  A direct computation reveals
$$T=(x_{12}^{a_{12}-3})(x_{12}^{2}x_{123}^{a_{123}-a_{123}'-1}x_{124}^{a_{124}}x_{1234}^{a_{1234}}x_{1324}^{a_{1324}})(x_{12}x_{123}^{a_{123}'+1}).$$
%The first of these factors is undone in $a_{12}-3$ crossing changes. 
Since $a_{123}-a_{123}'-1$ is odd, previous results in the theorem show that these can be undone in $a_{12}-3$, $2$, and 1 crossing changes repectively
%we have seen that the second can be undone in two crossing changes.  Finally, the last is undone in one crossing change.  

It remains only to show that any link with $\lk(L_1,L_2)\ge 3$ can be undone in $\lk(L_1,L_2)+2$ crossing changes.  To do so use the factorization
$$T=(x_{12}^{a_{12}}x_{123}^{a_{123}}x_{124}^{a_{124}}x_{1234}^{a_{1234}}x_{1324}^{a_{1324}})(x_{134}^{a_{134}}x_{234}^{a_{234}}).$$
We have just verified that the first of these factors is undone in $a_{12}$ crossing changes.  The second is a string link with vanishing pairwise linking numbers and which is undone in two crossing changes by Theorem~\ref{thm:nhl for linking number zero}.  
\end{proof}

\subsection{Links with multiple nonvanishing linking numbers}\label{subsect: many nonvanishing linking}Theorem~\ref{thm:nonvanishing linking} classifies homotopy trivializing numbers of 4-component links with at least two nonvanishing pairwise linking numbers. Recall that we reorder and reorient the components as needed to ensure that $\lk(L_1,L_2)\ge |\lk(L_i,L_j)|$ for all $i,j$ and so that as many pairwise linking numbers as possible are positive.  Similarly to Section~\ref{thm:linking greater one}, we proceed by cases, sorted by the complexity of the pairwise linking numbers, starting with the case that $\lk(L_1,L_2)$ and $\lk(L_3, L_4)$ are the only nonvanishing linking numbers.
\begin{proof}[Proof of Theorem~\ref{thm:nonvanishing linking} when $\lk(L_1,L_2)=\lk(L_3,L_4)=1$.]
Let $L$ be a 4-component link with $\lk(L_1,L_2)=\lk(L_3,L_4)=1$ and all other pairwise linking numbers vanishing.  Let $T\in \H(4)$ satisfy $\widehat{T}=L$.  Then $T=x_{12}x_{34}x_{123}^{a_{123}}x_{124}^{a_{124}}x_{134}^{a_{134}}
x_{234}^{a_{234}}
x_{1234}^{a_{1234}}x_{1324}^{a_{1324}}$. 

In order for $L$ to be undone in precisely two crossing changes, it must be that $T$ factors as $T=(V^{-1}x_{12}V)(W^{-1}x_{34}W)$.  By Lemma~\ref{cor: nhL=1} and commutator table \ref{commutator table},
$$
\begin{array}{rcl}T&=&
%(x_{12}x_{123}^\alpha x_{124}^\beta x_{1234}^\gamma x_{1324}^{-\alpha\beta})(x_{34}x_{134}^{\alpha'}x_{234}^{\beta'} x_{1234}^{\gamma'} x_{1324}^{-\alpha'\beta'})
%\\&=&
x_{12}x_{34}x_{123}^\alpha x_{124}^\beta x_{134}^{\alpha'}x_{234}^{\beta'}x_{1234}^{\gamma+\gamma'+\alpha-\beta}x_{1324}^{-\alpha\beta-\alpha'\beta'}.
\end{array}$$
The fact that $L$ can be undone in two crossing changes if and only if $a_{1324}=-a_{123}a_{124}-a_{134}a_{234}$ follows immediately.  

In order to see that any link with $\lk(L_1,L_2)=\lk(L_3,L_4)=1$ can be undone in four crossing changes, we appeal to Theorem~\ref{upper bound theorem main}, after permuting the components, 
%to minimize $Q(L)$, to see that 
$n_h(L)\le \Lambda(L)+2Q(L)=2+2$.  
%
%In order to see that any link with $\lk(L_1,L_2)=\lk(L_3,L_4)=1$ can be undone in four crossing changes, we factor $T$ as 
%$$T=(x_{12}x_{34}x_{123}^{a_{123}}x_{124}^{a_{124}}x_{134}^{a_{134}}
%x_{234}^{a_{234}}
%x_{1234}^{a_{1234}}x_{1324}^{-a_{123}a_{124}-a_{134}a_{234}})(x_{1324}^{a_{1324}+a_{123}a_{124}+a_{134}a_{234}}).$$
%The first of these factors can be undone in two crossing changes per the part of the theorem we have already proven.  The second can as well by Theorem~\ref{thm:nhl for linking number zero}.
\end{proof}

\begin{proof}[Proof of Theorem~\ref{thm:nonvanishing linking} when $\lk(L_1,L_2) =2$, $\lk(L_3, L_4)=1$.]
Let $L$ be a 4-component link with $\lk(L_1,L_2)=2$, $\lk(L_3,L_4)=1$, and all other pairwise linking numbers vanishing.  Let $T\in \H(4)$ satisfy $\widehat{T}=L$.  Then $T=x_{12}^2x_{34}x_{123}^{a_{123}}x_{124}^{a_{124}}x_{134}^{a_{134}}
x_{234}^{a_{234}}
x_{1234}^{a_{1234}}x_{1324}^{a_{1324}}$. 

Notice that $L$ can be undone in precisely three crossing changes precisely when $T$ factors as $T=RS$ where $R$ has $\lk(R_1,R_2)=2$, $n_h(R)=2$, and $S$ is a conjugate of $x_{34}$.  Appealing to Theorems~\ref{thm:linking greater one} and \ref{cor: nhL=1}, this happens if and only if 
$$
T=(x_{12}^2x_{123}^{\alpha} x_{124}^{\beta}x_{1234}^\gamma x_{1324}^{\delta})(x_{34}x_{134}^{\alpha'}x_{234}^{\beta'}x_{1234}^{\gamma'} x_{1324}^{-\alpha'\beta'}),
$$
where either $\alpha$ or $\beta$ is odd or $\delta$ is even.  Appealing to Table \ref{commutator table}, 
$$
T=x_{12}^2 x_{34}
x_{123}^{\alpha} x_{124}^{\beta}
x_{134}^{\alpha'}x_{234}^{\beta'}
x_{1234}^{\gamma+\gamma'+\alpha-\beta} x_{1324}^{\delta-\alpha'\beta'},
$$
$T$ can be put in such a form if and only if $a_{123}=\alpha$ is odd, $a_{124}=\beta$ is odd, or $a_{1324}+a_{134}a_{234} = \delta$ is even. 

The fact that any such $L$ can be undone in five crossing changes follows from the same appeal to Theorem~\ref{upper bound theorem main} as in the previous argument. 
\end{proof}

\begin{proof}[Proof of Theorem~\ref{thm:nonvanishing linking} when $\lk(L_1,L_2) =2$ and $\lk(L_3, L_4)=2$]
Let $L$ be a 4-component link with $\lk(L_1,L_2)=2$, $\lk(L_3,L_4)=2$, and all other pairwise linking numbers vanishing.  Let $T\in \H(4)$ satisfy $\widehat{T}=L$.  Then $T=x_{12}^2x_{34}^2x_{123}^{a_{123}}x_{124}^{a_{124}}x_{134}^{a_{134}}
x_{234}^{a_{234}}
x_{1234}^{a_{1234}}x_{1324}^{a_{1324}}$. 

Notice $L$ can be undone in precisely four crossing changes precisely when $T$ factors as $T=RS$ where $R$ has $\lk(R_1,R_2)=2$, $\lk(R_3,R_4)=1$, $n_h(R)=3$, and $S$ is a conjugate of $x_{34}$.  Appealing to the case of Theorem~\ref{thm:nonvanishing linking} which we have already proven  and to \ref{cor: nhL=1}, this happens if and only if $T$ factors as 
$$
\begin{array}{rcl}T&=&(x_{12}^2x_{34}x_{123}^{\alpha} x_{124}^{\beta}x_{134}^\gamma x_{234}^\delta x_{1234}^\epsilon x_{1324}^{\zeta})(x_{34}x_{134}^{\gamma'}x_{234}^{\delta'}x_{1234}^{\epsilon'} x_{1324}^{-\gamma'\delta'})
\\&=&
x_{12}^2x_{34}^2x_{123}^{\alpha}x_{124}^{\beta}x_{134}^{\gamma+\gamma'}x_{234}^{\delta+\delta'}x_{1234}^{\epsilon+\epsilon'+\alpha-\beta}x_{1324}^{\zeta-\gamma'\delta'}
\end{array}
$$
where 
\begin{itemize}
\item $a_{123} = \alpha$ is odd or $a_{124} = \beta$ is odd, or
\item $a_{1234}+a_{134}a_{234}+2\gamma\delta-\delta a_{134}-\gamma a_{234}$ is even.
\end{itemize}
The latter bullet point is satisfied for some choice of $\gamma$ and $\delta$ in $\Z$ if and only if at least one of $a_{134}$, or $a_{234}$ is odd or $a_{1324}$ is even.  

We now close with the same appeal  Theorem~\ref{upper bound theorem main} to conclude $n_h(L)\le \Lambda(L)+2=6$.
\end{proof}

%When $\lk(L_1,L_2)\ge 3$ and $\lk(L_3, L_4)\ge 1$ we make no assumptions about the vanishing of any other pairwise linking number.  

\begin{proof}[Proof of Theorem~\ref{thm:nonvanishing linking} when  $\lk(L_1,L_2) \ge 3$ and $\lk(L_3, L_4)\ge 1$]

Let $L$ be a 4-component link with $\lk(L_1,L_2)\ge 3$ and $\lk(L_1,L_4)\ge1$.  We make no assumptions about any other linking numbers.  After changing $\Lambda(L)-4$ crossings we can replace $L$ with a new link $L'$ with $\lk(L_1',L_2')=3$, $\lk(L_1',L_4')=1$, and all other linking numbers vanishing.  Let $T\in \H(4)$ satisfy $\widehat{T}=L'$.  Then $$T=x_{12}^{3}x_{34}x_{123}^{a_{123}}x_{124}^{a_{124}}x_{134}^{a_{134}}
x_{234}^{a_{234}}
x_{1234}^{a_{1234}}x_{1324}^{a_{1324}}.$$ 
We need only factor $T$ as a $T=RS$ where $\lk(R_1,R_2)=n_h(R)=3$ and $\lk(S_3,S_4)=n_h(S)=1$.  String links satisfying these conditions are classified in Theorem~\ref{thm:linking greater one} and Lemma~\ref{cor: nhL=1} respectively.  Motivated by these we use the commutator table \ref{commutator table} to factor $T$ as 
$$T=(x_{12}^{3}x_{123}^{a_{123}}x_{124}^{a_{124}}
x_{1234}^{a_{1234}-a_{134}+a_{234}}x_{1324}^{a_{1324}+a_{134}a_{234}})
(x_{34}x_{134}^{a_{134}}x_{234}^{a_{234}}x_{1324}^{-a_{134}a_{234}}).$$
%completing the computation of the homotopy trivializing number for all 4-component links with  $\lk(L_1,L_2)$ and $\lk(L_3, L_4)$ as their only nonvanishing linking number.  
\end{proof}

This completes the analysis when $\lk(L_1,L_2)$ and $\lk(L_3, L_4)$ are the only nonvanishing pairwise linking numbers.  If $L$ has exactly two non-vanishing linking number and they both involve a shared component $L_i$, then up to reordering and reorienting, we assume that $\lk(L_1,L_2)\ge 1$, $\lk(L_1, L_3)\ge 1$ and that all other pairwise linking numbers vanish.

\begin{proof}[Proof of Theorem~\ref{thm:nonvanishing linking} when $\lk(L_1,L_2) \ge 1$ and $\lk(L_1, L_3)\ge 1$]

Let $L$ be a 4-component link with $\lk(L_1,L_2)\ge 1$, $\lk(L_1,L_3)\ge 1$, and all other pairwise linking numbers vanishing.  Let $T=x_{12}^{a_{12}}x_{13}^{a_{13}}x_{123}^{a_{123}}x_{124}^{a_{124}}x_{134}^{a_{134}}
x_{234}^{a_{234}}
x_{1234}^{a_{1234}}x_{1324}^{a_{1324}}\in \H(4)$ satisfy $\widehat{T}=L$.   

Now $\Lambda(L)=n_h(L)$ if and only if $T$ can be written as a product of conjugates of positive powers of $x_{12}$ and $x_{13}$.  By Lemma~\ref{cor: nhL=1} and Table~\ref{commutator table}, it is clear that this will imply that $a_{234}=0$.

Conversely, suppose that $a_{234}=0$. 
 We begin by making $\Lambda(L)-2$ crossing changes so that $a_{12}=a_{13}=1$. Using Table~\ref{commutator table}, it follows that $T$ factors as
$$T=(x_{12}x_{123}^{a_{123}}x_{124}^{a_{124}}x_{1234}^{a_{1234}}x_{1324}^{-a_{123}a_{124}})(x_{13}x_{134}^{a_{134}}x_{1324}^{a_{1324}+a_{123}a_{124}-a_{124}}).
$$
Lemma~\ref{cor: nhL=1} allows us to reduce this to a homotopy trivial link by two crossing changes.

Finally, to see that $T$ can be undone in $\Lambda(L)+2$ crossing changes, notice that by reordering components 
%by the permutation $(2,4)$,
we arrange that $Q(L)= 2$ and   Theorem~\ref{upper bound theorem main} concludes that $n_h(L)\le \Lambda(L)+2$.

\end{proof}

It remains only to cover the case that at least three linking numbers of $L$ are nonzero.  There are three relevant cases to consider (up to reordering).  First, all of the linking numbers involving $L_1$ may be non-zero.  Secondly, $\lk(L_1,L_2)$, $\lk(L_2, L_3)$, and $\lk(L_1, L_3)$ may be nonzero.  Finally, $\lk(L_1,L_2)$, $\lk(L_2, L_3)$, and $\lk(L_3, L_4)$ might be nonzero.  

\begin{proof}[Proof of Theorem~\ref{thm:nonvanishing linking}
when at least three linking numbers are nonzero]

Let $L$ be a 4-component link for which $\lk(L_1,L_2)>0$, $\lk(L_1, L_3)>0$, $\lk(L_1,L_4)>0$, and all other pairwise linking numbers vanish.  Let $T\in \H(4)$ satisfy $\widehat{T}=L$.  Then \[T=x_{12}^{a_{12}}x_{13}^{a_{13}}x_{14}^{a_{14}}x_{123}^{a_{123}}x_{124}^{a_{124}}x_{134}^{a_{134}}
x_{234}^{a_{234}}
x_{1234}^{a_{1234}}x_{1324}^{a_{1324}}.\] 

If $n_h(L)=\Lambda(L)$ then $T$ must be a product of positive powers of $x_{12}$, $x_{13}$ and $x_{14}$.  A glance at Lemma~\ref{cor: nhL=1} reveals that any such product will have $a_{234}=0$.  

Conversely, if $a_{234}=0$ then we note that after $a_{14}$ crossing changes we can arrange that $a_{14}=0$. Theorem~\ref{thm:nonvanishing linking} now concludes that such a link can be undone in $a_{12}+a_{13}$ crossing changes.

On the other hand, if $a_{234}\neq 0$ then since $x_{234}^{a_{234}}=[x_{23},x_{34}^{a_{234}}]$, it follows that $x_{234}^{a_{234}}$ can be undone in two crossing changes.  After making these two crossing changes,  we proceed as above for a total of $\Lambda(L)+2$ crossing changes. 

Now let $L$ be a link for which $\lk(L_1,L_2)$, $\lk(L_2,L_3)$ , and $\lk(L_3,L_4)$ are all nonzero.  
Permute the components of $L$ by the permutation $(1,3,4,2)$.  You will now see that $Q(L)=0$, so that Theorem~\ref{upper bound theorem main} completes the proof.  

Finally, let $L$ be a link for which $\lk(L_1,L_2)$, $\lk(L_1,L_3)$, $\lk(L_2,L_3)$ are all nonzero. Up to reversing orientations of some components, we may assume that with the possible exception of $\lk(L_2,L_3)$, these are all positive.  
 First we change  $\Lambda(L)-3$ crossings in order to arrange that $\lk(L_1,L_2)=\lk(L_1,L_3)=|\lk(L_2,L_3)|=1$ and that all other linking numbers vanish.  Let $T$ be a string link with $\widehat{T}=L$.  Then 
 $$
T=x_{12}x_{13}x_{23}^\epsilon x_{123}^{b_{123}}x_{124}^{b_{124}}x_{134}^{b_{134}}x_{234}^{b_{234}}x_{1234}^{b_{1234}}x_{1324}^{b_{1324}}
 $$
 with $\epsilon=\pm1$.  We use Table~\ref{commutator table} to verify the following factorization,
 $$
 \begin{array}{l}
 T=(x_{12}x_{124}^{a_{124}}x_{1234}^{u})(x_{13}x_{134}^{a_{134}}x_{1324}^{v})(x_{23}x_{123}^{a_{123}\epsilon}x_{234}^{a_{234}\epsilon}x_{1324}^{a_{123}a_{124}})^\epsilon
 \end{array}
 $$
 where $u$ and $v$ are chosen so that $a_{1234}=\epsilon a_{124}-\epsilon a_{134}+u$ and $a_{1324}=a_{124}(1-\epsilon)+a_{134}\epsilon +\epsilon a_{123}a_{124}+v$.   Each of these factors is undone by one crossing change thanks to Lemma~\ref{cor: nhL=1}.
\end{proof}

\section{Links with large homotopy trivializing number}\label{sect: links with large n_h}

We have shown that any $n$-component link $L$ with vanishing linking numbers can be reduced to a homotopy trivial link in $(n-1)(n-2)$ crossing changes.  What needs further investigation is the sharpness of this bound.  More precisely, for $n>4$ we do not know whether there exists an $n$-component link $L$ with vanishing pairwise linking numbers and $n_h(L)=(n-1)(n-2)$.  In this section, we make partial progress on this problem by exhibiting a sequence of links whose homotopy trivializing numbers grow quadratically in the number of components.  

%Our techniques will rely on some graph theory, and so we begin with the recollection of some notions.  The vertex and edge sets of a graph $\Gamma$ will be denoted by $V(\Gamma)$ and $E(\Gamma)$ respectively.  The graphs we consider will not contain any one-vertex loops, nor will then contain multiple edges between any since pair of vertices. 

Since the proof technique is different from what we have done so far in this paper, we begin by proving the following proposition.  While it is a weaker result than our main result which we will later prove (Theorem~\ref{thm: large nhl using 4-component sublinks}), its proof is easier while similar in spirit and results in links whose homotopy trivializing numbers grow quadratically in $n$.

\begin{proposition}\label{prop: lower bound warmup}Let $n\ge 3$ and $L$ be an $n$-component link with vanishing pairwise linking numbers and which satisfies that no $3$-component sublink of $L$ is homotopy trivial.  Then
\[n_h(L)\geq 2\left\lfloor\frac{(n-1)^2}{4}\right\rfloor.\]
\end{proposition}
\begin{proof}
Let $L$ be an $n$-component link with linking number zero but whose every 3-component sublink is not homotopy trivial.  Let $S$ be any sequence of crossing changes reducing $L$ to a homotopy trivial link realizing $n_h(L)$. 

We construct a graph $\Gamma$ that records the crossing changes in $S$.  Let the vertex set of $\Gamma$ be $V(\Gamma) = \{v_1,\dots, v_n\}$ and let the edge set $E(\Gamma)$ include the edge from $v_i$ to $v_j$ if $S$ includes at least one crossing change between the $i$'th and $j$'th components.  For convenience, we use the notations $v_iv_j$ and $e_{ij}$ to refer to this edge, denoting edges using their incident vertices or by their indices. Note that this graph has no multi-edges, so we do not track whether or not more than one crossing is changed between $L_i$ and $L_j$.  It also has no loops, as a self-crossing change preserves link homotopy type.  Given this graph $\Gamma$, we create its complement, $\Gamma^c$, by setting $V(\Gamma^c)= V(\Gamma)$ and letting $E(\Gamma^c)$ be the complement of $E(\Gamma)$.  

Since every 3-component sublink $L_i\cup L_j\cup L_k$ is non-trivial it follows that at least one of $e_{ij}, e_{ik}, e_{jk}$ must be in $\Gamma$, and so cannot be in $\Gamma^c$.  That is, $\Gamma^c$ contains no cycle of length 3.  A classical theorem due to Mantel from extremal graph theory (see \cite{Mantel1907} or, for a more modern reference, \cite[Theorem 1.9]{GraphsAndDigraphs}) says that any graph with $n$ vertices and more than $\lfloor n^2/4 \rfloor$ edges contains a cycle of length 3.  Thus, $\Gamma^c$ has at most $\lfloor n^2/4 \rfloor$ edges.  As a consequence, $\Gamma$ must include at least ${n\choose 2} - \lfloor n^2/4 \rfloor$ edges.  As $L$ has vanishing linking number, if $e_{ij}$ is an edge in $\Gamma$ then $S$ includes at least two crossing changes between $L_i$ and $L_j$. Thus, $n_h(L)\ge 2\left({n\choose 2} - \lfloor n^2/4 \rfloor\right)=2\left\lfloor {(n-1)^2}/{4} \right\rfloor$, where the equality follows from a direct case-wise analysis based on  the parity of $n$.  
\end{proof}

The above proof argues that since each 3-component sublink of $L$ is not homotopy trivial, the graph $\Gamma$ must contain certain edges. Our goal now is to strengthen this lower bound to a new bound whose proof instead considers 4-component sublinks.  

%\largenhlusingfourcomponentsublinks*
%\begin{restatable}{theorem}{largenhlusingfourcomponentsublinks}
\begin{theorem}
\label{thm: large nhl using 4-component sublinks}
For any $n\ge 4$ there is a link with \(n_h(L)
%\ge 4\cdot {\left\lceil\frac{n-1}{3}\right\rceil\choose 2}+ 2{\left\lfloor \frac{2n+1}{3}\right\rfloor\choose 2} 
= 2\left\lceil\frac{1}{3}n(n-2)\right\rceil.\)
\end{theorem}
%\end{restatable}

We put off the proof until the end of the section once we have built a bit more machinery.  In order to produce the needed examples, we start by proving the existence of an $n$-component  link $L$ whose every 4-component sublink, $J$ has $n_h(J) = C_4=6$.  We begin with the choice of $J$.

\begin{example}
    \label{ex: large nhl 4-component}
    Consider the string link 
    \[T=x^3_{123} x^3_{124} x^3_{134} x^3_{234}  x_{1234}x_{1324},\]
    which is depicted in Figure \ref{fig:4_comp_J}. Note the link  $T$ has vanishing pairwise linking numbers and $a_{123}=3$, $a_{124}=3$, $a_{134}=3$, $a_{234}=3$, $a_{1234}=1$, and $a_{1324}=1$. Therefore, by Theorem \ref{thm:nhl for linking number zero}, if $J=\widehat{T}$ then $n_h(J)=6.$

    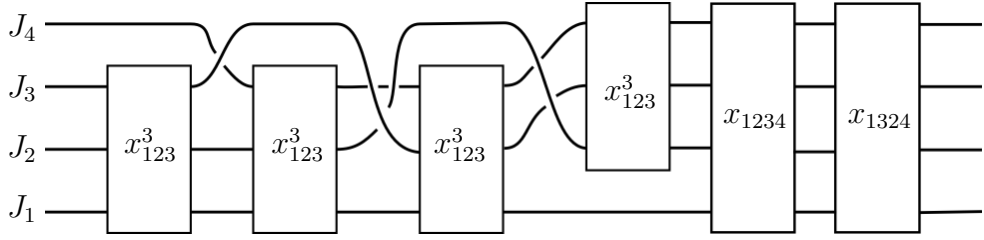
\begin{figure}
        \centering
        \begin{tikzpicture}
            \draw (0, 0) node[inner sep=0] {\includegraphics[width=0.75\linewidth]{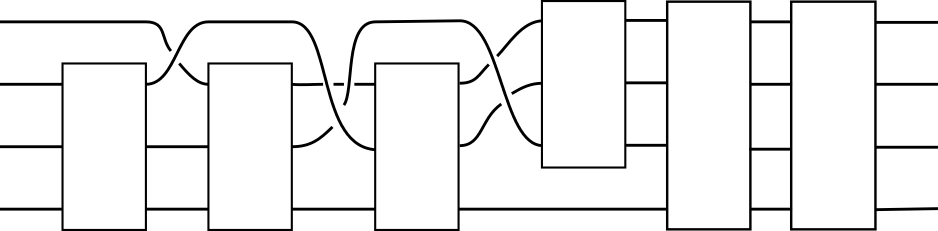}};
            \draw (-4.8, -0.35) node {$x_{123}^3$};
            \draw (-2.85, -0.35) node {$x_{123}^3$};
            \draw (-0.7, -0.35) node {$x_{123}^3$};
            \draw (1.55, 0.35) node {$x_{123}^3$};
            \draw (3.2, 0) node {$x_{1234}$};
            \draw (4.85, 0) node {$x_{1324}$};
            \draw (-6.5, -1.2) node {$J_1$};
            \draw (-6.5, -0.4) node {$J_2$};
            \draw (-6.5, 0.4) node {$J_3$};
            \draw (-6.5, 1.2) node {$J_4$};
\end{tikzpicture}
        \caption{A string link whose closure $J$ has $n_h(J) = 6$.}
        \label{fig:4_comp_J}
    \end{figure}
\end{example}

\begin{proposition}\label{prop: example with good 4-component sublinks}
For any $n\ge4$ there is an $n$-component link $L$ with pairwise linking number zero and whose every 4-component sublink $J$ has $n_h(J)=6$.

\end{proposition}
\begin{proof}
We require an $n$-component string link $T$ whose every $4$-component sublink is the string link of Example \ref{ex: large nhl 4-component} above.  To be precise, let
\[T = \Prod_{1\leq i<j<k\leq n} x^3_{ijk} \Prod_{1\leq i<j<k<l\leq n}  x_{ijkl}\Prod_{1\leq i<j<k<l\leq n} x_{ikjl}.\]
Then every 4-component sublink of $T$ is the link $J$ of Example~\ref{ex: large nhl 4-component}.  Therefore $\widehat{T}$ is the desired link.
\end{proof}

Consider now the $n$-component link $L$ of Proposition~\ref{prop: example with good 4-component sublinks}.  Let $S$ be any sequence of crossing changes reducing $L$ to a homotopy trivial link and realizing $n_h(L)$.  Now form a weighted graph $\Gamma$ with vertices $v_1, \dots, v_n$.  We assign the edge $e_{ij}$ from $v_i$ to $v_j$ a weight $\wt(e_{ij})$ equal to half of the number of crossing changes between the components $L_i$ and $L_j$ in $S$.  Recall that since $L$ has vanishing pairwise linking numbers, this number of crossing changes must be even.  Then, we define the \emph{total weight} of the graph $\Gamma$ to be $\wt(\Gamma) = \Sum_{i,j} \wt(e_{ij})$. %From this point on we will assume that every $e_{ij}$ is an edge in $\Gamma$, although the weight on that edge may be zero.  

Each 4-component sublink of $L$ has homotopy trivializing number equal to  $n_h(J)=6$.  Thus, the subgraph of $\Gamma$ spanned by any 4-component sublink has total weight at least $3$. The following extremal graph theory result, which is slightly stronger than that stated in the introduction as Theorem~\ref{thm:min_weight_phi_n}, will now imply Theorem~\ref{thm: large nhl using 4-component sublinks}.  

\begin{theorem}\label{thm: Graph theorem}
Define $\Phi_n$ to be the set of all graphs with $n$ vertices and non-negative integer weights on their edges which satisfy that for every $G\in \Phi_n$ each subgraph of $G$ spanned by at least 4 vertices has total weight at least 3.  Let $\phi_n$ denote the minimum total weight among all graphs in $\Phi_n$. For $n\geq 4$, $$\phi_n 
%= 2\cdot {\left\lceil\frac{n-1}{3}\right\rceil\choose 2}+ {\left\lfloor \frac{2n+1}{3}\right\rfloor\choose 2}
=\left\lceil\frac{1}{3}n(n-2)\right\rceil.$$ 
\end{theorem}

We now gather together what we have to prove Theorem~\ref{thm:min_weight_phi_n}.

\begin{proof}[Proof of Theorem~\ref{thm: large nhl using 4-component sublinks}]
Let $L$ be the $n$-component link of  Proposition \ref{prop: example with good 4-component sublinks}, and $S$ be be any sequence of crossing changes transforming $L$ to a homotopy trivial link.  Let $\Gamma$ be the weighted graph on vertices $v_1,\dots, v_n$ with weights given by setting $\wt(e_{ij})$ equal to half of the number of crossing changes in $S$ between $L_i$ and $L_j$.  Then $\Gamma\in \Phi_n$ and so $\wt(\Gamma) \ge \left\lceil\frac{1}{3}n(n-2)\right\rceil$.  The total weight of $\Gamma$ is equal to half the number of crossing changes in $S$.  Thus $n_h(L)\ge 2\left\lceil\frac{1}{3}n(n-2)\right\rceil$, as we claimed.
\end{proof}

Before giving an inductive proof of Theorem~\ref{thm: Graph theorem}, similar to the proof of Mantel's theorem in \cite{mantel}, we first introduce the following lemma, which is key to our inductive step.  For any vertex $v$ in a weighted graph $G$, $d(v) = \Sum_{u\in V(G)\setminus \{v\}} \wt(uv)$
%Previously $d(v) = \Sum_{u} \wt(uv)$
is the sum of the weights of the edges incident to $v$.  

\begin{lemma}\label{lem:vertex_degree_5comp}
 Let $n\ge 5$ and $G\in \Phi_n$. Then there is a vertex $v$ for which $d(v)\ge \frac{2}{3}n-\frac{4}{3}$.
\end{lemma}
    \begin{proof}[Proof of Lemma~\ref{lem:vertex_degree_5comp}]
    We open with a special case.  Suppose that there are three vertices $p,q,r$ with $\wt(pq)=\wt(pr)=\wt(qr)=0$.  It follows then that for any $s\in V(G)\setminus\{p,q,r\}$, $\wt(ps)+\wt(qs)+\wt(rs)\ge 3$, and so 
    $$d(p)+d(q)+d(r) = \sum_{s\notin \{p,q,r\}}\wt(ps)+\wt(qs)+\wt(rs)\ge 3(n-3).
    $$
    In particular, then, the average of $d(p), d(q), d(r)$ is at least $n-3$ which is at least as large as $\frac{2}{3}n-\frac{4}{3}$ as long as $n\ge 5$.  
    Thus, we may assume that no such triple $\{p,q,r\}$ of vertices connected by weight 0 edges exists.
    %Thus, we may assume that no such triple of weight 0 edges exists.  

    Suppose for the sake of contradiction that $d(v)< \frac{2}{3}n-\frac{4}{3}$ for every vertex $v$.  For any  vertex $v$, set
    $$N_v=\{x\in V(G)\mid \wt(xv)=0\}.$$
    Since $d(v)< \frac{2}{3}n-\frac{4}{3}$, it follows that $|N_v|\ge (n-1)-d(v)> \frac{1}{3}n+\frac{1}{3}$. Finally, note that if $x,y\in N_v$ and $\wt(xy)=0$, then $v,x,y$ spans a triangle whose every edge has weight 0, putting us in the situation addressed at the start of the proof.  Thus,  $\wt(xy)\ge 1$ for every $x,y\in N_v$.

    We claim that there must exist 
    %an edge on $N_v$ having a weight of 1.
    some $x,y\in N_v$ with $\wt(xy)=1$.
    %Change made above.
    Indeed, suppose $\wt(xy)\ge 2$ for every $x,y\in N_v$. For any $x\in N_v$, if we sum up only the weights of edges between $x$ and elements of $N_v$ we get $d(x)\geq\sum_{y\in N_v\setminus \{x\}} \wt(xy)\ge 2(|N_v|-1)>\frac{2}{3}n-\frac{4}{3}$, contradicting the assumption that  $d(x)<\frac{2}{3}n-\frac{4}{3}$ for every vertex $x$.  
    
    Thus, there exists some $p,q\in N_v$ such that $\wt(pq)=1$. Notice $v\in N_p\cap N_q$.   If $u\in N_p\cap N_q\setminus \{v\}$,  consider
    %change made here Consider ->consider
    the graph spanned by $p,q,u,v$ to see that $$3\le \wt(p,q)+\wt(v,u)+\wt(v,p)+\wt(u,p)+\wt(v,q)+\wt(u,q)=1+\wt(u,v)$$ and hence $\wt(u,v)\ge 2$.

    In Figure \ref{fig:graph_problem}, we summarize what we have shown above. In particular, fix some vertex $v$.  There are vertices $p,q\in N_v$ with $\wt(pq)=1$.  For any $u\in N_p\cap N_q\setminus \{v\}$,  $\wt(uv)\geq 2$.  Additionally, by the same argument we used for $N_v$, for any $w\in N_p\cup N_q$, $\wt(wv)\ge 1$. Thus, %  for any $u\in N_p\cap N_q$ and $\wt(w,v)\geq 1$ for any $w\in N_p\cup N_q$, we have that
    $$d(v)\ge |N_p\cup N_q\setminus (N_p\cap N_q)|+2|N_p\cap N_q\setminus \{v\}|=|N_p|+|N_q|-2.$$
    
    Moreover, since by the same argument we applied to $|N_v|$, we also have $|N_p|>\frac{1}{3}n+\frac{1}{3}$ and $|N_q|>\frac{1}{3}n+\frac{1}{3}$, hence we may we conclude that 
    $$d(v)\ge |N_p|+|N_q|-2> \frac{2}{3}n-\frac{4}{3}.$$
    This contradicts the assumption that $d(v)\le \frac{2}{3}n-\frac{4}{3}$ for every vertex $v$, completing the proof.  
\end{proof}

\begin{figure}
    \centering
    \begin{tikzpicture}
    \node (Npq) at (0, 2) {$N_p\cap N_q$};
    \node (Nq) at (-1.5, 0.5) {$N_q$};
    \node (Np) at (1.5, 0.5) {$N_p$};
    \node (u) at (0, 1.5) {$u$};
    \node (v) at (0, 0) {$v$};
    \draw [dashed, gray, right] (u) -- (v) node [midway] {$\geq 2$};
    \node (w1) at (-1.8, -0.5) {$w_q$};
    \node (w2) at (1.8, -0.5) {$w_p$};
    \draw [dashed, gray, below] (w1) -- (v) node [midway] {$\geq 1$};
    \draw [dashed, gray, below] (w2) -- (v) node [midway] {$\geq 1$};
    \node (p) at (-1.2, -2) {$p$};
    \node (q) at (1.2, -2) {$q$};
    \draw [dashed, gray, above] (p) -- (q) node [midway] {$1$};
    \node (Nv) at (0, -2.5) {$N_v$};
    \draw (2.6,-2.8) arc[start angle=30, end angle=150,radius=3cm];
    \draw (0.6,2.5) arc[start angle=30, end angle=-120,radius=2.5cm];
    \draw (-0.6,2.5) arc[start angle=150, end angle=330,radius=2.5cm];
\end{tikzpicture}
    \caption{An edge $(p,q)$ of weight 1 in $N_v$ and minimum weights between vertices in $N_p$ and $N_q$ with a vertex $v\in N_p\cap N_q$.}
    \label{fig:graph_problem}
\end{figure}
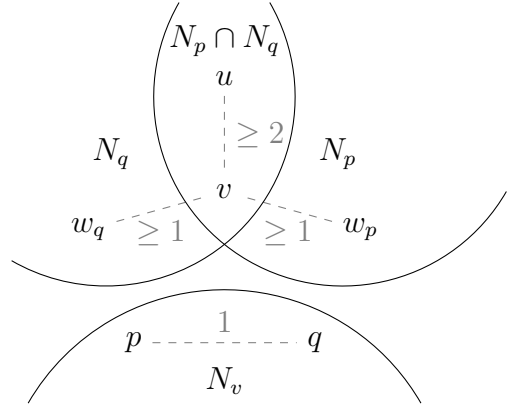
    
\begin{proof}[Proof of Theorem~\ref{thm: Graph theorem}]
%%Change made here theorem 1.7-->Theorem 7.5
We construct a graph on $n$ vertices $a_1,\dots, a_{k}, b_1,\dots, b_{\ell}$ where $k=\left\lceil\frac{n-1}{3}\right\rceil$ and $\ell = n-k=\left\lfloor \frac{2n+1}{3}\right\rfloor$.  Set $\wt(a_i,a_j)=2$, $\wt(b_i,b_j)=1$ and $\wt(a_i,b_j)=0$ for all relevant $i,j$.
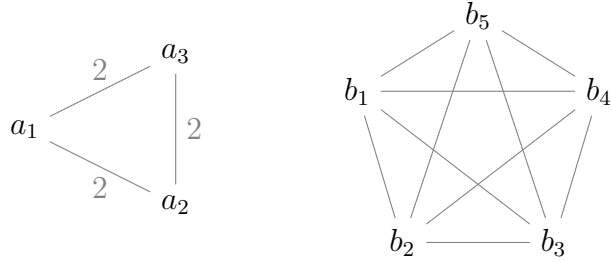
\begin{figure}
    \centering
    \begin{tikzpicture}
    \node (a1) at (-5, 0.5) {$a_1$};
    \node (a2) at (-3, -0.5) {$a_2$};
    \node (a3) at (-3, 1.5) {$a_3$};
    \draw [gray, below] (a1) -- (a2) node [midway] {$2$};
    \draw [gray, above] (a1) -- (a3) node [midway] {$2$};
    \draw [gray, right] (a2) -- (a3) node [midway] {$2$};
    
    \node (b1) at (-0.6, 1) {$b_1$};
    \node (b2) at (0, -1) {$b_2$};
    \node (b3) at (2, -1) {$b_3$};
    \node (b4) at (2.6, 1) {$b_4$};
    \node (b5) at (1, 2) {$b_5$};
    
    \draw [gray] (b1) -- (b2);
    \draw [gray] (b1) -- (b3);
    \draw [gray] (b1) -- (b4);
    \draw [gray] (b2) -- (b3);
    \draw [gray] (b2) -- (b4);
    \draw [gray] (b3) -- (b4);
    \draw [gray] (b1) -- (b5);
    \draw [gray] (b2) -- (b5);
    \draw [gray] (b3) -- (b5);
    \draw [gray] (b4) -- (b5);
\end{tikzpicture}
    \caption{A graph with $n=8$ vertices with $k=3$, $\ell=5$, and total weight $\phi_8=2{3 \choose 2}+{5\choose 2}=16$.}
%    \label{fig:enter-label}
\end{figure}

We can now compute the total weight of any 4-component subgraph:
$$
\begin{array}{ccc}
\wt(\langle a_p,a_q,a_r,a_s\rangle)=12,& 
\wt(\langle a_p,a_q,a_r,b_s\rangle)=6,&
\wt(\langle a_p,a_q,b_r,b_s\rangle)=3, \\
\wt(\langle a_p,b_q,b_r,b_s\rangle)=3,& 
\wt(\langle b_p,b_q,b_r,b_s\rangle)=6.
\end{array}
$$
Each is at least 3 so $G\in \Phi_n$.  Next note that $\wt(G) = 2\cdot {\left\lceil\frac{n-1}{3}\right\rceil\choose 2}+ {\left\lfloor \frac{2n+1}{3}\right\rfloor\choose 2}$.  
Since $\phi_n$ is the minimum total weight amongst all such graphs, $\phi_n\le 2\cdot {\left\lceil\frac{n-1}{3}\right\rceil\choose 2}+ {\left\lfloor \frac{2n+1}{3}\right\rfloor\choose 2}$.  This upper bound is equal to $\left\lceil\frac{1}{3}n(n-2)\right\rceil$ by a straightforward case-wise proof based on the class of $n$ mod $3$.  

We prove the reverse inequality by induction.  It is obvious that $\phi_4=3$, since the total weight of a 4-vertex graph is the same as the weight of its only 4-vertex subgraph.  %On the other hand
%$$
%2\cdot {\left\lceil\frac{4-1}{3}\right\rceil\choose 2}+ {\left\lfloor \frac{2\cdot 4+1}{3}\right\rfloor\choose 2}=2\cdot {1\choose 2}+{3\choose 2}=0+3.
%$$
Hence the theorem holds for $n=4$.

Now fix some $n\ge 5$.  As $\phi_n$ is defined to be a minimum, there is some graph $G$ on $n$ vertices, whose every 4-vertex subgroup has weight 3, and for which $\wt(G)=\phi_n$.  By Lemma \ref{lem:vertex_degree_5comp}, there is some vertex $v$ with $d(v)\ge \left\lceil\frac{2n-4}{3}\right\rceil$.  Set $G'$ to be the $n-1$ vertex subgraph spanned by $V(G)\setminus \{v\}$.  We may inductively assume that $\wt(G')\ge \phi_{n-1}=\left\lceil\frac{1}{3}(n-1)(n-3)\right\rceil
%2\cdot {\left\lceil\frac{n-2}{3}\right\rceil\choose 2}+ {\left\lfloor \frac{2n-1}{3}\right\rfloor\choose 2}
$.
Thus, 
$$
\phi_n=\wt(G) = \wt(G')+d(v)\ge\phi_{n-1}+d(v)\ge  \left\lceil\frac{1}{3}(n-1)(n-3)\right\rceil + \left\lceil\frac{2n-4}{3}\right\rceil.$$

That the rightmost term in the above inequality is precisely equal to $\left\lceil\frac{1}{3}n(n-2)\right\rceil$ follows from a casewise argument depending on the class of $n$ mod $3$.
\begin{comment}Indeed, suppose $n\equiv 0\mod{3}$. Then $n=3k$ for some integer $k$ and we may evaluate directly,
\[
   \left\lceil\frac{1}{3}(n-1)(n-3)\right\rceil + \left\lceil\frac{2n-4}{3}\right\rceil =\left\lceil\frac{1}{3}(3k-1)(3k-3)\right\rceil + \left\lceil\frac{6k-4}{3}\right\rceil = 3k^2-2k
\]
yet also
\[
   \left\lceil\frac{1}{3}n(n-2)\right\rceil =\left\lceil\frac{1}{3}3k(3k-2)\right\rceil = 3k^2-2k.
\]
   
The argument that $\left\lceil\frac{1}{3}(n-1)(n-3)\right\rceil + \left\lceil\frac{2n-4}{3}\right\rceil = \left\lceil\frac{1}{3}n(n-2)\right\rceil$ is the same in the cases that $n\equiv 1,2\mod{3}$. 
\end{comment}
Therefore we have shown
$\phi_n\ge \left\lceil\frac{1}{3}n(n-2)\right\rceil$
for all $n$, completing the proof. 
\end{proof}

\bibliographystyle{plain}

\bibliography{biblio}

\end{document}